\documentclass[a4paper]{amsart}

\author{Eske Ewert}
\title[\(\Psi\)DOs on filtered manifolds as generalized fixed points] {Pseudodifferential operators on filtered manifolds as generalized fixed points}
\address[Eske Ewert]{Institut für Analysis, Leibniz Universität Hannover, Welfengarten 1,
	30167 Hannover}
\usepackage[utf8]{inputenc}
\usepackage[english]{babel}
\usepackage{mathbbol}
\usepackage{amsmath}
\usepackage{amsrefs,amssymb}
\usepackage[top=0.5in,margin=1.2in,includeheadfoot,bottom=1in]{geometry} 
\usepackage{lmodern}
\usepackage{hyperref}
\usepackage{todonotes}

\usepackage[capitalize]{cleveref}
\crefdefaultlabelformat{#2{\upshape#1}#3}
\usepackage[norefs]{refcheck}

\usepackage{tikz-cd}
\usetikzlibrary{babel}

\usepackage{mathtools}

\newtheorem{theorem}{Theorem}[section]
\newtheorem{lemma}[theorem]{Lemma}
\newtheorem{proposition}[theorem]{Proposition}
\newtheorem{corollary}[theorem]{Corollary}

\theoremstyle{definition}
\newtheorem{definition}[theorem]{Definition}

\theoremstyle{remark}
\newtheorem{remark}[theorem]{Remark}
\newtheorem{example}[theorem]{Example}
\usepackage[shortlabels]{enumitem}
\setlist[enumerate,1]{label=\textup{(\roman*)}}
\setlist[enumerate,2]{label=\textup{(\alph*)}}

\newcommand*{\defeq}{\mathrel{\vcentcolon=}}

\DeclarePairedDelimiter{\abs}{\lvert}{\rvert}
\DeclarePairedDelimiter{\norm}{\lVert}{\rVert}

\DeclarePairedDelimiterX{\braket}[2]{\langle}{\rangle}{#1\,\delimsize\vert\,\mathopen{}#2}
\DeclarePairedDelimiterX{\braketop}[3]{\langle}{\rangle}{#1\,\delimsize\vert #2\delimsize\vert\,\mathopen{}#3}
\DeclarePairedDelimiterX{\BRAKET}[2]{\langle}{\rangle}{\!\delimsize\langle#1\,\delimsize\vert\,\mathopen{}#2\delimsize\rangle\!}
\DeclarePairedDelimiter{\BRA}{\langle\!\langle}{\rvert}
\DeclarePairedDelimiter{\KET}{\lvert}{\rangle\!\rangle}

\newcommand{\idealin}{\mathrel{\triangleleft}} 
\newcommand{\properideal}{%
\mathrel{\ooalign{$\lneq$\cr\raise.22ex\hbox{$\lhd$}\cr}}}

\newcommand*{\injto}{\hookrightarrow}

\newcommand*{\acts}{\curvearrowright}

\newcommand*{\Id}{\mathrm{id}}
\newcommand*{\ev}{\mathrm{ev}}
\newcommand{\rank}{\mathrm{rank}}
\DeclareMathOperator{\supp}{supp}

\newcommand*{\Mult}{\mathcal M}
\newcommand*{\Uni}{\mathcal U}

\DeclareMathOperator{\Op}{Op}

\DeclareMathOperator{\cp}{{cp}}

\DeclareMathOperator{\End}{End}

\newcommand{\W}{\mathbb W}
\newcommand{\U}{\mathbb U}
\newcommand{\V}{\mathbb V}
\newcommand{\J}{\mathbb J}

\newcommand{\F}{\mathbb F}
\newcommand{\E}{\mathbb E}
\newcommand{\f}{\mathbb f}
\newcommand{\D}{\mathbb D}

\newcommand{\C}{\mathbb{C}}
\renewcommand{\t}{\mathbb{t}}
\newcommand{\distru}{\mathbb{u}}

\newcommand{\N}{\mathbb{N}}
\newcommand{\Z}{\mathbb{Z}}
\newcommand{\R}{\mathbb{R}}
\newcommand{\Q}{\mathbb{Q}}
\newcommand{\T}{\mathbb{T}}

\newcommand{\Rp}{{\mathbb{\R}_{>0}}}

\DeclareMathOperator{\ind}{ind}
\renewcommand{\P}{\mathbb{P}}
\newcommand{\Pseu}{\mathbb{\Psi}}

\newcommand*{\Hils}[1][H]{\mathcal #1}
\newcommand{\Comp}{\mathbb K}
\newcommand{\Bound}{\mathbb B}

\newcommand*{\Cont}{\mathrm C}
\newcommand{\Schwartz}{\mathcal{S}}
\newcommand{\Smooth}{\mathcal{E}}

\newcommand*{\princ}[1][]{{s_H^{#1}}}
\newcommand*{\Princ}[1][]{{S_H^{#1}}}

\newcommand{\Four}{\mathcal{F}}
\newcommand{\grpd}{\mathcal{G}}
\newcommand{\kernel}{\mathcal{K}}
\newcommand{\Ess}{\mathrm{Ess}}
\newcommand{\A}{\mathcal{A}}

\newcommand*{\Fix}{\mathrm{Fix}}
\newcommand{\triv}{\mathrm{triv}} 

\newcommand{\Rel}{\mathcal{R}} 
\newcommand{\si}{\mathrm{si}} 

\DeclareMathOperator{\pr}{pr}
\DeclareMathOperator{\p}{p}

\newcommand*{\nb}{\nobreakdash}
\newcommand*{\Cst}{\mathrm C^*}
\newcommand*{\Cred}{\mathrm C^*_\mathrm r}

\newcommand*{\diff}{\,\mathrm{d}}

\makeatletter
\newsavebox\myboxA
\newsavebox\myboxB
\newlength\mylenA

\newcommand*\xoverline[2][0.75]{%
\sbox{\myboxA}{$\m@th#2$}%
\setbox\myboxB\null
\ht\myboxB=\ht\myboxA%
\dp\myboxB=\dp\myboxA%
\wd\myboxB=#1\wd\myboxA
\sbox\myboxB{$\m@th\overline{\copy\myboxB}$}
\setlength\mylenA{\the\wd\myboxA}
\addtolength\mylenA{-\the\wd\myboxB}%
\ifdim\wd\myboxB<\wd\myboxA%
\rlap{\hskip 0.5\mylenA\usebox\myboxB}{\usebox\myboxA}%
\else
\hskip -0.5\mylenA\rlap{\usebox\myboxA}{\hskip 0.5\mylenA\usebox\myboxB}%
\fi}
\makeatother

\newcommand*{\cl}[1]{\xoverline{#1}}
\newcommand*{\conj}[1]{\overline{#1}} 

\newcommand*{\K}{\mathrm{K}}
\newcommand*{\KK}{\mathrm{KK}}

\newcommand{\lie}[1]{\mathfrak{#1}} 

\makeatletter
\newcommand{\refcheckize}[1]{%
\expandafter\let\csname @@\string#1\endcsname#1%
\expandafter\DeclareRobustCommand\csname relax\string#1\endcsname[1]{%
\csname @@\string#1\endcsname{##1}\wrtusdrf{##1}}%
\expandafter\let\expandafter#1\csname relax\string#1\endcsname
}
\makeatother

\refcheckize{\cref}

\begin{document}
	
\begin{abstract}
	On filtered manifolds one can define a different notion of order for the differential operators. In this paper, we use generalized fixed point algebras to construct a pseudodifferential extension that reflects this behaviour. In the corresponding calculus, the principal symbol of an operator is a family of operators acting on certain nilpotent Lie groups. The role of ellipticity as a Fredholm condition is replaced by the Rockland condition on these groups. Our approach allows to understand this in terms of the representation of the corresponding algebra of principal symbols. Moreover, we compute the \(\K\)-theory of this algebra. 
\end{abstract}

\maketitle

\section{Introduction}
A filtered manifold is a smooth manifold \(M\) whose tangent bundle is filtered by subbundles \[0=H^0\subseteq H^1\subseteq H^2\subseteq \ldots\subseteq H^r=TM.\] Moreover, one requires that \([X,Y]\in \Gamma^\infty(H^{i+j})\) for all \(X\in\Gamma^\infty(H^i),Y\in\Gamma^\infty(H^j)\), where one sets \(H^{k}=TM\) for all \(k\geq r\). It is called a filtration is of step \(r\), if \(H^r=TM\) and \(H^{r-1}\neq TM\). Then one defines the order of \(X\in \Gamma^\infty(H^i)\setminus\Gamma^\infty(H^{i-1})\) to be \(i\) to obtain a new notion of order for the differential operators. For example, the contact structure of a contact manifold gives rise to a step~\(2\) filtration. In this case, one attaches order~\(2\) to the Reeb vector field, while the sections of the contact structure have order \(1\). Other examples are graded Lie groups, general Heisenberg manifolds or foliations. Also note that every smooth manifold can be understood as a filtered manifold of step \(1\).

This approach of different orders was first used to examine hypoelliptic operators like Hörmander's sum of squares operators, Kohn's Laplacian and other operators that do not fit into the classical pseudodifferential calculus, see \cites{follandstein1974, folland_subelliptic,rothschildstein}. Several  pseudodifferential calculi were defined to examine these operators, see for example \cites{christ1992pseudo,fischer2016quantization, taylornoncommutative} for graded Lie groups,  \cites{bealsgreiner} for Heisenberg manifolds or the unpublished manuscript \cites{melin} for general filtered manifolds. Recently, in noncommutative geometry a tangent groupoid approach was used to study these manifolds and their operators, see \cites{Baumvanerp,choi2017privilegedI,choi2017privilegedII, choi2019tangent, davehaller, davehaller-heat, vanerpcontactI, vanerpcontactII, erp2015groupoid, erp2017tangent, haj-higson, mohsen2018deformation, mohsen2020index, pongeheisenberg}.

A crucial difference to the usual differential calculus is that the highest order parts of operators do not necessarily commute. This is reflected by the fact that one can understand the highest order part to act as right-invariant operators on certain graded Lie groups, which are not necessarily Abelian. Namely, as the filtration is compatible with the Lie bracket, it induces a smooth family of Lie brackets on the fibres of the graded bundle
\[\lie{t}_HM\defeq \bigoplus_{i=1}^r H^i/H^{i-1}\to M.\]
The resulting Lie algebras in the fibres are nilpotent and integrate to the so called osculating Lie groups \(G_x\) for \(x\in M\). Together they form the bundle of osculating groups \(T_HM\). It is not a group bundle in the usual sense as the fibres can be non-isomorphic as groups. 

The Rockland conditions is useful to analyse the highest order part of an operator acting as right-invariant operators on the osculating groups. It was introduced by Rockland in \cite{rockland}, where he showed that a right-invariant differential operator \(D\) on a graded Lie group \(G\) is hypoelliptic if and only if \(D\) and its formal adjoint \(D^t\) satisfy the Rockland condition. Every irreducible unitary representation \(\pi\colon G\to \Uni(\Hils)\) induces an infinitesimal representation \(\diff \pi\) of the Lie algebra of \(G\) as (possibly unbounded) operators on \(\Hils\). Then one says that \(D\) satisfies the Rockland condition if  \(\diff\pi(D)\) is injective (in a suitable sense) for all representations \(\pi\) except for the trivial representation. This condition was generalized in \cite{christ1992pseudo} to right-invariant operators on \(G\) that appear as the highest order parts of pseudodifferential operators. 

Van Erp and Yuncken \cite{erp2015groupoid} defined a pseudodifferential calculus for filtered manifolds based on the tangent groupoid and a zoom action of \(\Rp\). This builds on the work of Debord and Skandalis~\cite{debordskandalis2014} who showed that the classical pseudodifferential calculus can be recovered from the zoom action on the tangent groupoid in the case where no filtration is present. The tangent groupoid of a filtered manifold was constructed in \cites{choi2019tangent,erp2017tangent,haj-higson,mohsen2018deformation} using different methods. As a set it is given by
\[\T_HM =T_HM\times\{0\}\cup M\times M\times(0,\infty).\]
It is in fact a smooth field of Lie groupoids over \([0,\infty)\), where one uses at \(t=0\) the group multiplication in the fibres and for \(t>0\) the pair groupoid structure. The zoom action of \(\Rp\) on \(\T_HM\) is given by
\begin{align*}
	\alpha_\lambda(x,y,t)&=(x,y,\lambda^{-1}t)&&\text{for }x,y\in M\text{, }t>0,\\
	\alpha_\lambda(x,v,0)&=(x,\delta_\lambda(v),0) &&\text{for }x\in M\text{, }v\in G_x.
\end{align*}
Here, \(\delta\) denotes the dilation action of \(\Rp\) on \(T_HM\) induced by \(\delta_\lambda(v)=\lambda^i v\) for \(\lambda>0\), \(v\in H^i_x/H^{i-1}_x\) and~\(x\in M\) to encode the new notion of order. 

In this article, we use a different approach to the pseudodifferential calculus using generalized fixed point algebras. This extends the construction in \cite{ewert2020pseudodifferential}, where we considered graded Lie groups. This is in nature closer to the approach of Debord and Skandalis in \cite{debordskandalis2014} where the pseudodifferential operators are obtained by averaging certain functions on the tangent groupoid over the zoom action.

Generalized fixed point algebras were introduced by Rieffel \cites{rieffel1998,rieffel1988} to define a noncommutative version of orbit spaces of proper group actions. When a locally compact group \(G\) acts on a locally compact Hausdorff space \(X\) properly, the functions on the orbit space \(\Cont_0(G\backslash X)\) can be understood as \(G\)-invariant multipliers of \(\Cont_0(X)\). For example, the principal symbols of classical pseudodifferential operators of order zero belong to \(\Cont_0(S^*M)\), where \(S^*M\) denotes the cosphere bundle. This is the generalized fixed point algebra of the \(\Rp\)-action on \(T^*M\setminus M\times\{0\}\) given by scaling in the cotangent direction. Here, one needs to take out the zero section of the cotangent bundle to obtain a proper action.

Now suppose \(G\) acts instead on a \(\Cst\)-algebra \(A\). Meyer gives in \cite{meyer2001} conditions under which one can build a generalized fixed point algebra \(\Fix^G(A)\) inside the \(G\)-invariant multiplier algebra of \(A\). Namely, one needs to find a continuously square-integrable subset \(\Rel\subset A\). Then one obtains elements of the generalized fixed point algebra by averaging elements of \(\Rel\) over the group action in an appropriate sense. Moreover, the constructions yields a Morita--Rieffel equivalence between~\(\Fix^G(A)\) and an ideal in the reduced crossed product \(\Cred(G,A)\). 

 To find such a subset \(\Rel\) for the zoom action on the groupoid \(\Cst\)-algebra of the tangent groupoid, we define an analogue of the Schwartz type algebra of Carillo Rouse \cite{rouseschwartz} adapted to the filtered setting. It consists of smooth functions on \(\T_HM\) such that \(f_t\) is compactly supported in \(M\times M\) for \(t>0\), whereas~\(f_0\) has rapid decay in the fibres of \(T_HM\). 
Furthermore, we need to take a zoom-invariant ideal \(\J\) in~\(\Cst(\T_HM)\) for the generalized fixed point algebra construction. This corresponds in the unfiltered case to taking out the zero section under the Fourier transform \(\Cst(TM)\to\Cont_0(T^*M)\) at~\(t=0\). For filtered manifolds, \(\J\) consists of all elements that restricted to \(t=0\) and \(x\in M\) lie in the kernel of the trivial representation \(\pi_\triv\colon G_x\to\C\) of the osculating group. Evaluation at \(t=0\) leads to a short exact sequence of \(\Cst\)-algebras with compatible \(\Rp\)\nb-actions
\begin{equation*}
	\begin{tikzcd}
		\Comp(L^2M)\otimes\Cont_0(\Rp) \arrow[r,hook] & \J\arrow[r,twoheadrightarrow,"\ev_0"] & \J_0.
	\end{tikzcd}
\end{equation*}
We show that one can find continuously square-integrable subsets for them such that there is a corresponding extension of generalized fixed point algebras
\begin{equation*}
	\begin{tikzcd}
		\Comp(L^2M) \arrow[r,hook] & \Fix^\Rp(\J)\arrow[r,twoheadrightarrow,"S_H"] & \Fix^\Rp(\J_0).
	\end{tikzcd}
\end{equation*}
We call it a pseudodifferential extension and \(S_H\) the principal symbol map. In fact, we show that the sequence is the \(\Cst\)-completion of the order zero extension of~\cite{erp2015groupoid}. 

Consequently, a pseudodifferential operator \(P\) of order zero on a compact filtered manifold is Fredholm if and only if its principal symbol~\(S_H(P)\) is invertible in \(\Fix^\Rp(\J_0)\). The principal symbol algebra \(\Fix^\Rp(\J_0)\) is a continuous field of \(\Cst\)-algebras over \(M\) where the fibre over \(x\in M\) is \(\Fix^\Rp(\ker(\pi_\triv\colon G_x\to \C))\). To obtain a more concrete Fredholm condition, we compute the spectrum of the latter \(\Cst\)-algebras. This yields a natural way to show that invertibility in the principal symbol algebra is equivalent to the Rockland condition on all osculating groups. Namely, for every \(x\in M\) and every non-trivial unitary irreducible representation \(\pi\colon G_x\to\Uni(\Hils)\) of \(G_x\) one can define an operator \(\pi(S_H(P)_x)\in\Bound(\Hils)\). The Rockland condition asks that \(\pi(S_H(P)_x)\in\Bound(\Hils)\) is invertible for all \(\pi\neq\pi_\triv\) and \(x\in M\). In case of a step \(1\) filtration, all osculating groups are isomorphic to \((\R^{\dim M},+)\) and we recover the well-known ellipticity condition that the principal symbol~\(p\) satisfies \(p(x,\xi)\neq 0\) for all \(\xi\neq 0\) and \(x\in M\). 

Our approach allows to show that the algebra of principal cosymbols \(\Fix^\Rp(\J_0)\) is \(\KK\)\nb-equivalent to the usual algebra of principal symbols \(\Cont_0(S^*M)\). This uses the mentioned Morita--Rieffel equivalence with an ideal in \(\Cred(\Rp,\J_0)\) coming from the generalized fixed point algebra construction. Then we use results from \cite{ewert2020pseudodifferential} to deduce that it is is fact Morita--Rieffel equivalent to the whole crossed product. Now, one can use a deformation to the Abelian case by scaling the Lie brackets to zero and the Connes--Thom isomorphism to relate it to the unfiltered case. Lastly, we show that the index problem amounts to inverting the Connes--Thom isomorphism, which was shown in the contact manifold case in \cite{Baumvanerp}, see also \cite{mohsen2020index} for  filtered manifolds. 
 
 This article is organised as follows. The definition of a filtered manifold and its tangent groupoid is recalled in \cref{sec:filtered_groupoid}. Its groupoid \(\Cst\)-algebra is introduced in \cref{sec:groupoid_c_star}. In \cref{sec:schwartztype} the Schwartz type algebra is defined. The pseudodifferential extension is built as an extension of generalized fixed point algebras in \cref{sec:gfpa}. In \cref{sec:principal_symbol} we show that the resulting algebra of principal symbols is a field over \(M\) whose fibres consist of operators of type zero on the osculating groups. We compare our construction to the calculus by van Erp and Yuncken in \cref{sec:comparison} . The results on the Morita equivalences are shown in \cref{sec:morita}. In \cref{sec:k-theory_index} we describe the \(\K\)-theory of the principal symbol algebra and show that index problem reduces to the Atiyah--Singer Index Theorem. 
 
 The results from this article are also contained in the author's PhD thesis \cite{ewert2020index}. 
 
\subsection*{Acknowledgements} The author would like to thank her PhD advisors Ralf Meyer and Ryszard Nest. Moreover, she would like to thank Thomas Schick, Elmar Schrohe and Niels Martin M\o ller for helpful remarks on her thesis. 
\section{Filtered manifolds and their groupoids}\label{sec:filtered_groupoid}
First, we recall the definition of a filtered manifold, its tangent groupoid and zoom action.
	\begin{definition}[\cite{tanaka}]
	A \emph{filtered manifold} \((M,H)\) is a smooth manifold with a filtration of its tangent bundle \(0=H^0\subseteq H^1 \subseteq H^2 \subseteq\ldots\subseteq H^r=TM\) consisting of smooth subbundles satisfying
	\begin{align}\label{eq:filtration} 
	\left[\Gamma^\infty(H^i),\Gamma^\infty(H^j)\right]\subseteq \Gamma^\infty(H^{i+j}) \quad \text{for all }i,j\geq 0.
	\end{align}
	Here, we set \(H^i=TM\) for all \(i\geq r\). A manifold is filtered of \emph{step \(r\)}, if \(H^r=TM\) and \(H^{r-1}\neq TM\). 
\end{definition} 
A filtered manifold of step \(r=1\) is just the data of an ordinary smooth manifold. A contact manifold is an example of a filtered manifold of step \(2\). \begin{example}\label{ex:filtered_graded_group}
	A graded Lie group \(G\) of step \(r\) is a Lie group whose Lie algebra admits a grading
	\[\lie{g}=\bigoplus_{i=1}^r\lie{g}_i\] 
	with \(\lie{g}_r\neq0\) and \([\lie{g}_i,\lie{g}_j]\subset\lie{g}_{i+j}\) for all \(i,j\geq 0\). Again, one sets \(\lie{g}_i=0\) for \(i>r\). In particular, the Lie algebra \(\lie{g}\) is nilpotent and, consequently, the group \(G\) as well. For more details on graded Lie groups and the analysis on these groups we refer to \cites{folland1982homogeneous,fischer2016quantization}.
	Every graded Lie group of step \(r\) can be understood as a filtered manifold of step \(r\) in the following way.  Let~\(n_i=\dim\lie{g}_i\). Define a basis \(\{X_1,\ldots,X_n\}\) of \(\lie{g}\) by choosing a basis \(\{X_{n_{i-1}+1},\ldots,X_{n_i}\}\) for all \(\lie{g}_i\). Extend these to right-invariant vector fields \(\{X_1,\ldots,X_n\}\) on \(G\) and define \(H^i\) to be the subbundle spanned by  \(\{X_1,\ldots,X_{n_i}\}\). This defines a filtration of the tangent bundle of \(G\).
\end{example} For a broader overview on different types of filtered manifolds appearing in various areas of mathematics see \cite{choi2017privilegedI}*{Section 2.3}. 

\subsection{The osculating groupoid}
The filtration of the tangent bundle of a filtered manifold \((M,H)\) of step \(r\) allows to define the graded vector bundle
\[ \bigoplus_{i=1}^r H^i/H^{i-1}\to M.\]
This bundle can be equipped with the structure of a Lie algebroid over \(M\), which we will denote by \(\lie{t}_HM\). For the general theory of Lie algebroids and Lie groupoids see, for example, \cite{mackenzie}. The bracket 
\[[\,\cdot\,,\cdot\,]\colon \Gamma^\infty(\lie{t}_HM)\times\Gamma^\infty(\lie{t}_HM) \to \Gamma^\infty(\lie{t}_HM)\]
is induced by the Lie bracket of vector fields on \(TM\). Let \(X\in\Gamma^\infty(H^i)\) and \(Y\in\Gamma^\infty(H^j)\) be representatives of \(\langle X\rangle\in\Gamma^\infty(H^i/H^{i-1})\) and \(\langle Y\rangle \in\Gamma^\infty(H^j/H^{j-1})\) and set
\[ [\langle X\rangle ,\langle Y\rangle]=\langle[X,Y]\rangle .\]
The condition~\eqref{eq:filtration} ensures that this is well-defined. The anchor \(\lie{t}_HM\to TM\) is given by the zero map. 
Therefore, the bracket restricts to each fibre \((\lie{t}_HM)_x\) for \(x\in M\) and turns \((\lie{t}_HM)_x\) into a graded Lie algebra. It is nilpotent by condition \eqref{eq:filtration} and the fact that \(H^i/H^{i-1}=0\) for \(i>r\).

The Lie algebroid \(\lie{t}_HM\) integrates to a Lie groupoid \(T_HM\) over \(M\) (see \cite{erp2017tangent}*{Sec.~3, Sec.~8}). As a manifold, it is the graded vector bundle \(\bigoplus_{i=1}^r H^i/H^{i-1}\). Its source and range map are the base projection.
\begin{definition}
	For \(x\in M\) denote by \(G_x\) the simply connected Lie group of~\((\lie{t}_HM)_x\) and call it the \emph{osculating group} at \(x\in M\). 
\end{definition}

Every osculating group is a graded Lie group of step \(r\). The group multiplication of \(G_x\) is uniquely determined in terms of the Lie bracket by the Baker--Campbell--Hausdorff formula (see, for example, \cite{corwin1990representations}*{1.2.1}). The groupoid multiplication in \(T_HM\) is given by the group product in the fibres. As the brackets vary smoothly along \(M\), this defines a Lie groupoid multiplication. The Lie groupoid \(T_HM\) is called the \emph{osculating groupoid} or the \emph{bundle of osculating groups} in \cite{erp2017tangent}. In \cite{choi2019tangent} it is called the \emph{tangent group bundle}. However, \(T_HM\) is in general not a group bundle in the sense of principal bundles as the group structure might vary from point to point.
\begin{example}
	Let \(M=\R^3\) and define three vector fields
	\[X=\partial_x+y^2\partial_z\text{,}\quad Y=\partial_y\quad\text{and}\quad Z=\partial_z.\]
	Let \(H\subset T\R^3\) be the subbundle spanned by \(X,Y\). One computes \([X,Y]=-2yZ\). Therefore, the osculating groups are Abelian whenever \(y=0\), while they are isomorphic to the Heisenberg group for \(y\neq 0\).
\end{example}
\begin{example}
	In the step \(r=1\) case and \(x\in M\), the osculating group \(G_x\) is the tangent space \(T_xM\) with group operation being the addition of tangent vectors. Hence, all osculating groups are isomorphic to the Abelian group \(\R^{\dim(M)}\). 
\end{example}

\begin{example}For a contact manifold of dimension \(2k+1\) all osculating groups are isomorphic to the \((2k+1)\)-dimensional Heisenberg group. 
\end{example}
\begin{example}
	When \(G\) is a graded Lie group, understood as a filtered manifold as in \cref{ex:filtered_graded_group}, all osculating groups are isomorphic to \(G\).
\end{example}	

\begin{definition}\label{def:dilation_osculating}
	The \emph{dilation action} of \(\Rp\) on \(\lie{t}_HM\) is defined by \(D_\lambda( v )\defeq\lambda^i v \) for \( v \in H^i_x/H^{i-1}_x\). It integrates to an action on \(T_HM\), which we denote by \(\delta_\lambda(\xi)=\lambda\cdot \xi\) for \(\xi\in G_x\). 
\end{definition}

From now on, we will always assume that the bundles \(H^i\) of the filtration of a filtered manifold~\((M,H)\) have constant rank, which is automatic if \(M\) is connected. 
\begin{definition}
	The \emph{homogeneous dimension} of a filtered manifold \((M,H )\) is 
	\[d_H \defeq \sum_{i=1}^r {i\left(\rank\left(H^i\right)-\rank\left(H^{i-1}\right)\right)}.\]
	In the following we denote by \(n\) the dimension of \(M\) as a manifold. Then the \emph{weight sequence} of \((M,H)\) is defined as
	\[(q_1,\ldots,q_n)=(1,\ldots,1,2,\ldots,2,\ldots,r,\ldots,r),\]
	where each \(1\leq i \leq r\) occurs \((\rank(H^i)-\rank(H^{i-1}))\)-times. 
\end{definition}
\begin{remark}
The homogeneous dimension of a filtered manifold is the homogeneous dimension of all osculating groups as defined in \cite{folland1982homogeneous}. 
\end{remark}
Assigning to a filtered manifold \((M,H)\) its osculating groupoid \(T_HM\) is functorial, when considering the following morphisms.
\begin{definition}
	A \emph{filtered manifold map} or \emph{Carnot map} \(f\colon(M_1,H_1)\to (M_2,H_2)\) is a smooth map between filtered manifolds \((M_1,H_1)\) and \((M_2,H_2)\) such that
	\begin{align}\label{eq:Carnot_map}
	\diff f(H_1^i)\subset H_2^i \quad\text{for all }i\in\N.
	\end{align}
	A \emph{Carnot diffeomorphism} is a diffeomorphism \(f\colon M_1\to M_2\) such that \(f\) and \(f^{-1}\) are filtered manifold maps. 
\end{definition}	
Let \(f\colon (M_1,H_1)\to (M_2,H_2)\) be a Carnot map. Condition \eqref{eq:Carnot_map} ensures that there is a well-defined induced vector bundle morphism \(\lie{t}f\colon \lie{t}_{H_1}M_1\to \lie{t}_{H_2}M_2\), which satisfies
\[[\lie{t}f(X),\lie{t}f(Y)]=\lie{t}f([X,Y])\quad \text{for }X,Y\in \Gamma^\infty(\lie{t}_{H_1}M_1).\]
Consequently, using the exponential maps, one obtains a Lie groupoid homomorphism between the osculating groupoids \(Tf\colon T_{H_1}M_1\to T_{H_2}M_2\). It restricts in each fibre to a homomorphism of graded Lie groups \(T_xf\colon G_x\to G_{f(x)}\).	The map \(Tf\) is equivariant for the dilation actions, that is,
\[T_xf(\lambda \cdot\xi)=\lambda\cdot T_xf(\xi) \quad\text{for }x\in M\text{, }\lambda>0\text{ and }\xi\in G_x.\]
We will consistently use the notation \(\diff f\colon TM_1\to TM_2\) for the usual differential, whereas the homomorphism between the osculating groupoids is denoted \(Tf\colon T_{H_1}M_1\to T_{H_2}M_2\). The latter is called the \emph{Carnot differential} in \cite{choi2019tangent}.
\begin{proposition}[\cite{choi2019tangent}*{5.5}]\label{res:osculating_functor}
	The assignment \((M,H)\mapsto T_HM\) and \(f\mapsto Tf\) defines a functor from the category of filtered manifolds with filtered manifold maps to the category of Lie groupoids with Lie groupoid homomorphisms. 
\end{proposition}	
\subsection{The tangent groupoid}
	Besides the osculating groupoid \(T_HM\), the pair groupoid is another important groupoid attached to a filtered manifold \((M,H)\).  
\begin{example}
	For a set \(M\), the \emph{pair groupoid of \(M\)} is the groupoid with arrow space \(M\times M\) and unit space \(M\). The range and source \(r,s\colon M\times  M\to M\) and unit \(u\colon M\to M\times M\) are given by
	\[r(x,y)=x, \quad s(x,y)=y,\quad u(x)=(x,x).\]
	The inverse and multiplication are defined by
	\[(x,y)^{-1}=(y,x)\quad\text{ and }\quad  (x,y)\cdot(y,z)=(x,z).\]
	If \(M\) is a smooth manifold, the pair groupoid of \(M\) is a Lie groupoid.
\end{example}
The two groupoids \(T_HM\) and \(M\times M\) can be glued together in a smooth way, yielding the \emph{tangent groupoid} of \(M\). We first discuss its groupoid structure. 
\begin{definition}
	The \emph{tangent groupoid} \(\T_HM\) of a filtered manifold~\((M,H)\) consists of the arrow space
	\begin{equation*}
	\T_HM = (T_HM \times \{0\}) \cup (M\times M\times(0,\infty))
	\end{equation*}
	and the unit space \(M\times[0,\infty)\). The range and source maps \(r,s\colon \T_HM\to M\times[0,\infty)\) are given by 
	\begin{align*}
	r(x,\xi,0)&=(x,0), &s(x,\xi,0)&=(x,0) &&\text{for }\xi\in G_x,\\
	r(x,y,t)&=(x,t), &s(x,y,t)&=(y,t)&&\text{for }x,y\in M\text{ and }t>0.
	\end{align*}
	The unit \(u\colon M\times[0,\infty)\to \T_HM\) and the inverse \(i\colon \T_HM\to \T_HM\) are defined by 
	\begin{align*}
	u(x,0)&=(x,0,0) &&\text{for }x\in M, &u(x,t)=(x,x,t) &&\text{for }x\in M\text{ and }t>0,\\
	i(x,\xi,0)&=(x,\xi^{-1},0) &&\text{for }\xi\in G_x, &i(x,y,t)=(y,x,t) &&\text{for }x,y\in M\text{ and }t>0.
	\end{align*}
	The multiplication \(m\colon\T_HM^{(2)}\to\T_HM\) is given by
	\begin{align*}
	(x,\xi,0)(x,\eta,0)&=(x,\xi\cdot \eta,0) &&\text{for }\xi,\eta\in G_x,\\
	(x,y,t)(y,z,t)&=(x,z,t) &&\text{for }x,y,z\in M\text{ and }t>0.
	\end{align*}
	At \(t=0\), the multiplication and inversion in the osculating groups are used. 
\end{definition}		
The range fibres are given by 
\begin{equation*}
\T_HM^{(x,t)}= 
\begin{cases}
\{(x,\xi,0)\mid\xi\in G_x\} &\text{for }x\in M\text{ and }t=0,\\
\{(x,y,t)\mid y\in M\} &\text{for }x\in M\text{ and }t>0.
\end{cases}
\end{equation*}	
\begin{example}\label{ex:tangentofgradedliegroup}
	Let \(M=G\) be a graded Lie group with the filtration~\(H\) as in \cref{ex:filtered_graded_group}. The tangent groupoid \(\T_HG\) is isomorphic to the transformation groupoid \[\grpd=(G\times[0,\infty))\rtimes G\]
	of the action \((G\times[0,\infty))\curvearrowleft G\) given by \((x,t).v=(x\delta_t(v),t)\). Here, we set \(\delta_0(v)=\lim_{t\to 0}\delta_t(v)=0\) for all \(v\in G\). 
	
	The isomorphism \(\phi\colon \T_HG\to\grpd\) is given by \(\phi(x,y,t)=(x,t,\delta_{t^{-1}}(x^{-1}y))\) for \(t>0\) and \(\phi(x,\xi,0)=(x,0,\xi)\) when identifying \(G_x\) with \(G\). The inverse is given by \((x,t,v)\mapsto (x,x\delta_t(v),t)\) for \(t>0\) and \((x,0,v)\mapsto(x,v,0)\). 
\end{example}
A crucial feature of the tangent groupoid of a filtered manifold is that it defines a Lie groupoid:
\begin{theorem}[\cites{erp2015groupoid,choi2019tangent}]
	The tangent groupoid \(\T_HM\) of a filtered manifold \((M,H)\) admits a smooth structure such that it becomes a Lie groupoid. 
\end{theorem}
We recall the construction of coordinate charts in \cite{choi2019tangent}. 
\begin{definition}[\cite{choi2019tangent}*{2.14, 9.2}]
	An \emph{\(H\)-frame} \(X=(X_1,\ldots,X_n)\) over an open subset \(V\subset M\) consists of vector fields \(X_i\colon V\to TM\), such that \(\left(X_1,\ldots,X_{\rank H^i}\right)\) defines a frame for \(H^i|_V\) for all \(i=1,\ldots,r\). 	
	An \emph{\(H\)-chart} is a local chart \(\kappa\colon V\to U\) between open subsets \(V\subset M\) and \(U\subset\R^n\) together with an \(H\)-frame \(X\) over \(V\). 
\end{definition}
For an \(H\)-chart \(\kappa\colon M\supseteq V\to U\subseteq \R^n\), Choi and Ponge construct in \cite{choi2019tangent}*{(9.5)} a chart \(\phi_\kappa\colon \T_HM\supseteq \V\to\U\), where
\begin{align*}
\V &= (T_HM|_V \times \{0\})\cup (V \times V \times (0,\infty)),\\
\U &= \left\{(x,v,t)\in U\times\R^n\times[0,\infty)\mid \left(\varepsilon^\kappa_x\right)^{-1}(t\cdot v)\in U\right\}.
\end{align*}
Here, \(\varepsilon^\kappa\) is the \(\varepsilon\)-Carnot map \(U\times\R^n\to\R^n\) associated with the \(H\)-chart \(\kappa\) as described in~\cite{choi2019tangent}*{4.17}. The map \(\phi_\kappa\) is given by
\begin{align*}
\phi_\kappa(x,\xi,0)&= (\kappa(x),T_x\kappa(\xi),0) &&\text{for }\xi\in G_x,\\
\phi_\kappa(x,y,t)&= \left(\kappa(x),t^{-1}\cdot \varepsilon^\kappa_{\kappa(x)}(\kappa(y)),t\right) &&\text{for }x,y\in V\text{ and }t>0. 
\end{align*}
Its inverse is 
\begin{equation*}
\phi_\kappa^{-1}(x,v,t)=
\begin{cases}
(\kappa^{-1}(x),(T_x\kappa)^{-1}(v), 0) &\text{for }(x,v)\in U\times\R^n \text{ and }t=0,\\
\left(\kappa^{-1}(x), \left(\varepsilon^\kappa_x\circ\kappa\right)^{-1}(t\cdot v),t\right) &\text{for }(x,v,t)\in\U \text{ and }t>0.
\end{cases}
\end{equation*}
The smooth structure of \(\T_HM\) is uniquely determined by the charts \(\phi_\kappa\) for all \(H\)\nb-charts \(\kappa\) of \(M\) and by requiring that the inclusion \(M\times M\times(0,\infty)\injto \T_HM\) is a smooth embedding (see \cite{choi2019tangent}*{9.7}).		
To shorten notation, we will sometimes denote \(V_\infty=M\) and \(\V_\infty = M\times M\times (0,\infty)\) in the following.
Note that for each \(H\)-chart \(\kappa\colon V\to U\), the open subset \(\V\) is a subgroupoid of~\(\T_HM\).
\begin{example}
	Let \(G\) be a graded Lie group and let \(\kappa\colon G\to \R^n\) be the global coordinate chart obtained from the exponential map. Then the \(\varepsilon\)-Carnot map is \(\varepsilon^\kappa_x(y)=x^{-1}\cdot y\), see \cite{choi2017privilegedII}*{9.12}. Therefore, \(\phi_\kappa\) is the isomorphism from \cref{ex:tangentofgradedliegroup}.
\end{example}
Often, it will be useful to understand \(\T_HM\) as a smooth field of groupoids over the space \([0,\infty)\) in the sense of \cite{landsman}*{5.2}. 
\begin{lemma}\label{res:contfield}
	The tangent groupoid \(\T_HM\) of a filtered manifold \((M,H)\) is a smooth field of groupoids over \([0,\infty)\) with fibres isomorphic to the pair groupoid of \(M\) for \(t>0\) and the osculating groupoid \(T_HM\) for \(t=0\). All these subgroupoids are amenable.
\end{lemma}
\begin{proof}The projection \(\theta\colon \T_HM\to [0,\infty)\) is a smooth submersion. It satisfies \(\theta=\pr_2\circ r= \pr_2\circ s\) where \(\pr_2\colon M\times[0,\infty)\to [0,\infty)\) is the projection to the second coordinate. Restricting the structure maps of \(\T_HM\) to \(t\geq 0\), it is clear that the groupoids \(\theta^{-1}\{t\}\) are the pair groupoid of \(M\) for \(t>0\) and \(T_HM\) for \(t=0\). 
	
	As all fibres of \(T_HM\) are nilpotent Lie groups, thus amenable, it follows from \cite{ad-renault-amenable}*{5.3.4} that \(T_HM\) is amenable. The pair groupoid of \(M\) is amenable as well. 
\end{proof}
\subsection{The zoom action}
	The following zoom action of \(\Rp\) on \(\T_HM\) by Lie groupoid automorphisms was defined in \cite{erp2015groupoid}*{Def.~17} and \cite{haj-higson}*{5.3}. It plays an essential role for the definition of the pseudodifferential calculus in \cite{erp2015groupoid}. We will use it to construct a generalized fixed point algebra.
\begin{definition}\label{def:zoomgroupoid}
	The \emph{zoom action} of \(\Rp\) on \(\T_HM\) is defined for \(\lambda>0\) by
	\begin{align*}
	\alpha_\lambda(x,y,t)&=\left(x,y,\lambda^{-1}t\right) &&\text{for }(x,y,t)\in M\times M\times(0,\infty),\\
	\alpha_\lambda(x,\xi,0)&=(x,\delta_\lambda(\xi),0) &&\text{for }(x,\xi)\in T_HM.
	\end{align*}
\end{definition}
\begin{lemma}
	The zoom action of \(\Rp\) on the tangent groupoid of a filtered manifold \((M,H)\) is a smooth action by Lie groupoid automorphisms. 
\end{lemma}
\begin{proof}
	It can be checked easily that all \(\alpha_\lambda\) are groupoid morphisms with underlying maps of the unit space 
	\[M\times[0,\infty)\to M\times[0,\infty), \quad (x,t)\mapsto(x,\lambda^{-1}t).\]
	They satisfy \(\alpha_{\lambda\mu}=\alpha_{\lambda}\circ\alpha_{\mu}\) for all \(\lambda,\mu>0\) and \(\alpha_1=\Id\). The smoothness on \(M \times M \times (0,\infty)\) is clear. Let \(\kappa\colon V\to U\) be an \(H\)\nb-chart for \(M\) and \(\phi_\kappa\colon\V\to\U\) the corresponding chart for \(\T_HM\). Then one computes for \(\lambda>0\) and \((x,v,t)\in \U\)
	\begin{equation}\label{eq:zoomincoord}
	\beta_\lambda(x,v,t)\defeq\left(\phi_\kappa\circ\alpha_\lambda\circ\phi_\kappa^{-1}\right)(x,v,t)=\left(x,\lambda\cdot v,\lambda^{-1}t\right),
	\end{equation}
	where \(\lambda\cdot (v_1,\ldots,v_n)= (\lambda^{q_1}v_1,\ldots,\lambda^{q_n}v_n)\).  Hence, the action is smooth.
\end{proof}
\begin{definition}
	Define a \emph{homogeneous quasi-norm} \(\norm{\,\cdot\,}\colon \U\to \R_{\geq 0}\) for \(H\)-charts \(V\to U\) by
	\begin{align}\label{eq:homogeneousquasi}
		\norm{(x,v,t)}=\sum_{j=1}^n\abs{v_j}^{1/q_j} \qquad\text{for }(x,v,t)\in\U.
	\end{align}
\end{definition}
Even though this is not a norm, it has the advantage of being compatible with the zoom action. Namely, it satisfies \(\norm{\beta_\lambda(\gamma)}=\lambda\norm{\gamma}\)  for all \(\lambda>0\) and \(\gamma\in\U\), where \(\beta\) is the zoom action in coordinates as in \eqref{eq:zoomincoord}.

\subsection{Functoriality}	
For a filtered manifold map \(f\colon(M_1,H_1)\to(M_2,H_2)\), we already know from \cref{res:osculating_functor} that it induces a Lie groupoid morphism \(Tf\colon T_{H_1}M_1\to T_{H_2}M_2\). It can be extended to a map of the corresponding tangent groupoids \(\f\colon  \T_{H_1}M_1\to \T_{H_2}M_2\) with
\begin{align*}
\f(x,y,t)&=(f(x),f(y),t)&&\text{for }(x,y,t)\in M_1\times M_1\times(0,\infty),\\
\f(x,\xi,0)&=(f(x),T_xf(\xi),0) &&\text{for }(x,\xi)\in T_{H_1}M_1.
\end{align*}
Note that the induced map \(\f\) is equivariant for the respective zoom actions.
\begin{proposition}[\cite{choi2019tangent}*{9.17, 9.18}]\label{res:inducedcarnot}
	Let \(f\colon(M_1,H_1)\to(M_2,H_2)\) be a filtered manifold map. Then \(\f\colon  \T_{H_1}M_1\to \T_{H_2}M_2\) is a Lie groupoid homomorphism with underlying map \((x,t)\mapsto(f(x),t)\) of the unit spaces. 
	The assignments \((M,H)\mapsto \T_HM\) and \(f\mapsto\f\) define a functor from the category of filtered manifolds to the category of Lie groupoids. 	
\end{proposition}
\section{The groupoid \(\Cst\)-algebra of the tangent groupoid}\label{sec:groupoid_c_star}
To build a generalized fixed point algebra out of the zoom action on the tangent groupoid, we need to attach a \(\Cst\)-algebra to the tangent groupoid. In this section we recall the construction of groupoid \(\Cst\)-algebras.  
\subsection{Haar system}
As \(\T_HM\) is a Lie groupoid, it admits a smooth left Haar system (see for example \cite{paterson}*{2.3.1}). 
In the following, we explicitly describe a left Haar system. 

Fix an atlas of \(H\)-charts \((\kappa_i\colon M\supseteq V_i\to U_i\subseteq \R^n)_{i\in I}\) for \(M\). Let \((\rho_i)_{i\in I}\) be a partition of unity which is subordinate to the open cover \((V_i)_{i\in I}\) of \(M\). One can define a measure \(\nu\) on \(M\) by setting
\[ \int_M f\diff \nu = \sum_{i\in I}\int_{U_i}(f\cdot \rho_i)(\kappa_i^{-1}(x))\diff x \quad\text{for }f\in\Cont_c(M),\]
where \(\diff x\) denotes the Lebesgue measure on \(U_i\subseteq \R^n\). Furthermore, the atlas of \(H\)-charts gives rise to a smooth family of measures on the osculating groups. Each \(H\)-chart \(\kappa\colon V\to U\) induces a local trivialisation \[T\kappa\colon T_HM|_V\overset{\cong}{\to} U\times \R^n.\]
The Lebesgue measure on \(\R^n\) can be pulled back using the graded isomorphism \[T_x\kappa\colon G_x\overset{\cong}{\to} \R^n \quad \text{for all }x\in V.\]  Write the vector fields corresponding to the \(H\)-frame \(\kappa_*(X_j)\) for \(j=1,\ldots,n\) in terms of the coordinate vector fields as
\[\kappa_*(X_j)=\sum_{k=1}^nb_{jk}\frac{\partial}{\partial x_k}\quad\text{with }b_{jk}\in\Cont^\infty(U).\]
Let \(B_X(x)\defeq(b_{jk}(x))_{j,k=1}^n\) for \(x\in U\) and define as in \cite{choi2017privilegedI}*{3.9} the invertible, affine linear map 
\begin{align}\label{eq:affine_map}A_{x}(y)\defeq(B_X(x)^t)^{-1}(y-x).\end{align} Recall that \(d_H\) denotes the homogeneous dimension of \(M\). Define for \(f\in\Cont_c(\T_HM)\)
\begin{align*} \int f \diff\nu^{(x,0)} &\defeq \sum_{i\in I} \int_{\R^n}\rho_i(x)\abs{\det{B_X(\kappa(x))}}f(x,(T_x\kappa_i)^{-1}(v),0)\diff v&&\text{for }x\in M,\\
\int f \diff\nu^{(x,t)}&\defeq t^{-d_H} \int_M f(x,y,t)\diff\nu(y)&&\text{for }x\in M\text{, }t>0.
\end{align*}
\begin{lemma}
	The family of measures \(\{\nu^{(x,t)}\}_{(x,t)\in M\times[0,\infty)}\) defines a smooth left Haar system on \(\T_HM\).
\end{lemma}
\begin{proof}
	For \((x,t)\in M\times[0,\infty)\) the support of \(\nu^{(x,t)}\) is contained in \(\T_HM^{(x,t)}\). The left invariance follows for \(t>0\) as in the pair groupoid case. For \(t=0\), this is due to the fact that the Lebesgue measure induces a bi-invariant Haar measure on the osculating groups. For \(f\in\Cont_c^\infty(\T_HM)\), we show that the map
	\[(x,t)\mapsto \int f\diff \nu^{(x,t)}\] is smooth. Using the partition of unity, \(f\) can be written as a finite sum 
	\(f=\sum_{i\in I\cup\{\infty\}}f_i\) 
	with~\(f_i\in\Cont_c^\infty(\V_i)\) for \(i\in I\cup\{\infty\}\). As smoothness for \(t>0\) is clear, it suffices to prove for all \(H\)-charts~\(\kappa\colon V\to U\) and \(f\in\Cont_c^\infty(\U)\) that the following map is smooth
	\begin{align*}
	(x,t)&\mapsto t^{-d_H}\int_U (f\circ\phi_\kappa)(x,\kappa^{-1}(y),t)\diff y \qquad \text{for }t>0 ,\\
	(x,0)&\mapsto \abs{\det{B_X(\kappa(x))}}\int_{\R^n} (f\circ\phi_\kappa)(x,(T_x\kappa)^{-1}(v),0)\diff v =\abs{\det{B_X(\kappa(x))}}\int_{\R^n}f(\kappa(x),v,0)\diff v,
	\end{align*}
	For \(x\in V\) and \(t>0\) consider the diffeomorphism 
	\begin{align*}
	\tilde{\phi}_\kappa^{(x,t)}\colon U &\to \U^{(x,t)}\defeq\{v\in \R^n \mid (\varepsilon^\kappa_{\kappa(x)})^{-1}(t\cdot v)\in U\},\\
	y &\mapsto t^{-1}\cdot \varepsilon^\kappa_{\kappa(x)}(y).
	\end{align*}
	By \cite{choi2019tangent}*{4.17} and \cite{choi2017privilegedII}*{9.15} \(\varepsilon^\kappa_{\kappa(x)}\) can be decomposed as \(\varepsilon^\kappa_{\kappa(x)}=\hat{\varepsilon}^\kappa_{\kappa(x)}\circ A_{\kappa(x)}\) with \(A_{\kappa(x)}\) as in~\eqref{eq:affine_map}. Moreover, it follows from their description of \(\hat{\varepsilon}^\kappa\) that the differential \(\diff(\hat{\varepsilon}^\kappa_{\kappa(x)})(y)\) is of upper triangular form with ones on the diagonal. Consequently, \(\abs{\det (\diff\tilde  {\phi}_\kappa^{(x,t)}(y))}=t^{-d_H}\abs{\det{B_X(\kappa(x))}}^{-1}\) holds for all \(y\in U\). Therefore, we obtain
	\begin{align*}
	t^{-d_H}\int_U (f\circ\phi_\kappa)(x,\kappa^{-1}(y),t)\diff y &= t^{-d_H}\int_U f(\kappa(x),t^{-1}\cdot \varepsilon^\kappa_{\kappa(x)}(y),t)\diff y\\
	& = \abs{\det{B_X(\kappa(x))}}\int_{\U^{(x,t)}} f(\kappa(x),v,t)\diff v.
	\end{align*}
	Thus the Haar measure is smooth. 
\end{proof}
\subsection{The groupoid \(\Cst\)-algebra}
Using the left Haar system \(\{\nu^{(x,t)}\}_{(x,t)\in M\times[0,\infty)}\),  the linear space \(\Cont_c(\T_HM)\) can be equipped with the following involution and convolution:
\begin{align}
f^*(\gamma)&=\conj{f(\gamma^{-1})},\label{def:inv}\\
(f*g)(\gamma)&= \int f(\gamma\eta)g(\eta^{-1})\diff\nu^{s(\gamma)}(\eta)=\int f(\eta)g(\eta^{-1}\gamma)\diff\nu^{r(\gamma)}(\eta)\label{def:conv}
\end{align}
for \(f,g\in\Cont_c(\T_HM)\) and \(\gamma\in\T_HM\). More explicitly, the involution is given by 
\[f^*(x,y,t)=\conj{f(y,x,t)} \quad\text{for }t>0, \qquad f^*(x,\xi,0)=\conj{f(x,\xi^{-1},0)}.\]
The convolution can be written as
\begin{align*}
(f*g)(x,y,t)&=t^{-d_H}\int_M f(x,z,t)g(z,y,t)\diff\nu(z) \quad\text{for }t>0,\\
(f*g)(x,\xi,0)&= \int f(x,\eta,0)g(x,\eta^{-1}\xi,0)\diff\nu^{(x,0)}(x,\eta,0).
\end{align*}
Let the \(I\)\nb-norm on \(\Cont_c(\T_HM)\) be given by \(\norm{f}_I=\max\{\norm{f}_{I,r},\norm{f}_{I,s}\}\), where
\begin{align*}
\norm{f}_{I,r}&= \sup_{(x,t)}{\int \abs{f}\diff \nu^{(x,t)}}
\end{align*}
and \(\norm{f}_{I,s}=\norm{f^*}_{I,r}\). The (full) groupoid \(\Cst\)-algebra of \(\T_HM\) is defined as the \(\Cst\)-completion of \(\Cont_c(\T_HM)\) with respect to representations that are bounded by the \(I\)-norm  as in \cite{renaultgroupoid}*{II,1.12}. 
\begin{example}
	Let \(G\) be a graded Lie group. By the description of \(\T_HG\) as a transformation groupoid in \cref{ex:tangentofgradedliegroup},  \(\Cst(\T_HG)\) is isomorphic to the crossed product
	\( \Cst(G,\Cont_0(G\times[0,\infty))\) (see \cite{renaultgroupoid}).
\end{example}

\subsection{Continuous field structure} As the tangent groupoid \(\T_HM\) is a continuous field of amenable groupoids, its \(\Cst\)-algebra admits a continuous field structure. The same is true for the bundle of osculating groups \(T_HM\) viewed as a continuous bundle of groups over \(M\).
\begin{proposition}
	The \(\Cst\)-algebra of the tangent groupoid \(\T_HM\) is a continuous field of \(\Cst\)-algebras over \([0,\infty)\) with fibres isomorphic to \(\Cst(T_HM)\) for \(t=0\) and the \(\Cst\)-algebra of compact operators \(\Comp(L^2M)\) for \(t>0\).
\end{proposition}
\begin{proof}
	Recall that \(\T_HM\) is a continuous field of groupoids by \cref{res:contfield}. As all groupoids~\(\theta^{-1}\{t\}\) are amenable, \(\Cst(\T_HM)\) defines a continuous field of \(\Cst\)\nb-algebras with fibres \(\Cst(\theta^{-1}\{t\})\) by \cite{landsman}*{5.6}. 	
	For \(t>0\) the groupoid \(\theta^{-1}\{t\}\) is isomorphic to the pair groupoid of \(M\). The Haar measure on \(\theta^{-1}\{t\}\) is given by
	\[\int K(\gamma)\diff\mu_t^x(\gamma)=t^{-d_H}\int_MK(x,y)\diff \nu(y) \qquad\text{for }K\in\Cont_c(M\times M).\]
	There is a well-known isomorphism \(\Phi_t\colon\Cst(M\times M,\mu_t)\to \Comp(L^2M)\) with
	\[(\Phi_t(K)\psi)(x)=t^{-d_H}\int_M K(x,y)\psi(y)\diff\nu (y)\]
	for \(K\in\Cont_c(M\times M)\) and \(\psi\in \Cont_c(M)\).	
	Hence, for \(t>0\) we obtain  epimorphisms \(p_t\colon \Cst(\T_HM)\to \Comp(L^2M)\) given by
	\begin{align}\label{eq:p_t}
	\left(p_t(f)\psi\right)(x)= t^{-d_H}\int_M{ f\left(x,y,t\right)\psi(y)\diff \nu(y)}
	\end{align} 
	for \(f\in\Cont_c(\T_HM)\), \(\psi\in \Cont_c(M)\) and \(x\in M\). 
\end{proof}
Now consider the \(\Cst\)-algebra of the osculating groupoid at \(t=0\). As \(T_HM\) is a continuous field of amenable groups over \(M\), \cite{landsman}*{5.6} applies again. 
\begin{lemma}
	The \(\Cst\)-algebra \(\Cst(T_HM)\) is a continuous field of \(\Cst\)-algebras over~\(M\) with fibre projections
	\[q_x\colon\Cst(T_HM)\to\Cst(G_x)\quad \text{for }x\in M.\]
\end{lemma}

\begin{lemma}
	Denote by \(p_0\colon\Cst(\T_HM)\to\Cst(T_HM)\) the \(^*\)\nb-homomorphism induced by restriction to \(t=0\). There is a corresponding short exact sequence 
	\begin{equation}\label{ses:tangentgroupoid}
	\begin{tikzcd}
	\Cont_0(\Rp)\otimes\Comp(L^2M) \arrow[r,hook] & \Cst(\T_HM)\arrow[r,twoheadrightarrow,"p_0"] & \Cst(T_HM).
	\end{tikzcd}
	\end{equation}		
\end{lemma}
\begin{proof}
	The subset \(M\times(0,\infty)\subset M\times [0,\infty)\) is open and invariant.  By \cites{hilsumskandalis}, the kernel of \(p_0\) is \(\Cst(\T_HM|_{M\times(0,\infty)})\). The fibre projections \(p_t\) from \eqref{eq:p_t} for \(t>0\) combine to an isomorphism
	\begin{align}\label{eq:iso_p}
	p\colon\Cst(\T_HM|_{M\times(0,\infty)})&\to\Cont_0(\Rp,\Comp(L^2M))
	\end{align}
	defined by \(p(f)(t)=p_t(f)\) for \(f\in\Cont_c(\T_HM|_{M\times(0,\infty)})\).
\end{proof}

\begin{remark}If the filtration is of step \(r=1\), the \(\Cst\)\nb-algebra \(\Cst(T_HM)\) is isomorphic to \(\Cont_0(T^*M)\). Namely, the fibrewise Fourier transform yields an isomorphism \(\Cst(TM)\to\Cont_0(T^*M)\). If the osculating groups are not Abelian, \(\Cst(T_HM)\) is noncommutative.
\end{remark}
\section{A Schwartz type algebra}\label{sec:schwartztype}
In this section the Schwartz type algebra \(\A(\T_HM)\subset\Cst(\T_HM)\) is defined by adapting the construction in \cite{rouseschwartz} to the filtered manifold setting. The Schwartz type algebra consists of functions \(f\in\Cont^\infty(\T_HM)\) which restrict at \(t>0\) to a compactly supported function~\(f_t\in\Cont^\infty_c(M\times M)\), whereas \(f_0\) has rapid decay in the fibres of~\(T_HM\).  This algebra will be convenient for the generalized fixed point algebra construction.
\subsection{Definition of the Schwartz type algebra}
First, the Schwartz type algebra will be defined locally as in \cite{rouseschwartz}*{4.1} using the charts  \(\phi_\kappa\colon \V\to\U\) of \(\T_HM\) obtained from \(H\)-charts \(\kappa\colon V\to U\). In the following we use the homogeneous quasi-norm on \(\U\) from~\eqref{eq:homogeneousquasi}. Consider the smooth function \[k\colon U\times\R^n\times[0,\infty)\to U\times\R^n\] given by \(k(x,v,t)=(x,t\cdot v)\). Recall that \(\theta\colon \T_HM\to[0,\infty)\) denotes the projection.
\begin{definition}\label{def:schwartz_local}
	Let \(\A(\U)\) consist of all functions \(f\in\Cont^\infty(\U)\) satisfying
	\begin{enumerate}
		\item\label{item:compactsupportlocal} there is \(T>0\) and a compact subset \(K\subset k(\U)\) such that \((k(\gamma),\theta(\gamma))\notin K\times[0,T]\) implies \(f(\gamma)=0\), 
		\item \label{item:rapiddecaylocal} for all \(p\in\N_0\) and \(\alpha=(\alpha_1,\alpha_2,\alpha_3)\in\N^n_0\times\N^n_0\times\N_0\) there is a constant \(D_{p,\alpha}>0\) such that
		\[\sup_{\gamma\in\U}{(1+\norm{\gamma})^p\abs{\partial_x^{\alpha_1}\partial_v^{\alpha_2}\partial_t^{\alpha_3}f(\gamma)}}\leq D_{p,\alpha}.\]
	\end{enumerate}
\end{definition}	
We check first that this space is invariant under Carnot diffeomorphisms.  
\begin{proposition}[\cite{rouseschwartz}*{4.2}]\label{res:schwartzinvar}
	Let \(F\colon U_1\to U_2\) be a Carnot diffeomorphism and \(\F\colon \U_1\to \U_2\) the induced map as in \cref{res:inducedcarnot}. Then \(f\circ \F\in \A(\U_1)\) for all \(f\in\A(\U_2)\).  
\end{proposition}
\begin{proof}
	As the induced map \(\F\colon \U_1\to \U_2\) is smooth, it is clear that \(f\circ \F\) is smooth for all \(f\in\A(\U_2)\).  In fact, as in the proof of \cite{choi2019tangent}*{9.15} we have
	\begin{align*}
	\F(x,v)=\begin{cases}(F(x),t^{-1}\cdot (\varepsilon_{F(x)}\circ F\circ \varepsilon_x^{-1})(t\cdot v),t)&\text{for }t>0,\\
	(F(x),T_xF(v),0) &\text{for }t=0,
	\end{cases}
	\end{align*}
	for \((x,v,t)\in\U_1\) and the respective \(\varepsilon\)-Carnot maps \( U_i\times\R^n\to\R^n\) for \(i=1,2\). 
	To show condition~\ref{item:compactsupportlocal}, define \(F_k\colon k_1(\U_1)\to k_2(\U_2)\) by
	\begin{align*}
	(x,v)\mapsto \left(F(x), \left(\varepsilon_{F(x)}\circ F\circ\varepsilon^{-1}_x\right)(v)\right).
	\end{align*}
	It is a diffeomorphism with inverse \((F^{-1})_k\). The following diagram commutes
	\begin{equation*}
	\begin{tikzcd}
	\U_1 \arrow[r, "\F"] \arrow[d, "k_1"]& \U_2\arrow[d, "k_2"]	\\
	k_1(\U_1) \arrow[r, "F_k"]& k_2(\U_2).
	\end{tikzcd}
	\end{equation*}	
	Let \(K_2\subset k_2(\U_2)\) be a compact subset for \(f\) as in \ref{item:compactsupportlocal}. Then \(K_1\defeq (F_k)^{-1}(K_2)\) is a compact subset such that \((f\circ\F)(\gamma)=0\) if \(k_1(\gamma)\notin K_1\). If \(f\) vanishes for \(t\geq T\), also \(f\circ\F\) vanishes for \(t\geq T\). 
	
	For the rapid decay property, write \(\F(x,v,t)=(F(x),w(x,v,t),t)\). Because of this structure of~\(\F\), one can write for \(\gamma=(x,v,t)\), \(\eta=\F(\gamma)\) and \(\alpha\in\N_0^n\times\N_0^n\times\N_0\)
	\begin{align*}
	\partial^\alpha_\gamma(f\circ\F)(\gamma)=\sum_{\abs{\delta}\leq\abs{\alpha}} \partial^\delta_\eta f(\eta)\cdot P_\delta(\gamma),
	\end{align*}
	where \(P_\delta\) is a finite sum of products of the form
	\[\partial_x^{\delta_1}F_i(x)\cdot \partial^{\delta_2}_\gamma w_j(x,v,t).\]
	We only need to estimate each \(P_\delta\) for \(\gamma=(x,v,t)\) such that  \(k(x,v,t)\in K_1\) as otherwise \(\partial^\delta_\eta f(\F(\gamma))=0\). In particular, \(x\) is contained in a compact subset, so that \(\partial_x^{\delta_1}F_i(x)\) is bounded. 
	By \cite{choi2019tangent}*{6.7},
	\[\varepsilon_{F(x)}\circ F \circ \varepsilon_x^{-1}(v)=T_xF(v)+O_q(\norm{v}^{q+1}).\]
	holds near \(v=0\). Here, \(O_q(\norm{v})\) is defined as in \cite{choi2019tangent}*{3.2}. 
	Hence, \cite{choi2019tangent}*{3.9} implies that there are smooth \(R_{j,\alpha}\) for \(j=1,\ldots,n\) such that
	\[w_j(x,v,t)=(T_xF(v))_j+\sum_{\abs{\alpha}\leq q_j +1 \leq [\alpha]}t^{[\alpha]-q_j}v^\alpha R_{j,\alpha}(x,t\cdot v). \]
	For all \((x,v,t)\in k^{-1}(K_1)\) the components \(x\) and \((x,t\cdot v)\) are in compact sets. Moreover, we only need to consider \(t\leq T\). It follows that one can find \(C_{\delta_2,j}>0\) and \(m_{\delta_2,j}\) such that
	\[\abs{\partial^{\delta_2}_\gamma w_j(\gamma)}\leq C_{\delta_2,j}(1+\norm{\gamma})^{m_{\delta_2,j}}\quad \text{for all }\gamma\in k^{-1}(K)\cap\theta^{-1}[0,T].\]
	Together, this means that there are \(C_\delta>0\) and \(m_\delta\in\N\) such that \(\abs{P_\delta(\gamma)}\leq C_\delta(1+\norm{\gamma})^{m_\delta}\) for all such \(\gamma\).
	As \(\F\) is a Carnot diffeomorphism, one can find likewise \(D>0\) and \(l>0\) such that for all \(\gamma\) with \(k(\gamma)\in K_1\)
	\[1+\norm{\gamma}\leq D (1+\norm{\F(\gamma)})^l.\]
	Let \(p\in\N_0\). Because \(f\) satisfies the rapid decay condition \ref{item:rapiddecaylocal}, there are constants \(D_{l(m_\delta+p),\delta}>0\) such that for all  \(\eta\in\U_2\)
	\[\abs{\partial^\delta_\eta f(\eta)}\leq D_{l(m_\delta+p),\delta}(1+\norm{\eta})^{-l(k+m_\delta)}.\]
	It follows that 
	\(\sup_{\gamma\in \U}(1+\norm{\gamma})^p\abs{\partial^\alpha_\gamma \F(\gamma)}< \infty\).
\end{proof}
The invariance under Carnot diffeomorphisms allows to define the Schwartz type algebra in the following way. 
For an \(H\)-chart \(\kappa\colon V\to U\) let \[\A(\V)\defeq\{f\in\Cont^\infty(\V)\mid f\circ\phi_\kappa^{-1}\in\A(\U)\}.\]	
We will denote by \(r_1,s_1\colon \T_HM\to M\) the maps given by \(r_1=\pr_1\circ r\) and \(s_1=\pr_1\circ s\).
\begin{definition}[\cite{rouseschwartz}*{4.4}]\label{def:schwartz_global}
	The \emph{Schwartz type algebra}  \(\A(\T_HM)\) is the space of functions \(f\in\Cont^\infty(\T_HM)\) such that
	\begin{enumerate}
		\item \label{item:compact} there are \(T>0\) and a compact subset \(K\subset M\times M\) such that  \((r_1,s_1,\theta)(\gamma)\notin K\times[0,T]\) implies \(f(\gamma)=0\),
		\item \label{item:rapid} \(f\) \emph{has rapid decay at \(t=0\)}, that is, for all \(H\)-charts \(\kappa\colon V \to U\) and \(\chi\in\Cont^\infty_c(V\times V\times[0,\infty))\), the function \(f_\chi\) belongs to \(\A(\V)\), where
		\[f_\chi(\gamma)=(\chi\circ (r_1,s_1,\theta))(\gamma)f(\gamma)\quad\text{for }\gamma\in\V.\]			
	\end{enumerate}
\end{definition}
We will verify later that \(\A(\T_HM)\) is indeed an algebra. 
\begin{lemma}[\cite{rouseschwartz}*{(5)}]\label{res:decompschwartz}
	Let \((\kappa_i\colon V_i\to U_i)_{i\in I}\) be an atlas of \(H\)-charts for \(M\). The space \(\A(\T_HM)\) can be decomposed as 
	\[\A(\T_HM)=\sum_{i\in I}\A(\V_i)+\Cont^\infty_c(\V_\infty).\]
\end{lemma}
\begin{proof}
	Let \(f\in\A(\V)\) for an \(H\)-chart \(\kappa\colon V\to U\). We claim that \(f\in\A(\T_HM)\). There is a diffeomorphism \(\kappa_k\colon V\times V\to k(\U)\) given by
	\[(x,y)\mapsto \left(\kappa(x),\varepsilon^\kappa_{\kappa(x)}(\kappa(y))\right).\] It makes the following diagram commute
	\begin{equation*}
	\begin{tikzcd}
	\V \arrow[r, "\phi_\kappa"] \arrow[d, "{(r_1,s_1)}"]& \U\arrow[d, "k"]	\\
	V\times V \arrow[r, "\kappa_k"]& k(\U).
	\end{tikzcd}
	\end{equation*}			
	Let \(K_\U\subset k(\U)\) be a compact subset for \(f\circ\phi_\kappa^{-1}\) as in \cref{def:schwartz_local}. Then \(\varphi^{-1}(K_\U)\) is compact and \(f(\gamma)=0\) whenever \((r_1,s_1)(\gamma)\notin \varphi^{-1}(K_\U)\). In fact, condition \ref{item:compact} from \cref{def:schwartz_global} for \(f\in\Cont^\infty(\V)\) is equivalent to \ref{item:compactsupportlocal} from \cref{def:schwartz_local} for \(f\circ\phi_\kappa^{-1}\in\Cont^\infty(\U)\). Moreover, \(f\) has rapid decay at \(t=0\) by the invariance under Carnot diffeomorphisms from \cref{res:schwartzinvar}. Clearly, \(\A(\T_HM)\) contains all smooth function with compact support in \(M\times M\times\Rp\) and is closed under finite sums. 
	
	For the converse inclusion, let \(f\in\A(\T_HM)\) and let \(K\subset M\times M\) be a compact set and \(T>0\) as in \ref{item:compact}. Note that \(V_i\times V_i\times[0,\infty)\) for \(i\in I\) and \(M\times M\times(0,\infty)\) define an open cover of \(K\times[0,T]\). Therefore, there is a finite partition of unity \((\rho_i)_{i\in I\cup\{\infty\}}\) consisting of smooth, compactly supported functions subordinate to this open cover. By \ref{item:rapid}, \(f_i\defeq f_{\rho_i}\) are in \(\A(\V_i)\) for \(i\in I\) and \(f_\infty\in\Cont^\infty_c(\V_\infty)\). This yields a decomposition of \(f\) as above. 
\end{proof}
For a vector bundle \(E\to M\) consider functions that have uniform rapid decay in the fibres. 
\begin{definition}[\cite{rouseschwartz}*{4.6}]
	Let \(\pi\colon E\to M\) be a smooth vector bundle. A function \(f\in\Cont^\infty(E)\) has \emph{uniform rapid decay in the fibres}, if for all local trivializations \(\varphi\colon E|_V\to V\times\R^m \), \(p\in\N_0\) and \(\alpha=(\alpha_1,\alpha_2)\in\N^n_0\times\N^m_0 \) and all cutoff functions \(\chi\in\Cont^\infty_c(V)\)
	\[\sup_{(x,v)\in V\times\R^m}(1+\abs{v})^p\abs{\partial_x^{\alpha_1}\partial_v^{\alpha_2} \chi(x)f(\varphi^{-1}(x,v))}<\infty. \]
	Let \(\Schwartz(E)\) be the space of functions with uniform rapid decay in the fibres. Let \(\Schwartz_{\cp}(E)\) consist of all \(f\in\Schwartz(E)\) such that \(\pi(\supp f)\) is compact.
\end{definition}	

\begin{lemma} \label{res:schwartzrestsurj}
	The restrictions \(e_t\colon f\mapsto f_t\) for \(t\in[0,\infty)\) yield surjections
	\begin{align*}
	e_t\colon \A(\T_HM) &\to \Cont_c^\infty(M\times M)\qquad\text{for } t>0,\\
	e_0\colon \A(\T_HM) &\to \Schwartz_{\cp}(T_HM).			
	\end{align*}					
\end{lemma}

\begin{proof}Let \(f\in\A(\T_HM)\). Condition \ref{item:compact} ensures that \(f_t\) is compactly supported for each \(t>0\). For \(t=0\) it implies compact support in the \(M\)-direction. Moreover, \(f_0\) belongs to \(\Schwartz_{\cp}(T_HM)\) as any locally defined norm on the fibres of \(T_HM\) is equivalent to the homogeneous quasi-norm. 
	
	For \(t>0\), surjectivity is easily seen by  extending a function in \(\Cont^\infty_c(M\times M\times\{t\})\) smoothly to a function in \(\Cont^\infty_c(M\times M\times\Rp)\). 
	
	At \(t=0\), it suffices to show that \(e_0\colon \A(\U)\to\Schwartz_{\cp}(U\times\R^n)\) is surjective for each \(H\)-chart \(\kappa\colon V\to U\). Let \(f_0\in\Schwartz_{\cp}(U\times\R^n)\) and let \(K_0\subset U\) be a compact subset such that \(f_0(x,v)=0\) whenever \(x\notin K_0\). Let \(q\) be a common multiple of the weights \(q_1,\ldots,q_n\in \N\). Define a smooth function \(\Phi\colon\R^n\to[0,\infty)\) which is \((2q)\)-homogeneous with respect to the dilations by \[\Phi(v)\defeq\sum_{j=1}^n v_j^{2q/q_j} \quad\text{for }v\in\R^n.\]
	One can estimate \(\norm{v}\leq \Phi(v)\leq n\norm{v}^{2q}\) for all \(v\in\R^n\) and the homogeneous quasi-norm. As \(K_0\) is compact there is a \(1>\delta>0\) such that \((\varepsilon^\kappa_x)^{-1}(v)\in U\) for all \(x\in K_0\) and \(\Phi(v)\leq \delta\).
	Choose a smooth function \(0\leq\chi\leq 1\) on \([0,\infty)\) which satisfies \(\chi(0)=1\) and \(\chi(t)=0\) whenever \(t\geq \delta\). Define \(f(x,v,t)\defeq f_0(x,v)\chi(t)\chi(\Phi(t\cdot v))\). The following set is compact and contained in \(k(\U)\) 
	\[K\defeq\left\{(x,v)\in U\times\R^n\mid x\in K_0\text{ and }\Phi(v)\leq \delta \right\}.\]
	Then \(f(x,v,t)=0\) whenever \(k(x,v,t)\notin K\) or \(t\geq \delta\). Moreover, \(f\) satisfies the rapid decay condition and \(e_0(f)=f_0\). 		
\end{proof}
\subsection{Algebra structure} We proceed by showing that the Schwartz type algebra is a \(^*\)-algebra with respect to the operations in \eqref{def:inv} and \eqref{def:conv}. First, we prove the following estimates for the groupoid inversion \(i\colon\U\to\U\) and product \(m\colon \U^{(2)}\to\U\) with respect to the homogeneous quasi-norm. 
\begin{lemma}\label{res:invmultest}
	Let \(K\subset k(\U)\) be compact, \(T>0\) and \(\alpha\in \N_0^{2n+1}\). 
	\begin{enumerate}
		\item \label{item:estim_inversion} There are \(C_{i,\alpha,K,T}>0\) and \(l_{i,\alpha}\in\N\) such that for all \(\gamma\in\U\) with \(k(\gamma)\in K\) and \(\theta(\gamma)\leq T\) 
		\[\abs{\partial^\alpha_\gamma i(\gamma)_j}\leq C_{i,\alpha,K,T}(1+\norm{\gamma})^{l_{i,\alpha}}\]
		for \(j=1,\ldots,2n+1\).
		\item\label{item:estim_mult} There are \(C_{m,\alpha,K,T}>0\) and \(l_{m,\alpha}\in\N\) such that for all \((\gamma,\eta)\in\U^{(2)}\) with \(k(\gamma), k(\eta)\in K\) and \(\theta(\gamma)\leq T\)
		\[\abs{\partial^\alpha_\gamma m(\gamma,\eta)_j}\leq C_{m,\alpha,K,T}(1+\norm{\gamma})^{l_{m,\alpha}}(1+\norm{\eta})^{l_{m,\alpha}}\]
		for \(j=1,\ldots,2n+1\).
	\end{enumerate}
\end{lemma}
\begin{proof}
	For \ref{item:estim_inversion} the inversion is given in local coordinates as \(i\colon \U\to\U\) with
	\begin{align*}
	i(x,v,t)=\begin{cases} \left(\varepsilon^{-1}_x(t\cdot v), t^{-1}\cdot \varepsilon_{\varepsilon^{-1}_x(t\cdot v)}(x),t\right) &\text{for }t>0,\\
	(x,-v,0) &\text{for }t=0,
	\end{cases}
	\end{align*}
	for \((x,v,t)\in\U\) by \cite{choi2019tangent}*{9.9}. As noted there, near \(v=0\)
	\[\varepsilon_{\varepsilon^{-1}_x(v)}(x)=-v+O_q(\norm{v}^{q+1})\] 
	holds. So \cite{choi2019tangent}*{3.9} and the compactness of \(K\times[0,T]\) can be used, similarly as in the proof of \cref{res:schwartzinvar}, to derive bounds of the desired form. 
	
	We proceed similarly for \ref{item:estim_mult} and write for \(((x,v,t),(\varepsilon^{-1}_x(t\cdot v),w,t))\in\U^{(2)}\)
	\begin{align*}
	m((x,v,t),(\varepsilon^{-1}_x(t\cdot v),w,t))=\begin{cases} \left(x,t^{-1}\cdot \left(\varepsilon_x\circ\varepsilon^{-1}_{\varepsilon^{-1}_x(t\cdot v)}\right)(t\cdot w), t\right) &\text{for }t>0,\\
	(x,v\cdot w,0) &\text{for }t=0,
	\end{cases}
	\end{align*}
	as in \cite{choi2019tangent}*{9.11}. By their argument, for all \(x\in U\)
	\[\varepsilon_x\circ\varepsilon^{-1}_{\varepsilon^{-1}_x(v)}(w)= v\cdot_x w+O(\norm{(v,w)}^{q+1})\]
	holds for \((v,w)\) near \((0,0)\). As \(K\times[0,T]\) is compact and the group multiplication is polynomial and depends continuously on \(x\), one obtains estimates of the claimed form using again \cite{choi2019tangent}*{3.9}.
\end{proof}
\begin{corollary}\label{res:schwartztriangle}
	For \(K\subset k(\U)\) compact and \(T>0\) there are \(C_{K,T}>0\) and \(l\in\N\) such that 
	\begin{enumerate}
		\item for all \(\gamma\in \U\) with \(k(\gamma)\in K\) and \(\theta(\gamma)\leq T\)
		\[1+\norm{\gamma^{-1}}\leq C_{K,T}(1+\norm{\gamma})^l,\]
		\item for all  \(\gamma,\eta\in \U\) with \((\eta^{-1},\gamma)\in\U^{(2)}\), \(k(\eta),k(\eta^{-1}\cdot \gamma)\in K\) and \(\theta(\eta)\leq T\)
		\[1+\norm{\gamma}\leq C_{K,T}(1+\norm{\eta})^{l}(1+\norm{\eta^{-1}\cdot \gamma})^{l}.\]		
	\end{enumerate}	
\end{corollary}
\begin{theorem}\label{res:schwartzalgebra}
	The Schwartz type algebra \(\A(\T_HM)\) is a \(^*\)\nb-algebra with respect to the involution and convolution defined in \eqref{def:inv} and \eqref{def:conv}. 
	Moreover, there are inclusions of \(^*\)-algebras \[\Cont^\infty_c(\T_HM)\subset \A(\T_HM)\subset \Cst(\T_HM).\] 
\end{theorem}
\begin{proof}
	For the involution, note that if \(K\subset M\times M\) is a compact subset for \(f\in\A(\T_HM)\) as in \cref{def:schwartz_global}, \(i(K)\subset M\times M\) is a compact subset for~\(f^*\). Here \(i\colon M\times M\to M\times M\) is the inversion on the pair groupoid. By \cref{res:decompschwartz} and linearity, it suffices to show that \(f^*\in\A(\V_j)\) for \(f\in\A(\V_j)\) for \(j\in I\). For \(j=\infty\), it is clear that \(f^*\) lies again in \(\Cont^\infty_c(\V_\infty)\). 
	
	Suppose now that \(f\in\A(\U)\) for an \(H\)-chart \(\kappa\colon V\to U\). 
	Denote by \(K_f, K_{f^*}\subset k(\U)\) the compact subsets for \(f,f^*\) respectively. 
	As \(f^*(\gamma)=\conj{f\circ i(\gamma)}\), the derivatives of the inversion \(i\) can be bounded by powers of \(1+\norm{\gamma}\) as in \cref{res:invmultest} with \(K=K_{f^*}\) and \(T\) such that \(f\) vanishes for \(t\geq T\).  By \cref{res:schwartztriangle},  \(1+\norm{\gamma}\leq C_{K_f,T}(1+\norm{i(\gamma)})^l\) holds for all \(\gamma\in k^{-1}(K_{f})\) with \(\theta(t)\leq T\). Then the Schwartz seminorms for \(f\) can be used to show that \(f^*\) has the rapid decay property. 
	
	For the convolution, note that \(f*g\) for \(f,g\in\A(\T_HM)\) is a well-defined smooth function as~\(f,g\) are Schwartz functions at \(t=0\) and compactly supported otherwise. If \(K_f,K_g\subset M \times M\) are compact subsets for \(f,g\), it follows that \(K\defeq m(K_f,K_g)\) is a compact subset for \(f*g\), where \(m\) denotes the multiplication in the pair groupoid.
	
	To prove the rapid decay, we decompose \(g\) as in \cref{res:decompschwartz} into a finite sum of \(g_j\in\A(\V_j)\), \(j=1,\ldots,m\), and \(g_\infty\in\Cont_c^\infty(\V_\infty)\) and use linearity to write
	\[f*g=\sum_{j\in\{1,\ldots,m,\infty\}}f*g_j.\]
	Let \(T>0\) be such that \(f\) and \(g\) vanish for \(t\geq T\).  As there is a compact subset \(K_{g_j}\subset V_j\times V_j\) for each \(g_j\), it follows that \[K_{f*g_j}\times[0,T]=m(K_f,K_{g_j})\times[0,T]\] 
	is openly covered by \(V_j\times V_j\times [0,\infty)\) and \(M\times M\times(0,\infty)\) as for \(t=0\) the source and range maps coincide. Let \(\{\chi,1-\chi\}\) be a corresponding compactly supported partition of unity.
	Let \(t_0>0\) be such that \(1-\chi\) vanishes for \(t<t_0\) and choose a bump function \(\omega\in\Cont^\infty(\Rp)\) with \(\omega(t)=1\) for \(t\geq t_0\) and \(\omega(t)=0\) for \(t<t_0/2\). Then we can write 
	\begin{align}\label{eq:omega}
	(f*g_j)_{(1-\chi)}= \left(f\cdot(\omega\circ\theta)*g_j\cdot(\omega\circ\theta)\right)_{(1-\chi)}.
	\end{align}
	This is a convolution of functions in \(\Cont^\infty_c(M\times M\times\Rp)\), so that the result is clearly contained in \(\A(\T_HM)\). 
	
	Consider now \((f*g_j)_\chi\). Let \(\omega_i\in\Cont_c^\infty(V_j)\) for \(i=1,2\) be such that \(\omega_1(x)=1\) for all \(x\in r_1(\supp \chi) \) and \(\omega_2(x)=1\) for all \(x\in r(K_{g_j})\). For each \(\gamma\in \T_HM\)
	\begin{align*}
	(f*g_j)_\chi(\gamma)&= \chi\circ(r_1,s_1,\theta)(\gamma)\int f(\gamma\eta)g_j(\eta^{-1})\diff^{s(\gamma)}(\eta)\\
	& = \chi\circ(r_1,s_1,\theta)(\gamma)\,\omega_1(r_1(\gamma))\int f(\gamma\eta)\,\omega_2(r_1(\eta^{-1}))g_j(\eta^{-1})\diff^{s(\gamma)}(\eta)\\
	&= \chi\circ(r_1,s_1,\theta)(\gamma)\int \omega_1(r_1(\gamma\eta))f(\gamma\eta)\,\omega_2(s_1(\gamma\eta)))g_j(\eta^{-1})\diff^{s(\gamma)}(\eta)\\
	&= \left(f\cdot(\omega_1\circ r_1)\cdot (\omega_2\circ s_1)*g_j\right)_\chi
	\end{align*}
	holds. In conclusion, we obtained a finite decomposition
	\begin{align}\label{eq:decomposeconvolution}
	f*g=\sum_{j=1}^m(f_j*g_j)_{\chi_j}+\sum_{j\in\{1,\ldots,m,\infty\}} f^\infty_j*g^\infty_j
	\end{align}
	with \(f^\infty_j,g^\infty_j\in\Cont^\infty_c(\V_\infty)\), \(f_j,g_j\in \A(\V_j)\) and \(\chi_j\in\Cont_c^\infty(V_j\times V_j\times[0,\infty))\). 
	Therefore, it is left to show that \(\A(\U)*\A(\U)\subseteq \A(\U)\).
	
	Let \(f,g\in\A(\U)\) and denote by \(K_f,K_g,K_{f*g}\subseteq k(\U)\) the respective compact subsets, and let \(T>0\) be such that \(f,g\) vanish for \(t\geq T\).
		To show rapid decay, we must estimate the derivatives
	\[\partial_\gamma^\alpha g(\eta^{-1}\gamma)=\sum \partial^{\delta}_{\eta^{-1}\gamma}g(\eta^{-1}\gamma)\cdot M^{\delta}(\eta^{-1},\gamma) \quad\text{for }\eta^{-1}\gamma\in k^{-1}(K_g)\text{ and }t< T,\]
	where \(M^{\delta}(\eta^{-1},\gamma)\) is a product of \(\partial^{\delta_i}_\gamma m_j(\eta^{-1},\gamma)\).
	The bounds from \cref{res:invmultest} for the multiplication and inverse allow to estimate \(\abs{M^\delta}\) for all \(\gamma\in k^{-1}(K_{f*g})\) and \(\eta\in k^{-1}(K_f)\) with \(t<T\) by \(C(1+\norm{\eta})^r(1+\norm{\gamma})^s\) for some \(C>0\) and \(r,s\in\N_0\). Therefore, we can use the rapid decay of \(f\) and \(g\) to estimate for \(p\in\N\) using \cref{res:schwartztriangle}
	\begin{align*}
	&(1+\norm{\gamma})^p\int \abs{f(\eta)} \abs{\partial^{\delta}_{\eta^{-1}\gamma}g(\eta^{-1}\gamma)\cdot M^{\delta}(\eta^{-1},\gamma)}\diff\nu^{r(\gamma)}(\eta)\\
	\lesssim\, & (1+\norm{\gamma})^{p+s}\int \abs{f(\eta)}(1+\norm{\eta})^r \abs{\partial^{\delta}_{\eta^{-1}\gamma}g(\eta^{-1}\gamma)}\diff\nu^{r(\gamma)}(\eta)\\
	\lesssim\, & (1+\norm{\gamma})^{p+s}\int \abs{f(\eta)} (1+\norm{\eta})^r (1+\norm{\eta^{-1}\gamma})^{-l(p+s)} \diff\nu^{r(\gamma)}(\eta)\\
	\lesssim\, & \int \abs{f(\eta)} (1+\norm{\eta})^{r+l(p+s)}\diff\nu^{r(\gamma)}(\eta)\\
	\lesssim\, & \int (1+\norm{\eta})^{-d_H-1}\diff\nu^{r(\gamma)}(\eta)<\infty
	\end{align*}		
	for all \(\gamma\in\U\). The last integral converges by \cite{folland1982homogeneous}*{1.17}. 
	This finishes the proof that \(\A(\T_HM)\) is a \(^*\)-algebra. 
	
	Clearly, \(\Cont^\infty_c(\T_HM)\) is contained in \(\A(\T_HM)\). For \(f\in\A(\U)\), we can construct a sequence \(f_m\in\Cont_c^\infty(\U)\) which converges to \(f\) in the \(I\)-norm. This will imply that \(\A(\T_HM)\subset\Cst(\T_HM)\). This can be done by choosing a sequence of functions \(\chi_m\in\Cont^\infty_c(\R^n)\) with \(0\leq\chi_m\leq 1\), \(\supp(\chi_m)\subset B(0,m)\) and \(\chi_m|_{B(0,m-1)}\equiv 1\). Then \(f_m(x,v,t)\defeq f(x,v,t)\chi_m(v)\) is such a sequence. 
\end{proof}

\section{Generalized fixed point algebras for filtered manifolds}\label{sec:gfpa}
In this secton we briefly recall the notion of a generalized fixed point algebra. Then we use the Schwartz type algebra to define a generalized fixed point algebra of the zoom action on a certain ideal in the \(\Cst\)-algebra of the tangent groupoid. 
\subsection{Generalized fixed point algebras}
Generalized fixed point algebras were defined by Rieffel as a noncommutative analogue of proper group actions on spaces in \cites{rieffel1998,rieffel1988}. 
Here, we follow the approach of Meyer in \cite{meyer2001}. We recall the main definitions and refer to \cites{meyer2001,ewert2020pseudodifferential} for more details. 

For this section, let \(G\) be a locally compact group and \(A\) a \(\Cst\)-algebra with a strongly continuous \(G\)-action \(\alpha\). For the pseudodifferential operators in the following sections we will always consider the multiplicative group~\(G=\Rp\). 
Define for \(a\in A\) the following operators as in \cite{meyer2001}*{(1),(2)}
\begin{align}\label{BRAKET}
	\BRA{a}\colon& A \to \Cont_b(G,A), &\left(\BRA{a}b\right)(x)&\defeq\alpha_x(a)^*b,\\
	\KET{a}\colon& \Cont_c(G,A) \to A, &\KET{a}f&\defeq \int_G \alpha_x(a)f(x)\diff x.
\end{align}
Here, \(\diff x\) denotes the Haar measure on \(G\). They are \(G\)\nb-equivariant for the diagonal action of \(G\) on \(\Cont_c(G,A)\) and \(\Cont_b(G,A)\), respectively. Furthermore they are adjoint to each other with respect to the pairings \(\braket{a}{b}=a^*b\) for \(a,b\in A\) and \(\braket{f}{g}=\int_G f(x)^*g(x)\diff x\) for \(f\in\Cont_b(G,A)\) and \(g\in\Cont_c(G,A)\). Recall that \(\Cont_c(G,A)\) can be completed into the right Hilbert \(A\)-module \(L^2(G,A)\).
\begin{definition}[\cite{meyer2001}]
	An element \(a\in A\) is \emph{square-integrable} if the operator \(\KET{a}\) extends to an adjointable operator \(\KET{a}\colon L^2(G,A)\to A\). 
\end{definition}
If \(a\in A\) is square-integrable, \(\BRA{a}\) can be understood as an adjointable operator \(A\to L^2(G,A) \) with adjoint \(\KET{a}\) as explained in \cite{meyer2001}*{Sec.~4}.
On the subspace of square-integrable elements \(A_\si\subseteq A\) one can define a norm \(\norm{\,\cdot\,}_{\si}\), which turns it into a Banach space:
\[\norm{a}_\si = \norm{a}+\norm{\BRA{a}\circ\KET{a}}^{1/2}\]
\begin{definition}[\cite{meyer2001}*{6.4}]
	A \emph{continuously square-integrable \(G\)-\(\Cst\)-algebra} \((A,\Rel)\) is a \(\Cst\)\nb-algebra \(A\) with a strongly continuous \(G\)-action together with a subset \(\Rel\subset A_\si\) which is
	\begin{enumerate}
		\item \emph{relatively continuous}, that is, \(\BRAKET{a}{b}\defeq \BRA{a}\circ\KET{b}\in\Cred(G,A)\subset \Bound(L^2(G,A))\) for all \(a,b\in\Rel\). 
		\item \emph{complete}, that is, \(\Rel\) is a closed subspace of \((A_\si,\norm{\,\cdot\,}_\si)\) and \(\KET{a}(\Cont_c(G,A))\subset \Rel\) for all \(a\in\Rel\).
		\item dense in \(A\).
	\end{enumerate}	
\end{definition}
We remark that not every \(G\)-\(\Cst\)-algebra has such a subset \(\Rel\) and that it does not have to be unique, see \cite{meyer2001}*{Sec.~8}.

\begin{definition}\label{def:fix}
Let \((A,\Rel)\) be a continuously square-integrable \(G\)-\(\Cst\)-algebra. The \emph{generalized fixed point algebra} \(\Fix^G(A,\Rel)\) is defined as the closed linear span of \(\KET{a}\BRA{b}\) for \(a,b\in\Rel\) inside the \(G\)-invariant multiplier algebra \(\Mult^G(A)\).
\end{definition} 
It is in fact a \(\Cst\)-algebra and \(\KET{\Rel}\subset \Bound(L^2(G,A),A)\) can be completed into an imprimitivity bimodule between \(\Fix^G(A,\Rel)\) and the ideal generated by \(\BRAKET{\Rel}{\Rel}\) inside \(\Cred(G,A)\) \cite{meyer2001}*{6.10}.
\begin{example}
	Let \(G\acts X\) be a proper group action on a locally compact Hausdorff space \(X\). Then the orbit space \(G\backslash X\) is a locally compact Hausdorff space. The group action induces an action on the \(\Cst\)-algebra \(\Cont_0(X)\) by \((g\cdot f)(x)\defeq f(g^{-1}\cdot x)\) for \(g\in G\), \(x\in X\). Denote by \(\cl{\Cont_c(X)}\) the closure of \(\Cont_c(X)\subseteq \Cont_0(X)_{\si}\) with respect to \(\norm{\,\cdot\,}_\si\). Then \((\Cont_0(X),\cl{\Cont_c(X)})\) is a continuously square-integrable \(G\)-\(\Cst\)-algebra. The generalized fixed point algebra \(\Fix^G(\Cont_0(X),\cl{\Cont_c(X)})\) is isomorphic to the \(\Cst\)-algebra of the orbit space \(\Cont_0(G\backslash X)\).
\end{example}
\begin{definition}[\cite{rieffel1988}]
	A \emph{continuously square-integrable \(G\)-\(\Cst\)-algebra} \((A,\Rel)\) is called \emph{saturated}, if the ideal generated by \(\BRAKET{\Rel}{\Rel}\subseteq \Cred(G,A)\) is \(\Cred(G,A)\). 
\end{definition}
\begin{example}[see \cite{rieffelapplications}, \cite{ewert2020pseudodifferential}*{2.18}]
	Let \(G\) act properly on a locally compact Hausdorff space \(X\). Then \((\Cont_0(X),\cl{\Cont_c(X)})\) is saturated if and only if the action \(G\acts X\) is free.
\end{example}

\subsection{The zoom action}
The zoom action \(\alpha\) on the tangent groupoid from \cref{def:zoomgroupoid} induces an action on the introduced convolution algebras. 
\begin{lemma}\label{res:zoom_on_tangent}
	The	maps \(\sigma_\lambda \colon \Cont_c(\T_HM) \to \Cont_c(\T_HM)\) defined by
	\begin{equation*}
		(\sigma_\lambda f)(\gamma) = \lambda^{d_H} f(\alpha_\lambda(\gamma)) \quad \text{for }\lambda > 0\text{ and }f \in \Cont_c(\T_HM)
	\end{equation*}
	extend to a strongly continuous \(\Rp\)-action on \(\Cst(\T_HM)\). Moreover, \(\A(\T_HM)\) is invariant under the action. 
\end{lemma}
\begin{proof}
	Note that the Haar system \(\{\nu^{(x,t)}\}_{(x,t)\in M\times[0,\infty)}\) satisfies 
	\begin{align}\label{eq:haar_zoom} \int \sigma_\lambda f \diff\nu^{(x,t)}=\int f \diff \nu^{(x,\lambda^{-1} t)} \quad \text{for all }f\in\Cont_c(\T_HM).
	\end{align}
	Using this one can show \(\sigma_\lambda(f*g)=\sigma_\lambda(f)*\sigma_\lambda(g)\) for \(f,g\in\Cont_c(\T_HM)\). Furthermore, all \(\sigma_\lambda\) are linear and satisfy \(\sigma_\lambda(f^*)=(\sigma_\lambda(f))^*\) for all \(f\in\Cont_c(\T_HM)\). As each \(\sigma_\lambda\) is an isometry with respect to the \(I\)-norm, it follows that \(\sigma\) extends to a strongly continuous action on \(\Cst(\T_HM)\). 
	
	To see that \(\A(\T_HM)\) is invariant, note that it suffices to show this for \(f\in\A(\U)\), as \(\Cont_c^\infty(M\times M\times\Rp)\) is invariant. Because \(k(\beta_\lambda(x,v,t))=(x,t\cdot v)=k(x,v,t)\) for all \(\lambda>0\), one can take the compact set \(K_f\) for \(f\) for all \(\sigma_\lambda(f)\). Furthermore, for fixed \(\lambda>0\) the function \(\sigma_\lambda(f)\) has compact support in the \(t\)-direction and satisfies the rapid decay condition. This follows from the homogeneity of the quasi-norm.
\end{proof}
\begin{lemma}\label{res:kernelev}
	The ideal \(\ker(p_0)\idealin\Cst(\T_HM)\) is invariant under the zoom action. Under the isomorphism from \eqref{eq:iso_p} \[p\colon \ker(p_0)\to\Cont_0(\Rp)\otimes\Comp(L^2M)\] the zoom action corresponds to the action \(\tau\otimes 1\), where \(\tau\) is induced by the free and proper scaling action of \(\Rp\) on itself, namely,
	\[(\tau_\lambda f)(t)=f(\lambda^{-1}t) \quad\text{for }f\in\Cont_0(\Rp) \text{ and }\lambda,t>0.\]		
\end{lemma}
\begin{proof}The homomorphism \(p_0\) is equivariant for the zoom action and the action on \(\Cst(T_HM)\) induced by the dilations in \cref{def:dilation_osculating}. The second claim follows from the computation that \(p\circ\sigma_\lambda=(\tau_\lambda\otimes 1)\circ p\) for all \(\lambda>0\).
\end{proof}
\subsection{Generalized fixed point algebras for filtered manifolds}
In this section, we show that the generalized fixed point algebra construction can be applied to certain ideals in the groupoid \(\Cst\)-algebras associated with the tangent groupoid of filtered manifolds. 
\begin{definition}
	Let \(\J_0\) be the ideal in \(\Cst(T_HM)\) defined as \[\J_0 = \bigcap_{x\in M}{\ker(\widehat{\pi}_\triv\circ q_x)}.\]
	Here, \(\widehat{\pi}_\triv\colon\Cst(G_x)\to \C\) denotes the representation induced by the trivial representation of the osculating group \(G_x\), that is,
	\[\widehat{\pi}_\triv(f)=\int_{G_x} f(x)\diff x\quad\text{for }f\in\Cont_c(G_x)\]
\end{definition}
The ideal \(\J_0\) can be extended to an ideal \(\J\) in \(\Cst(\T_HM)\) as follows. 
\begin{definition}
	Let \(\J\) denote the ideal in \(\Cst(\T_HM)\) given by
	\[ \J = \bigcap_{x\in M}{\ker(\widehat{\pi}_\triv\circ q_x\circ p_0)}.\]
\end{definition}
Both ideals \(\J\) and \(\J_0\) are invariant under the zoom action \(\sigma\) of \(\Rp\).
To apply the generalized fixed point algebra construction to this \(\Rp\)-action on \(\J\), consider the following \(^*\)-subalgebra of \(\J\).
\begin{definition}
	Let \(\Rel\subset \J\) consist of all \(f\in\A(\T_HM)\) such that
	\begin{align}\label{eq:vanishing_integral}
	\int f\diff\nu^{(x,0)}=0 \quad \text{for all }x\in M.
	\end{align}
\end{definition}

We show first the following lemma, which will replace an application of the mean value theorem in the Euclidean case. 
\begin{lemma}\label{res:meanvalue}
	Let \(g\in\A(\U)\) and \(K_1,K_2\subset k(\U)\) be compact subsets. For all \(a\in\N\) there are~\(D>0\) and \(b\in\N\) such that for all \((\gamma^{-1},\eta)\in\U^{(2)}\) with \(k(\gamma)\in K_1, k(\eta)\in K_2 \)
	\[\abs{g(\gamma^{-1}\eta)-g(\gamma^{-1})}\leq D\frac{(1+\norm{\eta})^b}{(1+\norm{\gamma})^{a}}\sum_{j=1}^n\norm{\eta}^{q_j}.\]
\end{lemma}
\begin{proof}
	It suffices to show the claim for real-valued  \(g\in\A(\U)\). Let \(\gamma^{-1}=(x,v,t)\) and \(\eta=(\varepsilon_x^{-1}(t\cdot v),w,t)\).  Define the function \(G\colon[0,1]\times\U^{(2)}\to \R\) by
	\[G(x,v,t,w,h)=g((x,v,t)\cdot(\varepsilon^{-1}_x(t\cdot v),h\cdot w,t)).\]
	Hence, we obtain
	\begin{align*}
	g((x,v,t)(\varepsilon^{-1}_x(t\cdot v), w,t))-g(x,v,t)=\int_0^1 \partial_hG(x,v,t,w,s)\diff s.
	\end{align*}
	To estimate \(\abs{\partial_hG}\) note that \(G=g\circ m\circ(\Id\times\delta)\), where \(\delta(w,h)=h\cdot w\). Writing \(\eta_h=(\varepsilon_x^{-1}(t\cdot v),h\cdot w,t)\), one calculates that
	\[\partial_hG(\gamma^{-1},\eta,s)=\sum_{i,j=1}^n \partial_{v_i}g(\gamma^{-1}\cdot\eta_h)\cdot \partial_{w_j}m_i(\gamma^{-1},\eta_h)\cdot \partial_h\delta_j(w,s).\]
	By the structure of \(\U\) we can find a compact subset \(K_2\subset \tilde{K}\subset k(\U)\) that is star-shaped in the sense that for \((x,v)\in \tilde{K}\) also \((x,h\cdot v)\in\tilde{K}\) holds for all \(h\in[0,1]\). 
	Let \(T>0\) be such that \(g\) vanishes for \(t\geq T\). 
	\cref{res:invmultest} applied to \(i(K_1)\cup\tilde{K}\) and \(T\) gives \(C>0\) and \(l\in\N\) with 
	\[ \abs{\partial_{w_j}m_i(\gamma^{-1},\eta_h)}\leq C(1+\norm{\gamma^{-1}})^{l}(1+\norm{\eta_h})^{l}\leq C(1+\norm{\gamma})^{l^2}(1+\norm{\eta})^{l}\]
	for all \(\gamma,\eta\) with \(k(\gamma)\in K_1\), \(k(\eta)\in K_2\) and \(t<T\). For these \(\gamma,\eta\) use the rapid decay condition for~\(g\) to estimate using \cref{res:schwartztriangle}
	\begin{align*}
	\abs{\partial_{v_i}g(\gamma^{-1}\cdot \eta_h)}&\lesssim (1+\norm{\gamma^{-1}\cdot\eta_h})^{-l^2(a+l^2)}
	\lesssim (1+\norm{\eta_h^{-1}\cdot\gamma})^{-l(a+l^2)}\\
	&\lesssim \frac{(1+\norm{\eta_h})^{l(a+l^2)}}{(1+\norm{\gamma})^{a+l^2}}\leq \frac{(1+\norm{\eta})^{l(a+l^2)}}{(1+\norm{\gamma})^{a+l^2}}.\end{align*}
	As \(\delta_j(w,h)=h^{q_j}w_j\), it follows that \(\abs{\partial_h\delta_j(w,s)}\lesssim \abs{w_j}\leq \norm{\eta}^{q_j}\). Together, these estimates imply the claim. 
\end{proof}
\begin{lemma}\label{res:mainestimate}
	Let \((M,H)\) be a filtered manifold. Consider the restricted zoom action \(\sigma\colon \Rp\acts \J\). For \(f\in\Rel\) the operator \(\BRA{f}\) as in \eqref{BRAKET} satisfies 
	\(\BRA{f}g\in L^1(\Rp,\J)\) for all \(g\in \Rel\). 
\end{lemma}
\begin{proof}
	We show that \((\lambda\mapsto \norm{\sigma_\lambda(f^*)*g}_I)\in L^1(\Rp,\tfrac{\diff\lambda}{\lambda})\) holds for all \(f,g\in\Rel\). 
	Because  \(\sigma_\lambda\) for~\(\lambda>0\) is an isometry with respect to the \(I\)\nb-norm
	\begin{equation*}
	\norm{\sigma_{\lambda^{-1}}(f^*)*g}_I = \norm{f^**\sigma_{\lambda}(g)}_I=\norm{\sigma_{\lambda}(g^*)*f}_I
	\end{equation*}
	holds for all \(f,g\in\Rel\). Therefore, and as \(\Rel\) is invariant under involution, it will suffice to show 
	\begin{equation}\label{eq:rgreater1} 
	\int_1^\infty \norm{\sigma_\lambda(f)*g}_I\tfrac{\diff \lambda}{\lambda}<\infty \quad \text{for all }f,g\in\Rel. 
	\end{equation}			
	We decompose  \(f\) as in \cref{res:decompschwartz} and write
	\[\sigma_\lambda(f)*g=\sum_{j\in \{1,\ldots,m,\infty\}} \sigma_\lambda(f_j)*g\] 
	with \(f_j\in\Rel_{\V_j}\defeq \Rel\cap\A(\V_j)\) and \(f_\infty\in\Cont^\infty_c(\V_\infty)\).
	We proceed to decompose this further as in~\eqref{eq:decomposeconvolution}. Let \(K_{f_j},K_g\) be compact subsets for \(f_j,g\) and let \(T>0\) be such that \(f_j\) and \(g\) vanish for \(t\geq T\).
	As before, let \(\{\chi,1-\chi\}\) be a compactly supported partition of unity subordinate to the open cover \(V_j\times V_j\times[0,\infty)\) and \(\V_\infty\) of \(m(K_{f_j},K_g)\times[0,T]\). 	
	As noted in the proof of \cref{res:zoom_on_tangent} we can take the same compact subset \(K_{f_j}\subset k(\U)\) for all \(\sigma_\lambda(f_j)\). For \(\lambda\geq 1\), let
	\(\chi_\lambda\) be the scaled version of \(\chi\) defined by \(\chi_\lambda(x,y,t)=\chi(x,y,\lambda^{-1}t)\). As \(\lambda\geq 1\), \(\{\chi_\lambda,1-\chi_\lambda\}\) is still a partition of unity for the cover of
	\[m(K_{f_j},K_g)\times[0,T]=m(K_{\sigma_\lambda(f_j)},K_g)\times[0,T].\]
	As before, we have a decomposition \[\sigma_\lambda(f_j)*g=(\sigma_\lambda(f_j)*g)_{\chi_\lambda}+(\sigma_\lambda(f_j)*g)_{(1-\chi_\lambda)}.\]
	Inspecting the argument for the first summand in the proof of \cref{res:schwartzalgebra}, we can find a \(g_j\in\Rel_{\V_j}\) that does not depend on \(\lambda\) such that  \[(\sigma_\lambda(f_j)*g)_{\chi_\lambda}=(\sigma_\lambda(f_j)*g_j)_{\chi_\lambda}.\]
	For the second summand, recall the construction of \(\omega\in\Cont^\infty([0,\infty))\) in \cref{res:schwartzalgebra}. It follows as in \eqref{eq:omega} that for \(\omega_\lambda(t)\defeq \omega(\lambda^{-1}t)\)
	\begin{align*}
	(\sigma_\lambda(f_j)*g)_{(1-\chi_\lambda)}&= \left(\sigma_\lambda(f_j)\cdot(\omega_\lambda\circ\theta)*g\right)_{(1-\chi_\lambda)}\\
	&=\left(\sigma_\lambda(f_j\cdot(\omega\circ\theta))*g\right)_{(1-\chi_\lambda)}.
	\end{align*}
	Note that \(f_j\cdot(\omega\circ\theta)\in\Cont_c^\infty(\V_\infty)\).
	As \(\abs{\chi_\lambda},\abs{1-\chi_\lambda}\leq 1\) for all \(\lambda\geq 1\), it suffices to prove \eqref{eq:rgreater1} for the two cases \(f\in\Cont^\infty_c(\V_\infty)\) and \(g\in\Rel\), and \(f,g\in\Rel_\U\). Here, \(\Rel_\U\) denotes all functions \(f\in\A(\U)\) with \(f\circ\phi_\kappa\in\Rel_\V\) for an \(H\)-chart \(\kappa\colon V\to U\). 
	
	Suppose first  that \(f\in\Cont^\infty_c(\V_\infty)\) and \(g\in\Rel\). Let \(T,t_0>0\) be such that \(f(x,y,t)=g(x,y,t)=0\) whenever \(t>T\) and \(f(x,y,t)=0\) for \(t<t_0\). For \(x\in M\) and \(t>0\) we have
	\[\left(\sigma_\lambda(f)*g\right)(x,y,t)=\lambda^{d_H}t^{-d_H}\int_M f(x,z,\lambda^{-1}t)g(z,y,t)\diff \nu(z),\]	
	which is only non-zero for \(t\leq T\). Moreover, it vanishes if \(\lambda>Tt_0^{-1}\) because then \(\lambda^{-1}t<t_0\) holds for  \(t\leq T\), so that \(f(x,z,\lambda^{-1}t)=0\). As \(\lambda\geq 1\), only \(t\geq t_0\) have to be considered. As \(g\) restricted to \(t\geq t_0\) is a compactly supported function, we can find a compact subset \(K\subset M\) such that \(f(x,y,t)=0\) if \((x,y)\notin K \times K\) and \(g(x,y,t)=0\) if \(t\geq t_0\) and \((x,y)\notin K\times K\).
	Moreover, there is a constant \(C>0\) such that \(\abs{g}\leq C\) for \(t\geq t_0\) and \(\abs{f}\leq C\). 
	We obtain
	\[\int_1^\infty\norm{\sigma_\lambda(f)*g}_{I}\tfrac{\diff\lambda}{\lambda}\leq T^{2d_H}	C^2\nu(K)^2\int_1^{T/t_0}\lambda^{d_H-1}\diff\lambda<\infty.\] 		
	Consider now the case that \(f,g\in\Rel_\U\). To shorten notation write \(b(x)=\abs{\det B_X(x)}\) for \(x\in U\). For \(\gamma\in\U\) with \(r(\gamma)=(x,t)\), one has
	\begin{align*}
	\left(\sigma_\lambda(f)*g^*\right)(\gamma)&= \lambda^{d_H}b(x)\int_{\U^{(x,t)}} f(x,\lambda\cdot z,\lambda^{-1}t)\,g^*((x,z,t)^{-1}\gamma)\diff z\\
	&=b(x)\int_{\U^{(x,t)}} f(x,z,\lambda^{-1}t)\,g^*((x,\lambda^{-1}\cdot z,t)^{-1}\gamma)\diff z.
	\end{align*}
	Let \(K_f,K_g\) be compact subsets for \(f\) and \(g\) as in \cref{def:schwartz_local}. As \(K_f\) is a compact subset for all \(\sigma_\lambda(f)\), we only need to consider \(\gamma\in K\defeq m(K_f,K_g)\) by \cref{res:schwartzalgebra}. Moreover, we only need to consider \(z\in\U^{(x,t)}\) with \((x,(\lambda^{-1}t)\cdot z)\in K_f\). 
	Define two functions \(R_1,R_2\) by
	\begin{align*}
	R_1(\gamma,\eta)\defeq & g^*(\eta^{-1}\cdot\gamma)-g^*(\gamma),\\
	R_2(x,z,t)\defeq & f(x,z,t)-f(x,z,0).
	\end{align*}
	We write 
	\begin{align*}
	\left(\sigma_\lambda(f)*g^*\right)(\gamma) = &b(x)\left( g^*(\gamma)\int f(x,z,0)\diff z +\int f(x,z,0)R_1(\gamma,(x,\lambda^{-1}\cdot z,t))\diff z\right.\\
	&\left.+ \int R_2(x,z,\lambda^{-1}t)g^*((x,\lambda^{-1}\cdot z ,t)\cdot\gamma)\diff z\right).
	\end{align*}
	The first term vanishes as \(f\in\Rel_\U\). 
	The mean value theorem in \cref{res:meanvalue} applied to \(i(K)\), \(K_f\), \(g\) and \(a=d_H+1\) yields \(D>0\) and \(b\in\N\) such that \begin{align*}
	\abs{R_1(\gamma,(x,\lambda^{-1}\cdot z,t))}&\leq  D\,\frac{(1+\norm{\lambda^{-1}\cdot z})^b}{(1+\norm{\gamma})^{d_H+1}}\sum_{j=1}^{n}\norm{\lambda^{-1}\cdot z}^{w_i}\\
	&= D\,\frac{(1+\lambda^{-1}\norm{ z})^b}{(1+\norm{\gamma})^{d_H+1}}\sum_{j=1}^{n}\lambda^{-w_i}\norm{z}^{w_i}\\
	&\leq D\lambda^{-1}\, \frac{(1+\norm{ z})^{b+d_H}}{(1+\norm{\gamma})^{d_H+1}}
	\end{align*}
	for \(\gamma\in K\) and \(z\in \U^{(x,t)}\) with \((x,(\lambda^{-1}t)\cdot z)\in K_f\).
	For the last inequality we used that \(\lambda\geq 1\).		
	The usual mean value theorem and the rapid decay of \(f\) allow to find \(C>0\) such that
	\[\abs{R_2(x,z,\lambda^{-1}t)}\leq \lambda^{-1}t C(1+\norm{z})^{-(l^2+1)(d_H+1)}. \]
	As \(f\) has rapid decay, one can estimate
	\begin{align*}
	\abs{f(x,z,0)}&\lesssim (1+\norm{z})^{-b-2d_H-1}.		
	\end{align*}
	As \(g^*\) is rapidly decaying, as well, and using \cref{res:schwartztriangle} and \(\lambda\geq 1\), we find
	\begin{align*}
	\abs{g^*((x,\lambda^{-1}z,t)\cdot \gamma)}&\lesssim (1+\norm{(x,\lambda^{-1}z,t)\cdot \gamma})^{-l(d_H+1)}\\
	&\lesssim \frac{(1+\norm{(x,\lambda^{-1}z,t)^{-1}})^{l(d_H+1)}}{(1+\norm{\gamma})^{d_H+1}} \\ &\lesssim	\frac{(1+\norm{(x,\lambda^{-1}z,t)})^{l^2(d_H+1)}}{(1+\norm{\gamma})^{d_H+1}} 	\leq 
	\frac{(1+\norm{z})^{l^2(d_H+1)}}{(1+\norm{\gamma})^{d_H+1}}.
	\end{align*}
	Therefore, we obtain for all \((x,t)\in M\times[0,\infty)\)
	\begin{align*}\int\abs{\sigma_\lambda(f)*g^*}\diff\nu^{(x,t)}\lesssim b(x)\lambda^{-1}(1+t).
	\end{align*}
	As there is a \(T>0\) such that \(g^*\) vanishes for \(t\geq T\) and \(f\) is compactly supported in \(x\) this implies \(\norm{\sigma_\lambda(f)*g^*}_{I,r}\lesssim \lambda^{-1}\).
	For \(\norm{\,\cdot\,}_{I,s}\), replace \(\gamma\) by \(\gamma^{-1}\) in the estimates above and use  \cref{res:schwartztriangle} to derive similar estimates. 	
	The convergence of \(\int_1^\infty \lambda^{-2}\diff\lambda\)
	finishes the proof of \eqref{eq:rgreater1}. 
\end{proof}
\begin{theorem}
	For a filtered manifold \((M,H)\) the \(^*\)-subalgebra \(\Rel\subset \J\) is square-integrable with respect to the zoom action of \(\Rp\). 	
	Denote by \(\cl{\Rel}\) its closure with respect to the \(\norm{\,\cdot\,}_\si\)-norm.
	Then \((\J,\cl{\Rel})\) is a continuously square-integrable \(\Rp\)-\(\Cst\)-algebra. 
\end{theorem}
\begin{proof}
	The Schwartz type algebra \(\A(\T_HM)\) is a \(^*\)-subalgebra of \(\Cst(\T_HM)\) by \cref{res:schwartzalgebra}. Condition \eqref{eq:vanishing_integral}, which is that the Haar integrals vanish at \(t=0\), is preserved when taking the involution or convolution of functions in \(\Rel\). Therefore, \(\Rel\) is a \(^*\)-subalgebra of \(\J\). Moreover, it is invariant under the zoom action of \(\Rp\), as \(\A(\T_HM)\) is invariant by \cref{res:zoom_on_tangent} and the \(d_H\)-homogeneity of the Haar system at \(t=0\), which follows from \eqref{eq:haar_zoom}. 
	
	To see that \(\Rel\) is dense in \(\J\), let \(f\in \J\) and \(\epsilon>0\). There is a \(g\in\Cont^\infty_c(\T_HM)\) with \(\norm{f-g}<\epsilon/2\). To adjust \(g\) to have vanishing integrals at \(t=0\), define the function \(h\in\Cont_c^\infty(M)\) by
	\[ h(x)= \int g \diff\nu^{(x,0)}\]
	for \(x\in M\). It satisfies \(\abs{h(x)}=\abs{\widehat{\pi}_{\triv}(q_x(p_0(g)))-\widehat{\pi}_{\triv}(q_x(p_0(f))) }\leq \norm{f-g}<\varepsilon/2\) for all \(x\in M\). Choose a function \(k\in\Cont^\infty_c(\T_HM)\) such that \(\int k \diff\nu^{(x,0)}=1\) for all \(x\in r(p_0(\supp g))\) and \(\norm{k}_I\leq 1\). This can be done by defining such a function locally on the charts \(\V_i\) and pasting them together with a smooth partition of unity. Let \(\tilde{g}\defeq (h\circ r_1)\cdot k\).
	It is a smooth, compactly supported function on \(\T_HM\), and \(\norm{\tilde{g}}\leq \epsilon/2\). As \(g-\tilde{g}\in\Rel\) and \(\norm{f-(g-\tilde{g})}\leq\epsilon\), this finishes the proof that \(\Rel\) is dense in \(\J\).

	Now, the estimate in \cref{res:mainestimate} together with \cite{ewert2020pseudodifferential}*{2.11, 2.12} imply that \((\J,\cl{\Rel})\) is a continuously square-integrable \(\Rp\)-\(\Cst\)-algebra. 
\end{proof}
Hence, the generalized fixed point algebra \(\Fix^\Rp(\J,\cl{\Rel})\) is defined as in \cref{def:fix}. It is the closed linear span of \(\KET{f}\BRA{g}\) for \(f,g\in\Rel\) by~\cite{ewert2020pseudodifferential}*{2.11}. By \cite{meyer2001}*{(19)} these can be described as strict limits in the multiplier algebra of \(\J\). Use a net \((\chi_i)_{i\in I}\) consisting of smooth, compactly supported functions \(\chi_i\colon \Rp\to [0,1]\) that converge uniformly on compact subsets to \(1\) to cut off at zero and infinity. Assume that \(\chi_i(\lambda)=\chi_i(\lambda^{-1})\) for all \(i\in I\) and \(\lambda>0\). Then \(\KET{f}\BRA{g}\) is given by the following strict limit \begin{align}\label{eq:strict_limit}\KET{f}\BRA{g}=\lim_{i,s}\int_0^\infty \chi_i(\lambda)\sigma_\lambda(f^**g)\tfrac{\diff \lambda}{\lambda}.\end{align}

\subsection{The pseudodifferential extension}
Recall that \(\J_0=p_0(\J)\). The \(\Cst\)-algebra extension for the tangent groupoid in~\eqref{ses:tangentgroupoid} restricts to the short exact sequence of \(\Rp\)-\(\Cst\)-algebras
\begin{equation}\label{ses:ideals}
\begin{tikzcd}
\Cont_0(\Rp)\otimes\Comp(L^2M) \arrow[r,hook] & \J\arrow[r,twoheadrightarrow,"p_0"] & \J_0.
\end{tikzcd}
\end{equation}
There is a corresponding short exact sequence of generalized fixed point algebras by \cite{ewert2020pseudodifferential}*{2.19}.  Using the subsets \(\tilde{\Rel}\defeq\cl{\Rel}\cap\Cont_0(\Rp)\otimes\Comp(L^2M)\) and \(\Rel_0\defeq p_0(\Rel)\) one obtains continuously square-integrable \(\Rp\)-\(\Cst\)-algebras \((\Cont_0(\Rp)\otimes\Comp(L^2M),\tilde{\Rel})\) and \((\J_0,\cl{\Rel_0})\) by \cite{ewert2020pseudodifferential}*{2.14, 2.15}. The same arguments as in the proof of \cite{ewert2020pseudodifferential}*{5.11} yield the following description of the corresponding extension of generalized fixed point algebras. 
\begin{proposition}
	Let \(\Rel_0\defeq p_0(\Rel)\). 
	The zoom action of \(\Rp\) on the extension in~\eqref{ses:ideals} gives rise to a short exact sequence of generalized fixed point algebras
	\begin{equation}\label{ses:fixedalg}
	\begin{tikzcd}
	\Comp(L^2M) \arrow[r,hook] & \Fix^\Rp(\J,\cl{\Rel})\arrow[r,twoheadrightarrow,"\widetilde{p}_0"] & \Fix^\Rp(\J_0,\cl{\Rel_0}).
	\end{tikzcd}
	\end{equation} 	
	Here, \(\widetilde{p_0}\) denotes the restriction of the strictly continuous extension of \(p_0\) to the corresponding multiplier algebras.
\end{proposition}
We call this extension the order zero pseudodifferential extension in the following. The symbol algebra \(\Fix^\Rp(\J_0,\cl{\Rel_0})\) is a continuous field of \(\Cst\)-algebras. 
\begin{proposition}\label{res:fix_0_isfield}
	The generalized fixed point algebra \(\Fix^\Rp(\J_0,\cl{\Rel_0})\) is a continuous field of \(\Cst\)-algebras over \(M\) with fibre projections
	\[\widetilde{q}_x\colon \Fix^\Rp(\J_0,\cl{\Rel_0})\to \Fix^\Rp(J_x,\cl{\Rel_x}).\]
	Here \(J_x\defeq\ker(\widehat{\pi}_\triv)\idealin\Cst(G_x)\) and \(\Rel_x\) consists of all \(f\in\Schwartz(G_x)\) with vanishing integral with respect to the Haar measure on \(G_x\).
\end{proposition}
\begin{proof}This follows from \cite{rieffel1988}*{3.2}, see also \cite{ewert2020pseudodifferential}*{Remark~2.24}.\end{proof}

The order zero operators in \(\Fix^\Rp(\J,\cl{\Rel})\) have a faithful representation as bounded operators on~\(L^2(M)\). The \(^*\)-homomorphisms \(p_t\colon \Cst(\T_HM)\to \Comp(L^2M)\) defined in \eqref{eq:p_t} for \(t>0\) can be restricted to the ideal \(\J\). The restrictions are still surjective.  Therefore, they yield strictly continuous representations
\[\widetilde{p}_t\colon \Fix^\Rp(\J,\cl{\Rel})\to \Mult(\Comp(L^2M))=\Bound(L^2M) \quad \text{for all }t>0.\]
\begin{lemma}
	The representation \(\widetilde{p}_1\colon \Fix^\Rp(\J,\cl{\Rel})\to\Bound(L^2M)\) is faithful. 
\end{lemma}
\begin{proof}The result holds by the same reasoning as in the proof of \cite{ewert2020pseudodifferential}*{5.13}.
\end{proof}
\begin{lemma}\label{res:fix_on_l2}
	Let \((\chi_i)_{i\in I}\) be a net of \(\chi_i\in\Cont^\infty_c(\Rp)\) that converge uniformly on compact subsets to~\(1\) and satisfy \(\chi_i(\lambda^{-1})=\chi_i(\lambda)\) for all \(\lambda>0\). Let \(f,g\in\Rel\) and \(h=f^**g\). Then the operators~\(T_i(h)\) given by
	\begin{align}\label{def:T_i}
	T_i(h)\psi(x)=\int_0^\infty \chi_i(\lambda)\lambda^{-d_H}\int h(x,y,\lambda)\psi(y)\diff\nu(y) \tfrac{\diff\lambda}{\lambda}\end{align} for \(\psi\in L^2(M)\), \(x\in M\),
	converge strictly to \(\widetilde{p}_1(\KET{f}\BRA{g})\) as multipliers of \(\Comp(L^2M)\). 
\end{lemma}

\begin{proof} 		
	Use the description of \(\KET{f}\BRA{g}\) as a strict limit as in \eqref{eq:strict_limit}. As \(\widetilde{p}_1\) is strictly continuous and  \(p_t\circ\sigma_\lambda = p_{t\lambda^{-1}}\) for all \(t,\lambda>0\) we get 
	\begin{align*}
	\widetilde{p}_1\left(\KET{f}\BRA{g}\right)=\lim_{i,s} \int_0^\infty \chi_i(\lambda )p_1(\sigma_\lambda(f^**g))\tfrac{\diff\lambda}{\lambda}=\lim_{i,s} \int_0^\infty \chi_i(\lambda)p_\lambda(f^**g)\tfrac{\diff\lambda}{\lambda}.
	\end{align*}
	The operators \(T_i(h)\) above are obtained by inserting the definition of \(p_\lambda\) in \eqref{eq:p_t}. 
\end{proof}
The same argument as in \cite{ewert2020pseudodifferential}*{5.15} shows:
\begin{lemma}\label{res:compboundcommute}
	The following diagram commutes
	\begin{equation*}
	\begin{tikzcd}
	\Comp(L^2M) \arrow[r,hook] \arrow[dr, hook]& \Fix^\Rp(\J,\cl{\Rel})\arrow[d, "\widetilde{p}_1"]	\\
	 & \Bound(L^2M).
	\end{tikzcd}
	\end{equation*}	
\end{lemma}
\begin{lemma}\label{res:compact_fact}
	Let \(h\in \Rel\cap\ker(p_0)\) and let \(T_i(h)\) be defined as in \eqref{def:T_i}. Then \((T_i(h))\) converges in norm in \(\Comp(L^2M)\). In particular, its strict limit as multipliers of \(\Comp(L^2M)\) exists and is contained in \(\widetilde{p}_1(\Fix^\Rp(\J,\cl{\Rel}))\). 
\end{lemma}
\begin{proof}
	As \(h\in\Rel\subset\A(\T_HM)\) vanishes for \(t=0\), it can be written as \(h=tf\) with \(f\in\A(\T_HM)\). By definition of the representation \(p_\lambda\) in \eqref{eq:p_t}, it follows that \(p_\lambda(h)=\lambda p_\lambda(f)\) for all \(\lambda>0\). Hence, for all \(\lambda>0\) \begin{align*}
	\norm{p_\lambda(h)}\leq \lambda\norm{p_\lambda(f)}\leq \lambda\norm{f}.
	\end{align*}
	We show that \((T_i(h))\) is Cauchy. Let \(T>0\) be such that \(h\) vanishes for \(t\geq T\). For \(j\geq i\), we estimate
	\begin{align*}
	\norm{T_j(h)-T_i(h)}\leq \int_0^\infty (\chi_j(\lambda)-\chi_i(\lambda))\norm{p_\lambda(h)}\tfrac{\diff\lambda}{\lambda}
	\leq \norm{f} \int_0^T (1-\chi_i(\lambda))\diff\lambda.
	\end{align*}
	As \(\chi_i\to 1\) on compact subsets, the claim follows. As \(\Comp(L^2M)\) is complete, it follows that \((T_i(h))\) converges in norm. The second claim follows as convergence in norm implies strict convergence and \(\Comp(L^2M)\) is contained in \(\widetilde{p}_1(\Fix^\Rp(\J,\cl{\Rel}))\) by \cref{res:compboundcommute}.
\end{proof}

\section{The principal symbol algebra}\label{sec:principal_symbol}
In this section we examine the principal symbol algebra \(\Fix^\Rp(\J_0,\cl{\Rel_0})\). It is a continuous field of \(\Cst\)-algebras over \(M\) with fibres \(\Fix^\Rp(J_x,\cl{\Rel_x})\) by \cref{res:fix_0_isfield}. Here, \(J_x\) is the kernel of the trivial representation \(\widehat{\pi}_\triv\colon \Cst(G_x)\to\C\). It was shown in \cite{ewert2020pseudodifferential}*{6.11} that \(\Fix^\Rp(J_x,\cl{\Rel_x})\) is the \(\Cst\)-closure of the operators of type zero on \(G_x\). We state now a bundle version of this result. We use tempered fibred distributions on \(T_HM\) as in \cite{erp2015groupoid}*{7.1}):
\begin{definition}For a smooth vector bundle \(\pi\colon E\to M\), a \emph{tempered fibred distribution with compact support in the \(M\)-direction} is a continuous \(\Cont^\infty_c(M)\)-linear map \(u\colon \Schwartz(E)\to \Cont^\infty_c(M)\). Denote by	 \(\Schwartz'_{\cp}(E)\) the linear space of tempered fibred distributions.
\end{definition}
For \(u\in\Schwartz'_{\cp}(E)\) and each \(x\in M\) there is a tempered distribution \(u_x\in \Schwartz'(E_x)\) such that \(\langle u, f\rangle (x)=\langle u_x, f_x\rangle\) for all \(f\in\Schwartz(E)\). 
For \(E=T_HM\), there is a well-defined convolution \[*\colon \Schwartz_{\cp}'(T_HM)\times \Schwartz(T_HM)\to \Schwartz_{\cp}'(T_HM),\] which restricts in the fibres to the convolution on the osculating groups. To define homogeneity of fibred distributions, recall that the dilations yield an \(\Rp\)-action on \(\Schwartz(T_HM)\) given by
\[(\sigma_\lambda f)(x,\xi)=\lambda^{d_H}f(x,\delta_\lambda(\xi)) \quad\text{for }\lambda>0\text{, }f\in\Schwartz(T_HM) \text{ and }\xi\in G_x.\]
This action can be extended to \(\Schwartz_{\cp}'(T_HM)\) by
\[\langle \sigma_{\lambda*} u, f\rangle \defeq \lambda^{d_H}\langle u, \sigma_{\lambda^{-1}} f\rangle\]
for \(u\in\Schwartz'_{\cp}(T_HM), f\in\Schwartz(T_HM)\) and \(\lambda >0\). It allows to 
extend the notion of kernels and operators of type \(\nu\) on graded Lie groups (see for example \cite{fischer2016quantization}*{3.2.9}) to the bundle of osculating groups. 
\begin{definition}\label{def:kerneltype}
	Let \(\nu\in\R\). A fibred distribution \(u\in\Schwartz_{\cp}'(T_HM)\) is called a \emph{kernel of type \(\nu\)} if it is smooth away from the zero section and \(\sigma_{\lambda*}(u)=\lambda^\nu u\) for all \(\lambda>0\). Denote by \(\mathcal{K}^\nu(T_HM)\) the space of kernels of type \(\nu\). 
	The corresponding continuous operator \(T_u\colon\Schwartz(T_HM)\to\Schwartz_{\cp}'(T_HM)\) given by \(T_u(f)=u* f\) is called an \emph{operator of type \(\nu\)}.		
\end{definition}
In particular, \(T_u\) restricts at \(x\in M\) to the right-invariant continuous linear operator \(\Schwartz(G_x)\to\Schwartz'(G_x)\) given by \(f\mapsto u_x*f\) for \(f\in\Schwartz(G_x)\). 
Moreover, one calculates as in \cite{fischer2016quantization}*{3.2.7} that an operator \(T\) of type \(\nu\) satisfies
\[ T(\sigma_{\lambda^{-1}}f)=\lambda^\nu\sigma_{\lambda^{-1}}(Tf) \quad \text{for all }\lambda>0 \text{ and }f\in\Schwartz(T_HM).\]
The left regular representations \(\lambda_x\colon \Cst(G_x)\to\Bound(L^2G_x)\) for \(x\in M\) allow to understand the elements \(\KET{f}\BRA{g}\) with \(f,g\in\Rel_0\) as a smooth family of bounded, right-invariant operators \((\KET{f_x}\BRA{g_x})\) on~\(L^2(G_x)\). 
By a bundle version of the arguments in \cite{ewert2020pseudodifferential}*{6.6-6.11} (see also \cite{ewert2020index}*{Ch.~8} for details) we obtain the following result.
\begin{theorem}\label{res:type_zero}
	Let \((M,H)\) be a filtered manifold. The linear span of \(\KET{f}\BRA{g}\) for \(f,g\in\Rel_0\) is \(\mathcal{K}^0(T_HM)\). Consequently, \(\Fix^\Rp(\J_0,\cl{\Rel_0})\) is the \(\Cst\)-completion of the operators of type zero. 
\end{theorem}
\section{Comparison to the calculus by van Erp and Yuncken}\label{sec:comparison}
The pseudodifferential calculus for filtered manifolds by van Erp and Yuncken is also based on the tangent groupoid and the zoom action. Note that they use a version of the tangent groupoid which is a field over \(\R\) and not \([0,\infty)\). In this section, we outline of their
construction and compare it to the generalized fixed point algebra approach. 
\subsection{The calculus of van Erp and Yuncken}
Recall the theory of fibred distributions on Lie groupoids and their convolutions as in \cite{erp2015groupoid}*{Section~2}, see also \cite{lescure_distr_groupoids}.
\begin{definition}[\cite{erp2015groupoid}*{6}]
	Let \(\grpd\) be a Lie groupoid with unit space \(\grpd^{(0)}\) and range and source \(r,s\colon \grpd\to\grpd^{(0)}\). An \emph{\(r\)-fibred distribution}, respectively an \emph{\(s\)\nb-fibred distribution}, is a continuous \(\Cont^\infty(\grpd^{(0)})\)-linear map
	\(u\colon \Cont^\infty(\grpd)\to\Cont^\infty(\grpd^{(0)})\),
	where the \(\Cont^\infty(\grpd^{(0)})\)-module structure on \(\Cont^\infty(\grpd)\) is induced by the range, respectively the source map. The spaces of \(r\)- and \(s\)-fibred distributions are denoted by \(\Smooth'_r(\grpd)\) and \(\Smooth'_s(\grpd)\). 
\end{definition}
There are well-defined convolution products
\(*\colon \Smooth'_\pi(\grpd)\times \Smooth'_\pi(\grpd)\to \Smooth'_\pi(\grpd)\) for \(\pi=r,s\) that turn \(\Smooth'_\pi(\grpd)\) into an associative algebra \cite{lescure_distr_groupoids}*{20}. 

\begin{definition}
	A subset \(X\subset\grpd\) is \emph{proper} if the restricted range and source maps
	\begin{align*}
	r|_X\colon X\to \grpd^{(0)} \quad\text{ and }\quad s|_X\colon X\to \grpd^{(0)}
	\end{align*} are proper. 
	For \(\pi=r,s\), let \(\Omega_\pi\) be the bundle of smooth densities tangent to the range fibres, respectively source fibres, of \(\grpd\). Let \(\Cont^\infty_{\p}(\grpd;\Omega_\pi)\) denote the space of smooth sections \(f\) such that \(\supp(f)\) is proper.
\end{definition}
Then \(\Cont^\infty_{\p}(\grpd;\Omega_r)\) is a right ideal in \(\Smooth'_r(\grpd)\), whereas \(\Cont^\infty_{\p}(\grpd;\Omega_s)\) is a left ideal in \(\Smooth'_s(\grpd)\) (see \cite{erp2015groupoid}*{9} and \cite{lescure_distr_groupoids}*{21}).  

The zoom action from \cref{def:zoomgroupoid} induces an \(\Rp\)-action \(\alpha_*\) on \(\Smooth'_r(\T_HM)\) by automorphisms. Each \(\alpha_{\lambda*}\) restricted to  \(\Cont^\infty_{c}(\T_HM;\Omega_r)\) coincides with \(\sigma_{\lambda^{-1}}\) as in \cref{res:zoom_on_tangent} when identifying \(\Cont^\infty_c(\T_HM;\Omega_r)\) with \(\Cont^\infty_c(\T_HM)\) using the left invariant Haar system. 

\begin{definition}[\cite{erp2015groupoid}*{18, 19}]
	A properly supported \(\P\in\Smooth'_r(\T_HM)\) is \emph{essentially homogeneous of weight \(m\in\R\)} if
	\[\alpha_{\lambda*}(\P)-\lambda^{m}\P\in \Cont^\infty_{\p}(\T_HM;\Omega_r) \quad \text{for all }\lambda>0.\] The space of these distributions is denoted by \(\Pseu^m_H(M)\).
	
	A distribution \(P\in\Smooth_r'(M\times M)\) is an \emph{\(H\)-pseudodifferential kernel of order \(\leq m\)} if there is a~\(\P\in\Pseu^m_H(M)\) that restricts to \(P\) at \(t=1\). Denote by \(\Psi^m_H(M)\) the space of these kernels. For~\(P\in\Psi^m_H(M)\) define the following corresponding \emph{\(H\)-pseudodifferential operator} by
	\[\Op(P)\colon\Cont^\infty(M)\to\Cont^\infty(M), \quad \Op(P)f(x)=\langle P_x,f\rangle.\]	
\end{definition}
Moreover, they define the principal cosymbol of \(P\in\Psi^m_H(M)\) by extending it to a \(\P\in\Pseu^m_H(M)\) and restricting \(\P\) to \(t=0\). To make this independent of the choice of the extension, the space of cosymbols is defined as follows. 
\begin{definition}[\cite{erp2015groupoid}*{34, 35}]
	For \(m\in\R\), let \(\Ess^m_H(M)\) consist of all properly supported \(u\in\Smooth'_{r}(T_HM)\) such that 
	\[\delta_{\lambda*}(u)-\lambda^{m}u \in \Cont^\infty_{\p}(T_HM;\Omega_r) \quad\text{for all }\lambda>0.\]
	Here, \(\delta_*\) is induced by the dilations in the fibres of \(T_HM\). Define
	\[\Sigma^m_{H}(M)\defeq \Ess^m_H(M) / \Cont^\infty_{\p}(T_HM;\Omega_r).\]
	The \emph{principal cosymbol} of \(P\in\Psi^m_H(M)\) is defined by extending \(P\) to \(\P\in\Pseu^m_H(M)\) and setting \[\princ[m](P)\defeq[\P_0]\in\Sigma^m_H(M).\] 
\end{definition}
Van Erp and Yuncken show in \cite{erp2015groupoid}*{47} that the wave front set of \(\P\in\Pseu^m_H(M)\) is contained in the conormal to \(M\times[0,\infty)\). This implies by \cite{erp2015groupoid}*{48} that \(\P\) belongs, in fact, to the space of proper \(r\)-fibred distributions \(\Smooth'_{r,s}(\T_HM)\) as defined in \cite{erp2015groupoid}*{11}. As \(\Cont^\infty_{\p}(\T_HM;\Omega_r)\) is a two-sided ideal in \(\Smooth'_{r,s}(\T_HM)\), it is then easy to see that \(\bigcup_{m\in\Z}\Psi^m_H(M)\) is a \(\Z\)-graded algebra, see \cite{erp2015groupoid}*{49}. Moreover, there is a well-defined involution on \(\Smooth'_{r,s}(\T_HM)\) by \cite{lescure_distr_groupoids}*{20}. We summarize now the main properties of the calculus.
\begin{proposition}[\cite{erp2015groupoid}*{36, 37, 38, 49, 50, 52, 53}, \cite{davehaller}*{3.6}]
	The pseudodifferential calculus on a filtered manifold \((M,H)\) satisfies:
	\begin{enumerate}
		\item For \(m\in\Z\) there is a short exact sequence
		\begin{equation}\label{ses:filtered_order_m}
		\begin{tikzcd}
		\Psi^{m-1}_H(M) \arrow[r,hook] & \Psi^m_H(M)\arrow[r,twoheadrightarrow,"{\princ[m]}"] & \Sigma^m_H(M).
		\end{tikzcd}
		\end{equation}
		The inclusion is well-defined, which can be seen by considering the map \(\Pseu^{m-1}_H(M)\to\Pseu^{m}_H(M)\) with \(\P\mapsto t\P\) for \(\P\in\Pseu^{m-1}_H(M)\), which does not change the kernel at \(t=1\).
		\item For \(P\in\Psi^m_H(M)\) and \(Q\in\Psi^l_H(M)\) with \(m,l\in\R\), the convolution is in \(\Psi^{m+l}_H(M)\) and \(\princ[m+l](P*Q)=\princ[m](P)*\princ[l](Q)\).
		\item For \(P\in\Psi^m_H(M)\) and \(m\in\R\), the formal adjoint \(P^*\) is in \(\Psi^m_H(M)\) and \(\princ[m](P^*)=\princ[m](P)^*\).	
		\item For \(k\geq 0\) and \(r\) the step of the filtration, the following regularities hold 
		\begin{align*}
		\Psi^{-d_H-kr-1}_H(M)&\subset \Cont_{\p}^k(M\times M;\Omega_r),\\ \bigcap_{m\in\Z}\Psi^{m}_H(M)&= \Cont^\infty_{\p}(M\times M;\Omega_r).
		\end{align*}
	\end{enumerate}		
\end{proposition}
\begin{remark}
	Dave and Haller extend in \cite{davehaller} the pseudodifferential calculus to operators acting between sections of vector bundles \(E,F\) over \(M\). In this case, distributions in \[\Smooth'_r(\T_HM;\hom(s^*(\E),r^*(\F))\otimes\Omega_r)\] are used, where \(\E= E\times\R\) and \(\F = F\times\R\) are the vector bundles over the unit space \(M\times\R\).
\end{remark}

We will use \emph{global exponential coordinates} as in \cite{erp2017tangent}*{15}. They identify an open subset \(\widetilde{\W}\) of \(\lie{t}_HM\times\R\), the Lie algebroid of \(\T_HM\), with an open, zoom-invariant neighbourhood  \(\W\) of \[T_HM\times{0}\cup \Delta_M\times(0,\infty)\subset \T_HM.\] Here, \(\Delta_M\) denotes the diagonal in \(M\times M\). 		

Suppose \(\varphi_1\in\Cont^\infty(M\times M)\) is constant \(1\) on \(\Delta_M\) and vanishes outside \(\W_1\). Then \(\varphi\in\Cont^\infty(\T_HM)\) defined by \(\varphi\defeq \varphi_1\circ (r_1,s_1)\) is called an \emph{exponential cutoff} in \cite{erp2015groupoid}*{27}. It is invariant under the zoom action. It is shown in \cite{erp2015groupoid}*{27} that each \(\P\in\Pseu^m_H(M)\) differs from a \(\Q\in\Pseu^m_H(M)\) supported on \(\W\) by an element of \(\Cont^\infty_{p}(\T_HM;\Omega_r)\), namely, set \(\Q = \varphi\P\) for an exponential cutoff \(\varphi\).

As we used the Schwartz type algebra in the construction of the generalized fixed point algebra \(\Fix^\Rp(\J,\cl{\Rel})\), the corresponding operators on \(L^2(M)\) as in \cref{res:fix_on_l2} have compactly supported kernels. Therefore, we compare it to the following variant of van Erp and Yuncken's calculus with compact instead of proper supports. 

\begin{definition}
	Let \(\Pseu^m_{H,c}(M)\) consist of all \(\P\in\Pseu^m_{H}(M)\) such that there is a compact subset \(K\subset M\times M\) such that all \(\P_t\) for \(t\neq0\) are supported in \(K\). Likewise, \(\Ess^m_{H,c}(M)\) denotes all \(u\in\Ess^m_{H}(M)\) with compact support in the \(x\)-direction and	 \[\Sigma^m_{H,c}(M)\defeq \Ess^m_{H,c}(M) / \Cont^\infty_{c}(T_HM;\Omega_r).\]
\end{definition}
\begin{lemma}
	Let \(m\in\Z\). The pseudodifferential extension of order \(m\) in \eqref{ses:filtered_order_m} restricts to a short exact sequence 
	\begin{equation}\label{ses:pseudo_comp_supp}
	\begin{tikzcd}
	\Psi^{m-1}_{H,c}(M) \arrow[r,hook]& \Psi^m_{H,c}(M)\arrow[r,twoheadrightarrow,"{\princ[m]}"] & \Sigma_{H,c}^m(M)	
	\end{tikzcd}
	\end{equation}
\end{lemma}
\begin{proof}
	This can be shown analogously to \cite{erp2015groupoid}*{36, 37, 38}. For surjectivity of the principal cosymbol map, extend a compactly supported \(u\in\Ess^m_{H,c}(M)\) as in  \cite{erp2015groupoid}*{36} to the constant distribution \(\tilde{\distru}\in\Smooth'_r(\lie{t}_HM\times\R)\) given by \(\tilde{\distru}_t=u\).  Using an exponential cutoff \(\varphi\)  for a \(\varphi_1\in\Cont^\infty_c(M\times M)\) that is constant \(1\) on a neighbourhood of \(\Delta_{r(\supp u)}\subset M\times M\) and zero outside of~\(\W_1\), one obtains an \(\distru\in\Pseu^m_{H,c}(M)\) that extends \(u\). 
\end{proof}

\subsection{Principal cosymbols as generalized fixed points}
First, we compare the space of principal cosymbols of order \(0\) to the generalized fixed point algebra at \(t=0\), namely, \(\Fix^\Rp(\J_0,\cl{\Rel_0})\).  Consider the following bundle version of the space of approximately homogeneous distributions defined in~\cite{taylornoncommutative}*{2}.
\begin{definition}
	A distribution \(u\in\Schwartz_{\cp}'(T_HM)\) is \emph{approximately \(0\)-homogeneous} if
	\begin{enumerate}
		\item \(u\) is smooth outside the zero section,
		\item \(u=u_1+u_2\) for some \(u_1\in\Schwartz_{\cp}(T_HM)\) and \(u_2\in\Smooth'_{\cp}(T_HM)\),
		\item \(\sigma_\lambda(u)-u\in\Schwartz_{\cp}(T_HM)\) for all \(\lambda>0\).
	\end{enumerate}
	Denote by \(\mathfrak{H}^0_{\cp}(T_HM)\) the space of approximately \(0\)-homogeneous distributions. 
\end{definition}
By \cite{taylornoncommutative}*{2.2, 2.4} they are closely related to kernels of type \(0\) from \cref{def:kerneltype}. 
\begin{proposition}\label{res:taylor=kernels}
	There is a surjective linear map \(\Phi\colon \mathfrak{H}^0_{\cp}(T_HM)\to \mathcal{K}^0(T_HM)\) with  \(\ker(\Phi)=\mathfrak{H}^0_{\cp}(T_HM)\cap\Cont^\infty_{\cp}(T_HM)\). 
\end{proposition}
\begin{proof}
	By the bundle version of \cite{taylornoncommutative}*{2.4} every \(u\in\mathfrak{H}^0_{\cp}(T_HM)\) can be written as \(u = f + w\) with \(f\in\Cont^\infty_{\cp}(T_HM)\) and \(w\in \mathcal{K}^0(T_HM)\). In fact, it is shown in \cite{taylornoncommutative}*{2.2}, that \(\widehat{u}\in\Cont^\infty(\lie{t}^*_HM)\) admits a radial limit \(\tilde{w}(x,\xi)\defeq\lim_{\lambda\to\infty}u(x,\lambda \xi)\), which defines a smooth, \(0\)-homogeneous function~\(\tilde{w}\). Therefore, \(\tilde{w}\) extends to a tempered fibred distribution and \(w\) is its inverse under Fourier transform. The above decomposition of \(u\) is unique as kernels of type \(0\) coincide with smooth \((-d_H)\)-homogeneous functions outside the zero section. Therefore, \(\Cont^\infty_{\cp}(T_HM)\cap \mathcal{K}^0(T_HM)=\{0\}\) and \(u\mapsto w\) is a well-defined linear map. 
	
	For surjectivity, let \(w\in\mathcal{K}^0(T_HM)\) be a kernel of type \(0\). The Euclidean Fourier transform \(\widehat{w}\in\Schwartz'(\lie{t}^*_HM)\) is smooth and \(0\)\nb-homogeneous outside the zero section \(M\times\{0\}\). Take a smooth cutoff function \(\chi\in\Cont^\infty(\lie{t}^*_HM)\) which vanishes near the zero section and is constant \(1\) outside an \(r\)-compact set. Then \(u=\Four^{-1}(\chi\widehat{w})\) is an approximately \(0\)-homogeneous distribution by \cite{taylornoncommutative}*{2.2} and one can write
	\[u = w - \Four^{-1}((1-\chi)\widehat{w}).\]
	The latter is smooth as \((1-\chi)\widehat{w}\in\Smooth_{\cp}'(\lie{t}_HM)\), so that \(\Phi(u)=w\). 		
	The claim concerning the kernel follows from the uniqueness of the decomposition above and the definition of \(\Phi\). 
\end{proof}
Note that \(\mathfrak{H}^0_{\cp}(T_HM)\) is larger than the space \(\Ess^0_{H,c}(M)\) we consider. However, the following result holds. See \cite{davehaller}*{3.8} for a more general result for \(\Sigma^m_{H}(M)\).
\begin{lemma}\label{res:sigma=kernels}There is a linear bijection \(\Theta\colon \Sigma^0_{H,c}(M)\to\kernel^0(T_HM)\).
\end{lemma}
\begin{proof}
	The map \(\Phi\) in \cref{res:taylor=kernels} induces a bijective, linear map
	\[\mathfrak{H}^0_{\cp}(T_HM)/\ker(\Phi)\to \mathcal{K}^0(T_HM).\]
	We show that there is a linear bijection \(\Sigma^0_{H,c}(M)\to \mathfrak{H}^0_{\cp}(T_HM)/\ker(\Phi)\).  The inclusion of \(\mathrm{Ess}^0_{H,c}(M)\) into \(\mathfrak{H}^0_{\cp}(T_HM)\) induces a linear map to the quotient
	\[\psi\colon \mathrm{Ess}^0_{H,c}(M) \to \mathfrak{H}^0_{\cp}(T_HM)/\ker(\Phi).\]
	To see that it is surjective, let \(u\in \mathfrak{H}^0_{\cp}(T_HM)\) and choose a function \(\omega\in\Cont^\infty_c(T_HM)\) which is constant \(1\) in a neighbourhood of \(r(\supp u)\times\{0\}\) and consider \(\omega u\in\Smooth'_{\cp}(T_HM)\). It is essentially \(0\)-homogeneous as for all \(\lambda>0\)
	\begin{align*}\sigma_\lambda(\omega u)-\omega u&=-\sigma_\lambda((1-\omega) u)+\sigma_\lambda(u)-u+(1-\omega)u\end{align*}
	is smooth. It is also compactly supported as the left hand side is. 
	As \(u-\omega u=(1-\omega)u\) is smooth, it follows that \(\psi(\omega  u)=[u]\). 
	
	Finally, we show that \(\ker(\psi)=\Cont^\infty_c(T_HM)\). If \(u\in \mathrm{Ess}^0_{H,c}(M)\) is contained in the kernel of \(\psi\), it is smooth by the description of the kernel of \(\Phi\). It is also compactly supported. So it must lie in \(\Cont^\infty_{c}(T_HM)\). As \(\Cont^\infty_c(T_HM)\subset\Cont^\infty_{\cp}(T_HM) \) the converse inclusion holds, too. 
\end{proof}
By the description of \(\kernel^0(T_HM)\) in \cref{res:type_zero}, we obtain a linear map \[\Theta\colon \Sigma^0_{H,c}(M)\to\Fix^\Rp(\J_0,\cl{\Rel})\]
with dense image. However, it is not clear from the proof that it is a \(^*\)-homomorphism for the convolution and involution of distributions. This will be shown later.

\subsection{Pseudodifferential operators  as generalized fixed points}
It was observed in \cite{debordskandalis2014}*{Theorem 3.7} that classical pseudodifferential operators on a manifold \(M\) can be written as averages over the zoom action of functions \(f\in\A(\T M)\) with \(f_0\in\Schwartz_0(TM)\). We will show this for filtered manifolds for the order zero case.
\begin{lemma}\label{res:fix_as_distr}
	Let \(g\in\Cont^\infty_c(\lie{t}^*_HM\times\R)\) vanish at all points \((x,0,0)\) for \(x\in M\). Then
	\(\int_0^\infty \chi_i(\lambda)g(x,\lambda\cdot\xi,\lambda )\tfrac{\diff \lambda}{\lambda}\)
	converges in \(\Schwartz_{\cp}'(\lie{t}^*_HM)\). 
\end{lemma}
\begin{proof}
	We show that for each \((x,\xi)\in\lie{t}^*_HM\) 
	\begin{align*}
	K(x,\xi)\defeq\lim_i \int_0^\infty \chi_i(\lambda)g(x,\lambda\cdot\xi,\lambda)\tfrac{\diff \lambda}{\lambda}
	\end{align*}
	converges and satisfies for all \(\alpha\in\N^n_0\) 
	\[\sup_{(x,\xi)}\abs{\partial^\alpha_x K(x,\xi)}<\infty.\]
	As \(g\) vanishes at all points \((x,0,0)\), it can be written as 
	\[g(x,\xi,t)=\sum_{j=1}^n \xi_jg_j(x,\xi,t)+t g_t(x,\xi,t)\] with \(g_j,g_t\in \Cont_c^\infty(\lie{t}^*_HM\times \R)\) for \(j=1,\ldots,n\). Since they are compactly supported, there are \(R,T>0\) such that \(g_j(x,\xi,t)=g_t(x,\xi,t)=0\) for \(j=1,\ldots,n\) whenever \(\norm{\xi}>R\) or \(t>T\). It follows that \(\abs{K(x,\xi)}\leq T(R+1)\max_{j=1,\ldots,n,t}\norm{g_j}_\infty\). This shows the claim for \(\alpha=0\). Derivatives in the \(x\)-direction give an expression of the same form so it holds for all~\(\alpha\in\N^n_0\).
\end{proof}
\begin{proposition}\label{res:esshom}
	Let \(\Theta\) be the map from \cref{res:sigma=kernels}.
	For \(\P\in\Pseu^0_{H,c}(M)\) there is \(f\in\Rel\) with \(f_0\in\Schwartz_0(T_HM)\) such that 
	\begin{align*}
	\P_1-\int_0^\infty p_1(\sigma_\lambda(f))\tfrac{\diff\lambda}{\lambda}\in\Cont^\infty_{c}(M\times M) \quad\text{ and }\quad
	\Theta([\P_0])=\int_0^\infty\sigma_\lambda(f_0)\tfrac{\diff\lambda}{\lambda}.
	\end{align*}		
\end{proposition}

\begin{proof}
	By \cite{erp2015groupoid}*{Lemma~27 and Lemma~42} \(\P\) differs by a kernel in \(\Cont^\infty_{\p}(\T_HM)\) from an element \(\Q\in\Pseu^0_{H}(M)\) which is supported in the global exponential coordinate patch and is homogeneous on the nose outside \([-1,1]\). Using that \(\P\) has a uniform compact support \(K\subset M\times M\) for \(t\neq0\), one can arrange that the same holds for the kernel in \(\Cont^\infty_{\p}(\T_HM)\) and that \(\Q\in\Pseu^0_{H,c}(M)\). 
	So assume, without loss of generality, that \(\P\) has is supported in the global exponential coordinate patch and is homogeneous on the nose outside \([-1,1]\).
	
	Let \(\tilde{\P}\in\Smooth'_r(\lie{t}_H M\times\R)\) be the pullback under the exponential map and \(\widehat{\P}=\Four(\tilde{\P})\in\Cont^\infty(\lie{t}^*_HM\times\R)\). Here, \(\Four\) is the fibrewise Euclidean Fourier transform \(\Four\colon \Schwartz'_r(\lie{t}_HM\times \R)\to\Schwartz'_r(\lie{t}^*_HM\times \R )\). As remarked in \cite{erp2015groupoid}*{7.3}, \(\widehat{\P}\) is approximately homogeneous for the \(\Rp\)-action \(\beta_\lambda(x,\xi,t)=(x,\lambda\cdot \xi,\lambda t)\) on \(\lie{t}^*_HM\times\R\). This is a dilation action when considering \(\lie{t}^*_HM\times\R\overset{\pi}{\to}M\) as a graded vector bundle over~\(M\) with the original dilations on \(\lie{t}^*_HM\) and the usual scaling on \(\R\) as observed in \cite{erp2015groupoid}*{7.3}.
	
	Now we use \cite{erp2015groupoid}*{Prop. 43}, which is based on the bundle version of \cite{taylornoncommutative}*{2.2}. It allows to find an \(A\in\Cont_{\cp}^\infty(\lie{t}_H^*M\times\R\setminus M\times 0\times 0)\) homogeneous of degree \(0\) such that \(\widehat{\P}-\chi A\in\Schwartz_{\cp,\pi}(\lie{t}_H^*M\times\R)\) for any smooth cutoff function \(\chi\) that vanishes in a neighbourhood of the zero section \(M\times 0\times 0\) and is constant \(1\) outside a \(\pi\)-compact set. 		
	Now choose \(g\in\Cont^\infty_c(\lie{t}^*_HM\times\R)\) vanishing with all derivatives at all points \((x,0,0)\) such that for all \((\xi,t)\neq (0,0)\)
	\[ A(x,\xi,t)=\int_0^\infty  g(x,\lambda\xi,\lambda t)\tfrac{\diff\lambda}{\lambda}.\]
	Then \(\int_0^\infty g(x,\lambda\xi,\lambda t )\tfrac{\diff \lambda}{\lambda}\) converges in \(\Schwartz'_{\cp}(\lie{t}_H^*M)\)  for \(t=0\) by \cite{folland1982homogeneous}*{1.65} and for \(t>0\) by \cref{res:fix_as_distr}. We write for all \(t\geq 0\) and a cutoff function \(\chi\) as above 
	\[\widehat{\P}_t-\int_0^\infty  g(x,\lambda\xi, \lambda t)\tfrac{\diff \lambda}{\lambda} = \left(\widehat{\P}_t-\chi(\xi,t)A_t\right)-(1-\chi(\xi, t))\int_0^\infty g(x,\lambda\xi,\lambda t)\tfrac{\diff \lambda}{\lambda}.\]	
	The first part lies in \(\Schwartz_{\cp}(\lie{t}^*_HM)\) and the second in \(\Smooth'_r(\lie{t}_H^*M)\) with compact support in the \(x\)-direction. So that we obtain under inverse Fourier transform 
	\begin{align*}
	\tilde{\P}_t&-\int_0^\infty (\tilde{\sigma}_\lambda\circ\Four^{-1}(g))(t)\tfrac{\diff\lambda}{\lambda}\in\Cont_{\cp}^\infty(\lie{t}_HM).
	\end{align*}
	Here, \(\tilde{\sigma}\) denotes the zoom action on \(\lie{t}_HM\times\R\).
	
	As \(\P\) has uniform compact support, there is an exponential cutoff \(\varphi\) built from a compactly supported \(\varphi_1\in\Cont_c^\infty(M\times M)\) such that \(\varphi \P = \P\). Let \(f\defeq\varphi \Four^{-1}(g)\), which is an element of \(\Rel\) with \(f_0\in\Schwartz_0(T_HM)\).
	We obtain for all \(t\geq 0\), using that \(\varphi\) is invariant under the zoom action, 
	\begin{align*}
	\P_t-\int_0^\infty p_t(\sigma_\lambda(f))\tfrac{\diff\lambda}{\lambda}&=\varphi_t\left(\P_t-\int_0^\infty  p_t\circ\sigma_\lambda\circ\Four^{-1}(g)\tfrac{\diff\lambda}{\lambda}\right).
	\end{align*}
	For \(t>0\), this is contained in \(\Cont^\infty_c(M\times M)\). For \(t=0\), note that \(\varphi_0=1\) on the support of \(\P_0\) and \(g\) in the \(x\)-direction. As \(\P_0\in\Sigma^0_{H,c}(M)\subseteq \mathfrak{H}^0_{\cp}(T_HM) \), the claim follows from \cref{res:taylor=kernels}.
\end{proof}
\begin{lemma}\label{res:average=fixed}
	For \(f\in\Rel\) with \(f_0\in\Schwartz_0(T_HM)\)
	\[\int^\infty_0\chi_i(\lambda)\sigma_\lambda(f)\tfrac{\diff\lambda}{\lambda}\]
	converges strictly to an element  \(Q\in\Fix^\Rp(\J,\cl{\Rel})\). Moreover, 
	\[\widetilde{p}_0(Q)= \lim_{i} \int^\infty_0\chi_i(\lambda)\sigma_\lambda(f_0)\tfrac{\diff\lambda}{\lambda}\in\Fix^\Rp(\J_0,\cl{\Rel_0}).\]
\end{lemma}
\begin{proof}
	Using the Dixmier--Malliavin Factorization Theorem \cite{dixmier1978factorization}*{7.2} as in the proof of \cite{ewert2020pseudodifferential}*{6.7} one can factorize \(f_0\in\Schwartz_0(T_HM)\) as \(f_0=\sum_{j=1}^n f_j^**g_j\) with \(f_j,g_j\in\Rel_0\). Let \(F_j, G_j\in\Rel\) be extensions of \(f_j,g_j\) to \(t>0\), which can be obtained as in the proof of \cref{res:schwartzrestsurj}.
	As \(h\defeq f-\sum_{j=1}F_j^**G_j\in\A(\T_HM)\) vanishes at \(t=0\), it follows from \cref{res:compact_fact} that
	\[\int_0^\infty\chi_i(\lambda)\sigma_\lambda(h)\tfrac{\diff\lambda}{\lambda}\]
	converges strictly to an element of the generalized fixed point algebra. As	
	\[\int_0^\infty\chi_i(\lambda)\sigma_\lambda(f)\tfrac{\diff \lambda}{\lambda}= \int_0^\infty\chi_i(\lambda)\sigma_\lambda(h)\tfrac{\diff\lambda}{\lambda}+ \sum_{j=1}^n\int_0^\infty\chi_i(\lambda)\sigma_\lambda(F_j^**G_j)\tfrac{\diff\lambda}{\lambda}\] 
	and the operators on the right converge to \(\KET{F_j}\BRA{G_j}\), it follows that the left hand side converges strictly to an element \(Q\in\Fix^\Rp(\J,\cl{\Rel})\). As \(h\in\ker(p_0)\), the decomposition above also shows that
	\[\widetilde{p}_0(Q)=\sum_{j=1}^n\KET{f_j^*}\BRA{g_j}= \lim_{i} \int^\infty_0\chi_i(\lambda)\sigma_\lambda(f_0)\tfrac{\diff\lambda}{\lambda}.\qedhere\]
\end{proof}

\begin{theorem}\label{res:vanerp_in_fix}
	The pseudodifferential extension of order zero from \eqref{ses:pseudo_comp_supp} for a filtered manifold \((M,H)\) embeds into the generalized fixed point algebra extension from \eqref{ses:fixedalg} such that the following diagram commutes
	\begin{equation}\label{ses:pseudo_in_fix}
	\begin{tikzcd}
	\Psi^{-1}_{H,c}(M) \arrow[d,hook]\arrow[r,hook]& \Psi^0_{H,c}(M)\arrow[r,twoheadrightarrow,"{\princ[0]}"]\arrow[d,hook] & \Sigma_{H,c}^0(M)\arrow[d,hook, "\Theta"]\\
	\Comp(L^2M) \arrow[r,hook]& \widetilde{p}_1(\Fix^\Rp(\J,\cl{\Rel})) \arrow[r,twoheadrightarrow,"{\Princ}"] & \Fix^\Rp(\J_0,\cl{\Rel_0}),
	\end{tikzcd}
	\end{equation}
	with \(\Princ=\widetilde{p}_0\circ (\widetilde{p}_1)^{-1}\).
\end{theorem}
\begin{proof}
	For the inclusion of \(\Psi^0_{H,c}(M)\) in the generalized fixed point algebra, let \(P\in\Psi^0_{H,c}(M)\). By \cref{res:esshom}, there is a function \(f\in\Rel\) with \(f_0\in\Schwartz_0(T_HM)\) and
	\[P-\int_0^\infty p_1(\sigma_\lambda(f))\tfrac{\diff \lambda}{\lambda}\in\Cont^\infty_c(M\times M).\]
	By \cref{res:average=fixed},  \(\int_0^\infty\chi_i(\lambda)\sigma_\lambda(f)\tfrac{\diff \lambda}{\lambda}\) converges to an element \(Q\) of the generalized fixed point algebra \(\Fix^\Rp(\J,\cl{\Rel})\).   Its image \(\widetilde{p}_1(Q)\) in \(\Bound(L^2M)\) is the unique bounded extension of the convolution operator associated with \(\int_0^\infty p_1(\sigma_\lambda(f))\tfrac{\diff \lambda}{\lambda}\). Convolution by the kernel in \(\Cont_c^\infty(M\times M)\) extends to a compact operator \(K\). It follows that \(K+Q\) is the unique continuous extension of \(\Op(P)\) and belongs to the generalized fixed point algebra. 		
	Moreover, \cref{res:average=fixed} also implies 
	\[\Theta(\princ[0](P))= \int_0^\infty \sigma_\lambda(f_0)\tfrac{\diff\lambda}{\lambda}=\sum_{j=1}^n\KET{f_j}\BRA{g_j}=\Princ(Q)=\Princ(\Op(P)).\]
	Hence, the right square in \eqref{ses:pseudo_in_fix} commutes. 		
	For \(P\in\Psi^{-1}_{H,c}(M)\) the commutativity of the right square and exactness of the rows yields that \(\Op(P)\) extends to a compact operator on \(L^2(M)\).
\end{proof}
\begin{remark}
	In particular, we recover the result in \cite{davehaller}*{3.9} that \(H\)-pseudo\-differential operators of order \(0\) extend to bounded operators on \(L^2(M)\), while the ones of order \(-1\) define compact operators. 
\end{remark}
As the left and middle vertical arrows in \eqref{ses:pseudo_in_fix} are inclusions of \(^*\)-algebras, the quotient map \(\Theta\colon\Sigma^0_{H,c}\to\Fix^\Rp(\J_0,\cl{\Rel_0})\) must be a \(^*\)-homomorphism, too.  Therefore, \(\Fix^\Rp(\J_0,\cl{\Rel_0})\) is the \(\Cst\)-completion of \(\Sigma_{H,c}^0(M)\). 
\begin{corollary}
	Let \(\Cst(\Psi^0_{H,c}(M))\) denote the \(\Cst\)-closure of \(\Psi^0_{H,c}(M)\) inside \(\Bound(L^2M)\). Then \(\Cst(\Psi^0_{H,c}(M))\) is isomorphic to \(\Fix^\Rp(\J,\cl{\Rel})\). There is a short exact sequence of \(\Cst\)-algebras
	\begin{equation*}
	\begin{tikzcd}
	\Comp(L^2 M)\arrow[hook,r] & \Cst(\Psi^0_{H,c}(M))\arrow[r,twoheadrightarrow,"\Princ"] &\Cst(\Sigma^0_{H,c}(M)),
	\end{tikzcd}
	\end{equation*}
	such that \(\Princ\) extends the principal cosymbol map \(\princ[0]\colon \Psi^0_{H,c}(M)\to \Sigma^0_{H,c}(M)\).
\end{corollary}
\begin{proof}
	We show that \(\Cst(\Psi^0_{H,c}(M))=\widetilde{p}_1(\Fix^\Rp(\J_0,\cl{\Rel}))\). The \(\Cst\)-algebra of  \(H\)-pseudodifferential operators of order \(0\) is contained in \(\widetilde{p}_1(\Fix^\Rp(\J_0,\cl{\Rel}))\) by \cref{res:vanerp_in_fix}. 
	
	For the converse, note first that \(\Comp(L^2M)\subset \Cst(\Psi^0_{H,c}(M))\). This holds as \(\Psi^0_{H,c}(M)\) contains the kernels in \(\Cont_c^\infty(M\times M)\) and these generate the compact operators on \(L^2(M)\). 		
	Now, let \(f,g\in\Rel\). Let \(u\in\Sigma^0_{H,c}(M)\) be the inverse of \(\KET{f_0}\BRA{g_0}\in\kernel^0(T_HM)\) under the map \(\Theta\) in \cref{res:sigma=kernels}. Since the principal cosymbol map is surjective, there is a \(P\in\Psi^0_{H,c}(M)\) with \(\princ[0](P)=u\).  Then the operator \[\widetilde{p}_1(\KET{f}\BRA{g})= \widetilde{p}_1(\KET{f}\BRA{g})-\Op(P)+\Op(P)\] is contained in \(\Cst(\Psi^0_{H,c}(M))\). This is because \(\Op(P)\) is and \(\widetilde{p}_1(\KET{f}\BRA{g})-\Op(P)\in\Comp(L^2M)\) as the diagram in \eqref{ses:pseudo_in_fix} commutes. The \(\Cst\)-algebra \(\Fix^\Rp(\J,\cl{\Rel})\) is generated by \(\KET{f}\BRA{g}\) with \(f,g\in\Rel\). Thus, the result follows. 
\end{proof}
The convolution algebra \(\Ess^0_{H,c}(M)\) can be completed into a \(\Cst\)-algebra.	
Consider the \(^*\)-representations \(\lambda_x\)  for \(x\in M\) given by the corresponding convolution operators, that is,
\[\lambda_x(u)f=u_x*f \quad\text{for }u\in\Ess^0_{H,c}(M)\text{ and }f\in\Schwartz(G_x).\]
To see that these extend to bounded operators on \(L^2(G_x)\), recall that by \cref{res:taylor=kernels} there is a \(w\in\kernel^0(T_HM)\) such that \(u-w\in\Cont^\infty_{\cp}(T_HM)\). 	
Let \(\chi\in\Cont^\infty_c(T_HM)\) be such that \(\chi\cdot u=u\). Then one can write \(u=\chi(u-w)+\chi w\). The first part is in \(\Cont^\infty_c(T_HM)\), so convolution with it extends to a bounded operator. As \(w_x\) extends to a continuous operator \(L^2(G_x)\to L^2(G_x)\) by \cite{knappstein}, also convolution with \(\chi_x w_x\) extends to a bounded operator by \cite{christ}*{2.10}. In fact, using the compact support in the base space, the proof of \cite{christ}*{2.10} yields a \(C>0\) with
\[\norm{\lambda_x(u)}\leq\norm{\chi(u-w)}+C\norm{w} \quad\text{for all }x\in M.\]
Thus, one can take the completion of \(\Ess^0_{H,c}(M)\) with respect to the norm
\[\norm{u}\defeq\sup_{x\in M}\norm{\lambda_x(u)} \quad\text{for }u\in\Ess^0_{H,c}(M).\]
Denote the resulting \(\Cst\)-algebra by \(\Cst(\Ess^0_{H,c}(M))\).	
The homomorphism \[\Theta\colon \Ess^0_{H,c}(M)/\Cont^\infty_c(T_HM)=\Sigma^0_{H,c}(M)\to\Fix^\Rp(\J_0,\cl{\Rel})\]
has dense image. Therefore, one obtains an extension of \(\Cst\)-algebras
\begin{equation}\label{eq:diskbundle}
\begin{tikzcd}
\Cst(T_HM) \arrow[r,hook]& \Cst(\Ess^0_{H,c}(M))\arrow[r,twoheadrightarrow] & 
\Fix^\Rp(\J_0,\cl{\Rel})
\end{tikzcd}
\end{equation}	
\begin{remark}
	For a step \(1\) filtration, the short exact sequence above corresponds under Fourier transform to the disk bundle extension
	\begin{equation*}
	\begin{tikzcd}
	\Cont_0(T^*M) \arrow[r,hook]& \Cont_0(B^*M)\arrow[r,twoheadrightarrow] & \Cont_0(S^*M).
	\end{tikzcd}
	\end{equation*}	
\end{remark}

Denote by \(\Pseu^0_{H,c}(M)|_{[0,1]}\) the \(^*\)-algebra obtained by restricting to \([0,1]\). It can be completed into a \(\Cst\)-algebra using the left regular representations \(\Op_t\) for \(t>0\) in the following way.
Let \(\P\in\Pseu^0_{H,c}(M)|_{[0,1]}\). Using \cref{res:esshom} one can find an element \(T\in\Fix^\Rp(\J,\cl{\Rel})\) with associated distributions \(K=(K_t)\) such that \(\P-K\in\Cont^\infty(\T_HM|_{[0,1]})\). Furthermore, there is a cutoff function \(\chi\in\Cont^\infty_c(\T_HM)\) that is \(1\) in a neighbourhood of the unit space such that \(\chi\P=\P\).  Then \(\chi(\P-K)\in\Cont^\infty_c(\T_HM|_{[0,1]})\). Moreover, \(\chi_t K_t\) extends to bounded operator for each \(t>0\) as \(K_t\) does and \((1-\chi_t)K_t\) is supported away from the diagonal and is, therefore, smooth and compactly supported. For \(t=0\), this was discussed above. It follows that
\[\norm{\Op_t(\P)}\leq\norm{\chi(\P-K)}+\norm{\chi K}\quad\text{for all }0<t\leq 1.\] Thus, one can complete \(\Pseu^0_{H,c}(M)|_{[0,1]}\) with respect to
\[\norm{\P}\defeq\max\left\{\sup_{0<t\leq 1}{\norm{\Op_t(\P)}},\norm{\P_0}\right\} \quad\text{for }\P\in\Pseu^0_{H,c}(M)|_{[0,1]}.\]
We will denote the \(\Cst\)-completion by \(\Cst(\Pseu^0_{H,c}(M))\).	As elements of \(\Pseu^0_{H,c}(M)\) are continuous families of distributions, \(\Cst(\Pseu^0_{H,c}(M))\) is a continuous field of \(\Cst\)\nb-algebras over \([0,1]\). Note that \(\Cst(\T_HM|_{[0,1]})\) is an ideal in \(\Cst(\Pseu^0_{H,c}(M))\).
\begin{lemma}\label{res:kernelcompact}
	The kernel of the homomorphism
	\[\mathbb{S}^0_H\colon\Cst(\Pseu^0_{H,c}(M))\to \Fix^\Rp(\J_0,\cl{\Rel_0})\]
	induced by \(\P\mapsto[\P_0]\) is \(\Cst(\T_HM|_{[0,1]})\).
\end{lemma}
\begin{proof}
	Since \(\Cst(\T_HM|_{[0,1]})\) is generated by \(\Cont^\infty_c(\T_HM|_{[0,1]})\) which is contained in the kernel of \(\mathbb{S}^0_H\), the first inclusion follows. For the converse inclusion, note that the kernel is generated by \(\Pseu^{-1}_{H,c}(M)|_{[0,1]}\). Let \(\P\in\Pseu^{-1}_{H,c}(M)|_{[0,1]}\). Tracing back the construction in \cref{res:esshom} one can take \(f\in\Rel\) such that \(f\) vanishes at \(t=0\). It follows that the corresponding element \(T\in\Fix^\Rp(\J,\cl{\Rel})\) defines a compact operator. Therefore, \(\Op(\P_t)\in\Comp(L^2M)\) for all \(t>0\). By assumption \(\P_0\in\Cont^\infty_c(T_HM)\subset\Cst(T_HM)\) holds, so that \(\P\) belongs  pointwise to \(\Cst(\T_HM)\). As \(\Cst(\Pseu^0_{H,c}(M))\) is a continuous field, it follows \(\P\in\Cst(\T_HM)\) by \cite{dixmier}*{10.4.2}. 
\end{proof}

\section{Morita equivalence}\label{sec:morita}

In this section, we will show that \((\J,\cl{\Rel})\) and \((\J_0,\cl{\Rel_0})\) are saturated for the zoom action of~\(\Rp\). Therefore, for each filtered manifold \((M,H)\) the \(\Cst\)\nb-algebras of order zero pseudodifferential operators \(\Fix^\Rp(\J,\cl{\Rel})\) and principal cosymbols \(\Fix(\J_0,\cl{\Rel_0})\) are Morita--Rieffel equivalent to \(\Cred(\Rp,\J)\) and \(\Cred(\Rp,\J_0)\), respectively. For the Euclidean scalings this is a result of \cite{debordskandalis2014}. 

First, recall the following result. For a graded Lie group \(G\) denote \(J_G=\ker(\widehat{\pi}_\triv)\) and \(\Rel_G\) the space of all functions \(f\in\Schwartz(G)\) with \(\int_Gf(x)\diff x=0\).

\begin{proposition}[\cite{ewert2020pseudodifferential}*{8.3}]\label{res:pointwisesaturated}
	For a graded Lie group \(G\) the \(\Rp\)-\(\Cst\)-algebra  \((J_G,\cl{\Rel_G})\) is saturated for the dilation action of \(\Rp\). 		
	The generalized fixed point algebra \(\Fix^\Rp(J_G,\cl{\Rel_G})\) is Morita--Rieffel equivalent to \(\Cred(\Rp, J_G)\).
\end{proposition}

From this we deduce saturatedness for the respective ideals in the \(\Cst\)-algebras of the osculating groupoid and the tangent groupoid. 
\begin{proposition}\label{res:morita}
	Let \((M,H)\) be a filtered manifold.
	The \(\Cst\)\nb-algebra of order~\(0\) principal cosymbols \(\Fix^\Rp(\J_0,\cl{\Rel_0})\) is Morita--Rieffel equivalent to \(\Cred(\Rp,\J_0)\). The \(\Cst\)\nb-algebra  of order \(0\) pseudodifferential operators \(\Fix^\Rp(\J,\cl{\Rel})\) is Morita--Rieffel equivalent to \(\Cred(\Rp,\J)\). 
\end{proposition}

\begin{proof}
As all fibres \((J_{x},\cl{\Rel_x})\) are saturated by \cref{res:pointwisesaturated},   \((\J_0,\cl{\Rel_0})\) is saturated by \cite{ewert2020pseudodifferential}*{2.26}. Therefore, the generalized fixed point algebra construction gives the Morita--Rieffel equivalence between \(\Fix^\Rp(\J_0,\cl{\Rel_0})\) and \(\Cred(\Rp,\J_0)\).
	
	 We show that \((\Cont_0(\Rp)\otimes\Comp(L^2M),\cl{\Rel}\cap\Cont_0(\Rp)\otimes\Comp(L^2M)) \) is saturated. The second claim follows then from the result on saturatedness for short exact sequences in \cite{ewert2020pseudodifferential}*{2.17} applied to the sequence \eqref{ses:ideals}. By \cref{res:kernelev} the \(\Rp\)-action on \(\Cont_0(\Rp)\otimes\Comp(L^2M) \)is given by \(\tau\otimes 1\), where \(\tau\) is induced by the action of \(\Rp\) on itself by multiplication. As \(\tau\) is proper \(\cl{\Rel}\cap (\Cont_0(\Rp)\otimes\Comp(L^2M))\) is the unique dense, relatively continuous and complete subspace by \cite{meyer2001}*{9.4}. As \(\tau\) is also a free action, it is saturated (see the preprint version of \cite{rieffel1998} and in \cite{anhuef2002}*{4.1}). The Morita--Rieffel equivalence follows again from the generalized fixed point algebra construction. 
\end{proof}
\section{\(\K\)-theory and index theory}\label{sec:k-theory_index}
In this section, we examine when an operator on a filtered manifold is elliptic in an appropriate sense. The short exact sequence~\eqref{ses:fixedalg} yields that \(P\in\Psi^0_H(M)\) is Fredholm if and only if its principal cosymbol \(S_H(P)\) is invertible in \(\Fix^\Rp(J_0,\cl{\Rel_0})\).
\begin{definition}
	A pseudodifferential operator \(P\in\Psi^0_H(M)\) is \emph{\(\Cst\)-\(H\)-elliptic} if its principal cosymbol \(S_H(P)\in\Fix^\Rp(J_0,\cl{\Rel_0})\) is invertible. 
\end{definition}
\subsection{Rockland condition and \(H\)-ellipticity}
The goal of this section is to understand \(\Cst\)-\(H\)-ellipticity better. Moreover, 
\(\Cst\)-\(H\)-ellipticity is compared to \(H\)-ellipticity, which was defined by van Erp and Yuncken. We also discuss the relation to the Rockland condition. 
\begin{definition}[\cite{erp2017tangent}*{54}]
	An operator \(P\in\Psi^m_{H}(M)\) on a compact filtered manifold \((M,H)\) is \emph{\(H\)-elliptic} if its principal cosymbol \(\princ[m](P)\) is invertible in \(\Smooth'_r(T_HM)/\Cont^\infty_c(T_HM)\).
\end{definition}
If \(P\in\Psi^m_H(M)\) is \(H\)-elliptic, it admits a two-sided parametrix \(Q\in\Psi^{-m}_H(M)\), that is, \(PQ-1,QP-1\in\Cont^\infty_{c}(M\times M)\), see \cite{erp2017tangent}*{60}. If \(P\) is an \(H\)-elliptic differential operator, this implies that \(P\) is hypoelliptic \cite{erp2017tangent}*{61}. 
For \(m=0\), \(P\) is \(H\)-elliptic if and only if \(\princ[0](P)\) is invertible in~\(\Sigma^0_{H}(M)\). This follows from \cite{erp2017tangent}*{55}.

The \(\Cst\)\nb-algebra of principal cosymbols \(\Fix^\Rp(\J_0,\cl{\Rel_0})\) is a continuous field of \(\Cst\)-algebras over~\(M\) with fibres \(\Fix^\Rp(J_x,\cl{\Rel_x})\) by \cref{res:fix_0_isfield}. 
\begin{lemma}
	Let \(A\) be a unital \(\Cst\)-algebra. Suppose \(A\) is a continuous field of \(\Cst\)-algebras over a compact Hausdorff space \(X\) with fibre projections \(q_x\colon A\to A_x\). Then \(a\in A\) is invertible if and only if \(a_x\defeq q_x(a)\) is invertible for all \(x\in X\).  
\end{lemma}
\begin{proof}
	Clearly, invertibility of \(a\) implies that \(q_x(a)\) is invertible for all \(x\in M\). 		
	Conversely, suppose that \(q_x(a)\) is invertible at every point \(x\in M\). Fix \(x\in X\) and let \(b_x\in A_x\) be an inverse of \(a_x\). As \(q_x\) is surjective, there is a \(b\in A\) with \(q_x(b)=b_x\). Let \(c\defeq 1-ab\). Because \(q_x(c)=0\), one can find an open neighbourhood \(U_x\) of \(x\) such that \(\norm{q_y(c)}\leq 1/2\) for all \(y\in U_x\) by continuity. By the von Neumann series \(ab\) is locally invertible on \(U_x\). Therefore, \(b(ab)^{-1}\) is right inverse to \(a\) on \(U_x\). Using a continuous partition of unity which is subordinate to the open cover \(U_x\) of \(X\), one can glue together the local inverses to a global right inverse of \(a\). Similarly, one can construct a left inverse of \(a\). It follows that \(a\) is invertible. 
\end{proof}
Therefore, \(\Cst\)-\(H\)-ellipticity is a pointwise condition.
\begin{corollary}\label{res:ellipticispointwise}
	Let \((M,H)\) be a compact filtered manifold. A principal cosymbol \(u\in\Fix^\Rp(\J_0,\cl{\Rel_0})\) is \(\Cst\)-\(H\)-elliptic if and only if \(u_x\in\Fix^\Rp(J_x,\cl{\Rel_x})\) is invertible for all \(x\in M\).
\end{corollary}
Consider now the Rockland condition as described in \cite{christ1992pseudo}, see also \cite{pongeheisenberg}*{Sec.~3.3.2} or \cite{davehaller}*{Sec.~3.4}. It generalizes the Rockland condition for differential operators defined in~\cite{rockland}. 

Let~\(G\) be a graded Lie group of homogeneous dimension \(Q\) and \(m\in\Z\). Let \(u\in\Smooth'(G)/\Cont^\infty_c(G)\) be \(m\)-homogeneous. It was shown in \cite{davehaller}*{3.8} that the class \(u\in\Smooth'(G)/\Cont^\infty_c(G)\) can be uniquely represented by a kernel \(a\) of type~\(-m\) if \(-m-Q\notin\N_0\). For \(m=0\) this is \cref{res:sigma=kernels}. If \(-m-Q\in\N_0\), it can be represented by \(a=k+p\log(\norm{x})\) with \(k\in\kernel^{-m}(G)\) and a \((-m-Q)\)-homogeneous polynomial~\(p\). This representation is not necessarily unique. However, the map \(\Schwartz_{0}(G)\to\Schwartz_{0}(G)\) given by \(f\mapsto a*f\) does not depend on the chosen representative \(a\) of \(u\) (see~\cite{davehaller}*{Sec.~3.4}). Here, \(\Schwartz_0(G)\) consists of \(f\in\Schwartz(G)\) with \(\int_Gx^\alpha f(x)\diff x\) for all \(\alpha\in\N^n_{0}\).

For a unitary, irreducible representation \(\pi\colon G\to\Uni(\Hils_\pi)\), let \(\Hils^0_\pi\) be spanned by \(\widehat{\pi}(f)v\) for \(f\in\Schwartz_0(G)\) and \(v\in\Hils_\pi\). The operator \(\pi(u)\) defined on \(\Hils^0_\pi\) by
\[\pi(u)(\widehat{\pi}(f)v)\defeq \widehat{\pi}(a*f)v \quad\text{for }f\in\Schwartz_0(G)\text{, }v\in\Hils_\pi,\]
is closable. Denote its closure by \(\cl{\pi(u)}\). 
\begin{definition}
	Let \(G\) be a graded Lie group and let \(u\in\Smooth'(G)/\Cont^\infty_c(G)\) be homogeneous. Then~\(u\) satisfies the \emph{Rockland condition} if \(\cl{\pi(u)}\) is injective on \(\Hils^\infty_\pi\) for all \(\pi\in \widehat{G}\!\setminus\! \{\pi_{\triv}\}\). Here, \(\Hils^\infty_\pi\) denotes the space of smooth vectors. 
\end{definition}

\begin{example}
	Recall that for a usual compact manifold of dimension \(n\) the osculating groups are isomorphic to \(\R^n\). Let \(P\) be a pseudodifferential operator on~\(M\) with model operators \(\Sigma(P)_x\) for \(x\in M\). Then \(\Sigma(P)_x\) satisfies the Rockland condition if and only if \(\widehat{\Sigma(P)_x}(\xi)\neq 0\) for all \(\xi\neq 0\). This is the usual condition on the principal symbol of \(P\). Therefore, \(P\) is elliptic if and only if \(\Sigma(P)_x\) satisfies the Rockland condition for all \(x\in M\). 
\end{example}

For a graded Lie group \(G\), \(\Fix^\Rp(J_G,\cl{\Rel_G})\) is the \(\Cst\)-algebra of kernels of type \(0\) by \cite{ewert2020pseudodifferential}*{6.11}. The spectrum of \(\Fix^\Rp(J_G,\cl{\Rel_G})\) can be identified with \((\widehat{G}\!\setminus\!\{\pi_{\triv}\})/\Rp\) by \cite{ewert2020pseudodifferential}*{8.4}. This allows to describe invertibility in \(\Fix^\Rp(J_G,\cl{\Rel_G})\) in terms of the representations of \(G\).
\begin{definition}
	An element \(u\in\Fix^\Rp(J_G,\cl{\Rel_G})\) satisfies the \emph{\(\Cst\)-Rockland condition} if \(\pi(u)\) is invertible for all \(\pi\in\widehat{G}\!\setminus\!\{\pi_{\triv}\}\).
\end{definition}
It is well-known that an element \(a\) of a unital \(\Cst\)-algebra \(A\) is invertible if and only if \(\pi(a)\) is invertible for all irreducible representations \(\pi\in\widehat{A}\), see \cite{exel}*{2.2} for a direct proof. Hence, the following holds.
\begin{proposition}\label{res:cstelliptic-rockland}
	Let \(u\in\Fix^\Rp(J_G,\cl{\Rel_G})\). Then \(u\) is invertible if and only if  \(u\) satisfies the \(\Cst\)-Rockland condition.
\end{proposition}
Now, we deduce that \(P\in\Psi^0_H(M)\) is \(\Cst\)-\(H\)-elliptic if and only if it is \(H\)-elliptic.
\begin{proposition}
	Let \((M,H)\) be a compact filtered manifold. For \(P\in\Psi^0_H(M)\) the following are equivalent:
	\begin{enumerate}
		\item\label{item:cstelliptic}\(P\) is \(\Cst\)-\(H\)-elliptic,
		\item\label{item:cstrockland} \(\princ[0](P)_x\) satisfies the \(\Cst\)-Rockland condition for all \(x\in M\),
		\item\label{item:rockland} \(\princ[0](P)_x\) and \(\princ[0](P)_x^*\) satisfy the Rockland condition for all \(x\in M\),
		\item\label{item:elliptic} \(P\) is \(H\)-elliptic.
	\end{enumerate}
\end{proposition}
\begin{proof}
	The equivalence of \ref{item:cstelliptic} and \ref{item:cstrockland} follows from \cref{res:ellipticispointwise} and \cref{res:cstelliptic-rockland}. By the results of G\l owacki in \cite{glowacki}*{4.3 and 4.9}, \ref{item:cstrockland} and \ref{item:rockland} are equivalent for all \(x\in M\). The arguments in \cite{davehaller}*{3.11, 3.12}, which follow \cite{christ1992pseudo} and \cite{pongeheisenberg}, show that \ref{item:rockland} and \ref{item:elliptic} are equivalent. 
\end{proof}
\begin{remark}
	The Rockland condition is also defined for operators acting between vector bundles \(E,F\) over \(M\) (see \cite{pongeheisenberg}*{Sec.~3.3.2} or \cite{davehaller}*{Sec.~3.4}). In this case, the model operators map \(\Schwartz_0(G_x,E_x)\to \Schwartz_0(G_x,F_x)\). It is shown in \cite{davehaller}*{3.11,~3.12} that \(H\)-ellipticity is again equivalent to satisfying the Rockland condition at all points. 
\end{remark}
\subsection{Deformation to Abelian case}\label{sec:adiabatic}
For a filtered manifold \((M,H)\), consider the restriction of the short exact sequence in \eqref{ses:tangentgroupoid} to \([0,1]\) 
\begin{equation*}
\begin{tikzcd}
\Cont_0((0,1])\otimes\Comp(L^2M) \arrow[r,hook] & \Cst(\T_HM|_{[0,1]})\arrow[r,twoheadrightarrow,"\ev^H_0"] & \Cst(T_HM).
\end{tikzcd}
\end{equation*}	
The \(\Cst\)-algebra on the left is contractible and \(\Cst(T_HM)\) is nuclear as a bundle of nilpotent Lie groups. Therefore, the class \([\ev_0^H]\in\KK(\Cst(\T_HM|_{[0,1]}),\Cst(T_HM))\) is invertible. As described in \cite{debordlescure} one can define a deformation element \[[\ev^H_0]^{-1}\otimes[\ev^H_1]\in\KK(\Cst(T_HM), \Comp)\]
associated to the short exact sequence above. 	 
Likewise, there is a deformation element \([\ev_0]^{-1}\otimes[\ev_1]\in\KK(\Cont_0(T^*M), \Comp)\) for the short exact sequence of Connes' tangent groupoid
\begin{equation*}
\begin{tikzcd}
\Cont_0((0,1])\otimes\Comp(L^2M) \arrow[r,hook] & \Cst(\T M|_{[0,1]})\arrow[r,twoheadrightarrow,"\ev_0"] & \Cont_0(T^*M).
\end{tikzcd}
\end{equation*}	
Connes showed that this class is the analytical index (see \cite{connes}).

Using the adiabatic groupoid of \(\T_HM\) one can relate the deformation classes in \(\KK(\Cst(T_HM),\Comp)\) and \(\KK(\Cont_0(T^*M),\Comp)\). 
This construction was carried out in \cite{vanerpcontactI} for contact manifolds, see also \cites{mohsen2018deformation,mohsen2020index} for the filtered manifold case. In the following, we recall the argument.

The Lie algebroid \(\t_HM\) of \(\T_HM\) is described in \cite{erp2017tangent}. Denote by \(\rho\colon \t_HM\to T(M\times[0,1])\) its anchor map and let
\[[\,\cdot\,,\,\cdot\,]\colon \Gamma^\infty(\t_HM)\times\Gamma^\infty(\t_HM)\to\Gamma^\infty(\t_HM)\]
be the bracket. Let \(\T_HM^{a}\) be the adiabatic groupoid of \(\T_HM\). It is easier to describe it in terms of its Lie algebroid (see \cite{debordskandalis2014}*{2.1}). It is the vector bundle \(\t_HM\times\R\) over \(M\times\R\times\R\) with anchor 
\begin{align*}\rho_{a}\colon \t_HM\times\R&\to T(M\times\R)\times T\R\\
\rho_{a}(x,t,U,s)&= (\rho(x,t,sU),s,0) 
\end{align*}
for \(x\in M\), \(t,s\in\R\) and \(U\in\t_HM_{(x,t)}\). The bracket is defined by
\[[X,Y]_{a}(x,t,s)=s[X,Y](x,t) \quad \text{for }X,Y\in\Gamma^\infty(\t_HM\times\R).\]
The resulting Lie groupoid can be viewed as a continuous field of groupoids over each copy of \(\R\) and over \(\R^2\). 	
The fibre over \(s=1\) is the tangent groupoid \(T_HM\) of the filtered manifold. 	

For \(s=0\), the anchor and bracket are zero. One obtains a bundle of Abelian groups. It is isomorphic to \(TM\times[0,1]\) via a splitting as defined in \cite{erp2017tangent}*{9}. A splitting is a vector bundle isomorphism \(\lie{t}_HM\to TM\), which restricts on \(H^j/H^{j-1}\) to a right inverse of \(H^j\to H^j/H^{j-1}\) for \(j=1,\ldots,r\). 

For \(t=0\) the anchor is zero. Therefore, all fibres over \((0,s)\) for \(s\in\R\) are bundles of nilpotent groups. The bundle of osculating groups \(T_HM\) at \(s=1\) is deformed into a bundle of Abelian groups at \(s=0\). The latter can be identified with \(TM\) using the splitting above. Denote the subgroupoid at \(t=0\) by \(\grpd\).

Note that the fibre at \((1,1)\) is the pair groupoid of \(M\). Its adiabatic groupoid is Connes' tangent groupoid \(\T M\). It is the fibre at \(t=1\) of \(\T_HM^{a}\).

Therefore, all edges of \([0,1]^2\) can be understood as deformation groupoids and one can associate corresponding deformation classes in the respective \(\KK\)-groups.

Denote the restriction to the edges by \(r_{t=1},r_{t=0},r_{s=1}, r_{s=0}\). The following diagram commutes
\begin{equation*}
\begin{tikzcd}
\T_HM^{a} \arrow[d,"r_{s=1}"]\arrow[r,"r_{t=1}"] & \T M \arrow[d,"\ev_1"] \\
\T_HM\arrow[r,"\ev_1^H"] & M\times M.
\end{tikzcd}
\end{equation*}
Therefore, the following \(\KK\)-classes coincide:
\begin{align}\label{eq:equofkk}[r_{t=1}]\otimes[\ev_1]=[r_{s=1}]\otimes[\ev^H_1].\end{align} Denote by \(b_0\) and \(b_1\) the restrictions to \(t=0\) and \(t=1\) on the trivial bundle at \(s=0\). The deformation class \([b_0]^{-1}\otimes[b_1]\) is the identity. It follows that
\begin{align}\label{eq:1}
[r_{t=1}]\otimes[\ev_0]=[r_{t=0}]\otimes[b_1]=[r_{t=0}]\otimes[b_0].
\end{align}
Denote by \(c_0\) and \(c_1\) the respective restrictions on \(\grpd\). We obtain
\begin{align}\label{eq:2}[r_{s=1}]\otimes[\ev^H_0]=[r_{t=0}]\otimes[c_1]=([r_{t=0}]\otimes[c_0])\otimes([c_0]^{-1}\otimes[c_1]).
\end{align}
Let \(r_{00}\) be the restriction to \(t=0,s=0\). One can show that it induces a \(\KK\)\nb-equivalence as in \cite{vanerpcontactI}*{21}. This is done by writing it as a composition of the restriction to the union of \(t=0\) and \(s=0\) and further restriction to \(t=s=0\). Both maps have contractible kernels. Using \eqref{eq:1} and \eqref{eq:2}, one obtains
\begin{align*}[r_{00}]^{-1}\otimes[r_{t=1}]&=[\ev_0]^{-1},\\
[r_{00}]^{-1}\otimes[r_{s=1}]&=([c_0]^{-1}\otimes[c_1])\otimes[\ev^H_0]^{-1}.
\end{align*}	 
Inserting this into \eqref{eq:equofkk} shows that the deformation classes for \(\T_HM\) and \(\T M\) are related by
\begin{align*}
[\ev_0]^{-1}\otimes[\ev_1]=([c_0]^{-1}\otimes[c_1])\otimes([\ev^H_0]^{-1}\otimes[\ev_1^H]).
\end{align*}
The class  \(\Psi\defeq \left([c_0]^{-1} \otimes [c_1]\right)\in \KK(\Cont_0(T^*M),\Cst(T_HM))\) is a \(\KK\)-equivalence. This is a well-known consequence of the Connes--Thom isomorphism, see \cite{connesthom}*{Corollary~7} and \cite{nistor}*{Corollary~1} for the bundle version. 
We show that the \(\KK\)-equivalence \(\Psi\) restricts to the ideals used in the generalized fixed point algebra construction.
\begin{lemma}\label{res:restricts}
	The \(\KK\)-equivalence \(\Psi\in\KK(\Cont_0(T^*M), \Cst(T_HM))\) restricts to a \(\KK\)-equivalence \(\Psi|\in\KK(\Cont_0(T^*M\!\setminus\!(M\times 0)), \J_0)\).
\end{lemma}
\begin{proof}	
	Define the ideal \(\J_\grpd\subset\Cst(\grpd)\) that consists of all sections \((a_s)\in\Cst(\grpd)\) such that all \(a_{s,x}\) for \(s\in[0,1]\) and \(x\in M\) lie in the kernel of the trivial representation of the nilpotent Lie group over \((s,x)\). The trivial representations induce a commuting diagram
	\begin{equation}\label{diag:comm_restr}
	\begin{tikzcd}
	\J_0 \arrow[r,hook,"i_1"] & \Cst(T_HM) \arrow[r,twoheadrightarrow,"q_1"] & \Cont_0(M)\\
	\J_\grpd\arrow[r,hook,"i"]\arrow{u}[swap]{e_1}\arrow{d}{e_0} & \Cst(\grpd)\arrow[r,twoheadrightarrow,"q"]\arrow{u}[swap]{c_1}\arrow{d}{c_0} & \Cont([0,1],\Cont_0(M)) \arrow{u}[swap]{f_1}\arrow{d}{f_0}\\
	\Cont_0(T^*M\!\setminus\!(M\times 0))\arrow[r,hook,"i_0"] & \Cont_0(T^* M) \arrow[r,twoheadrightarrow,"q_0"]& \Cont_0(M).
	\end{tikzcd}
	\end{equation}
	As \(\ker(e_0)\) is contractible, one can build the deformation class \(\Psi|\defeq([e_0]^{-1}\circ[e_1])\in\KK(\Cont_0(T^*M\!\setminus\!(M\times 0)), \J_0)\). Similarly, there is a class \(\alpha\defeq([f_0]^{-1}\circ[f_1])\in\KK(\Cont_0(M),\Cont_0(M))\).
	Because \eqref{diag:comm_restr} commutes, there is a commuting diagram in~\(\KK\):
	\begin{equation*}
	\begin{tikzcd}
	\J_0 \arrow[r,hook,"i_1"] & \Cst(T_HM) \arrow[r,twoheadrightarrow,"q_1"] &\Cont_0(M) \\
	\Cont_0(T^*M\!\setminus\!(M\times 0))\arrow[u,"\Psi|"]\arrow[r,hook,"i_0"] & \Cont_0(T^* M)\arrow[u,"\Psi"] \arrow[r,twoheadrightarrow,"q_0"]& \Cont_0(M)\arrow[u,"\alpha"].
	\end{tikzcd}
	\end{equation*}
	The \(\KK\)-classes in the middle and on the right are \(\KK\)-equivalences. The long exact sequences in \(\KK\)-theory and the Five Lemma yield that 
	\begin{align*}
	\,\textvisiblespace\,\otimes \Psi|\colon \KK(A,\Cont_0(T^*M\!\setminus\!(M\times 0)))&\to\KK(A,\J_0)\\
	\Psi|\otimes	\,\textvisiblespace\, \colon\KK(\J_0,B)&\to\KK(\Cont_0(T^*M\!\setminus\!(M\times 0)),B)
	\end{align*}
	are isomorphisms for all separable, nuclear \(\Cst\)-algebras \(A,B\). Taking \(A=\J_0\) and \(B=\Cont_0(T^*M\!\setminus\!(M\times 0))\), one obtains a class in \(\KK(\J_0,\Cont_0(T^*M\!\setminus\!(M\times 0))\) that is the \(\KK\)-inverse of \(\Psi|\).
\end{proof}
As a consequence, the \(\Cst\)-algebra of principal cosymbols of order~\(0\) has the same \(\K\)-theory as its unfiltered counterpart. 
\begin{theorem}
	Let \((M,H)\) be a filtered manifold. Then  \(\Cont_0(S^*M)\) and the \(\Cst\)-algebra of principal cosymbols \(\Fix^\Rp(\J_0,\cl{\Rel_0})\) are \(\KK\)-equivalent. 
\end{theorem}
\begin{proof}
	The \(\Cst\)-algebra \(\Fix^\Rp(\J_0,\cl{\Rel_0})\) is Morita--Rieffel equivalent to \(\Cred(\Rp,\J_0)\) by \cref{res:morita}. Therefore, they are \(\KK\)-equivalent. As \((\Rp,\,\cdot\,)\cong (\R,+)\) and by the Connes--Thom isomorphism, \(\Cred(\Rp,\J_0)\) is \(\KK\)\nb-equivalent to  \(\Cont_0(\R)\otimes \J_0\). This \(\Cst\)-algebra is \(\KK\)\nb-equivalent to \(\Cont_0(\R)\otimes \Cont_0(T^*M\setminus(M\times 0))\) by \cref{res:restricts}. The converse argument, applied to the step \(1\) filtration case, yields that \(\Cont_0(\R)\otimes \Cont_0(T^*M\setminus(M\times 0))\) is \(\KK\)\nb-equivalent to \(\Cont_0(S^*M)\). 
\end{proof}
\subsection{Towards index theory}
In this section, let \((M,H)\) be a compact filtered manifold. 
Following the explanation in \cite{Baumvanerp}*{5.3} (see also \cite{connes}*{§II.9.\(\alpha\)}), one can attach to an \(H\)-elliptic \(H\)-pseudodifferential operator \(P\) of order \(m\) a class in \(\K_0(\Cst(T_HM))\). Let \(\P\) be a lift of \(P\) to \(\Pseu^m_H(M)\). By definition, the equivalence class \([\P_0]\in\Smooth'_r(T_HM)/\Cont^\infty_c(T_HM)\) is invertible. So there is a \(\Q_0\in\Smooth'_r(T_HM)\) with \(S_0\defeq 1-\Q_0*\P_0\in\Cont^\infty_c(T_HM)\) and \(S_1\defeq 1-\P_0*\Q_0\in\Cont^\infty_c(T_HM)\).  

Let \(\sigma_H(P)\defeq [e]-[e_0]\in\K_0(\Cst(T_HM))\) be the formal difference of idempotents 
\begin{align*}
e  =\begin{pmatrix*}
1-S_1^2 & \P_0*S_0\\
S_0*\Q_0*(1+S_1) & S_0^2
\end{pmatrix*} \quad\text{and}\quad
e_0 = \begin{pmatrix*}
1 & 0\\
0 & 0
\end{pmatrix*}.
\end{align*}
The same construction works for operators that act on a vector bundles over \(M\).
\begin{lemma} Let \((M,H)\) be a compact filtered manifold. 
	Consider the short exact sequence from \eqref{eq:diskbundle} given by
	\begin{equation}\label{ses:symbols}
	\begin{tikzcd}
	\Cst(T_HM)\arrow[r,hook] & \Cst(\Ess^0_H(M)) \arrow[r,twoheadrightarrow]& \Cst(\Sigma^0_H(M)).\end{tikzcd}
	\end{equation}
	For an \(H\)-elliptic \(P\in\Psi^0_H(M)\) the class \(\sigma_H(P)\in\K_0(\Cst(T_HM))\) above is the image of \([\princ[0](P)]\in\K_1(\Cst(\Sigma^0_H(M)))\) under the boundary map in \(\K\)-theory of \eqref{ses:symbols}.
\end{lemma}
\begin{proof}
	Let \(\Q_0\in\Smooth'_r(T_HM)\) satisfy \(1-\Q_0*\P_0\in\Cont^\infty_c(T_HM)\) and \(1-\P_0*\Q_0\in\Cont^\infty_c(T_HM)\)  as above. By \cite{erp2015groupoid}*{55}  \(\Q_0\) is contained in \(\Ess^0_H(M)\). Hence, \(\P_0\) and \(Q_0\) are lifts of \([\P_0]\) and \([\P_0]^{-1}\) in \(\Ess^0_H(M)\). Computing the image of \([\princ[0](P)]\) under the index map as in \cite{Cuntz-Meyer-Rosenberg}*{1.46} gives exactly the class above. 
\end{proof}
Up to inverting the Connes--Thom isomorphism, we prove an index theorem for \(H\)-elliptic pseudodifferential operators of order zero. For contact manifolds, this is \cite{vanerpcontactII}*{Prop.~12} in the scalar-valued case or \cite{Baumvanerp}*{Thm.~5.4.1} for operators acting on vector bundles. 
Mohsen recently proved an index theorem for filtered manifolds in \cite{mohsen2020index}. His construction involves a ``larger'' bundle of graded Lie groups over \(M\) to obtain an index theorem that does not contain the Connes--Thom isomorphism anymore. 
\begin{theorem}
	Let \((M,H)\) be a compact filtered manifold and let \(P\) be an order zero \(H\)-elliptic \(H\)-pseudodifferential operator acting on vector bundles \(E,F\) over \(M\). Let \(\Psi\colon \K^0(T^*M)\to\K_0(\Cst(T_HM))\) denote the Connes--Thom isomorphism and \(\ind_t\colon \K^0(T^*M)\to \Z\) the topological index map. Then \(P\) is Fredholm and its Fredholm index is given by
	\[\ind(P)=\ind_t(\Psi^{-1}(\sigma_H(P))).\]
\end{theorem}
\begin{proof}
	Choose hermitean metrics on \(E\) and \(F\). Using polar decomposition and that the Fredholm index is invariant under homotopies, we can assume without loss of generality that the principal cosymbol of \(P\) is unitary.
	
	We follow the arguments in \cites{vanerpcontactI,vanerpcontactII,Baumvanerp} closely. Let \(\E\defeq E\times[0,1]\) and \(\F\defeq F\times[0,1]\) denote the vector bundles over the unit space \(M\times[0,1]\) of \(\T_HM|_{[0,1]}\) and extend \(P\) to \(\P\in\Pseu^0_H(M,\E,\F)\).
	
	As in \cite{Baumvanerp} construct a class \([\D]\in\KK(\Cont(M),\Cst(\T_HM|_{[0,1]}))\) from~\(\P\) as follows.
	Define the \(\Z_2\)-graded right Hilbert \(\Cst(\T_HM|_{[0,1]})\)-module \[\mathcal{E}=\Gamma_0(\E)\otimes_{\Cont(M\times[0,1])}\Cst(\T_HM|_{[0,1]})\oplus\Gamma_0(\F)\otimes_{\Cont(M\times[0,1])}\Cst(\T_HM|_{[0,1]})\] and let  
	\[\D= \begin{pmatrix*}
	0 & \P^*\\
	\P & 0
	\end{pmatrix*}\in\Cst(\Pseu^0_H(M))\otimes\End(\E\oplus \F).\]
	Note that elements of \(\Cst(\Pseu^0_{H}(M))\) act as multipliers on \(\Cst(\T_HM|_{[0,1]})\). By \cite{erp2015groupoid}*{25} there is a homomorphism \(\Cont(M)\to \Cst(\Pseu^0_H(M))\), \(f\mapsto \f\), where \(\f_t\) is the multiplication operator \(M_f\) on \(L^2(M)\) for \(t>0\) and \(\f_0\) is the fibred distribution given by \((f(x)\delta_x)_{x\in M}\). Therefore, there is a diagonal representation \(\phi\colon \Cont(M)\to\mathcal{L}(\mathcal{E})\). 	
	Moreover, \(\D\) acts as an odd operator on \(\mathcal{E}\).	
	We verify that \(\phi(f)(\D^2-1)\) and \([\phi(f),\D]\) lie in \(\mathcal{K}(\mathcal{E})\) for all \(f\in \Cont(M)\). 
	Compute \[\phi(f)(\D_t^2-1)=\phi(f)\begin{pmatrix*}
	\P_t^*\P_t-1&0\\
	0 & \P_t^*\P_t-1
	\end{pmatrix*}.\]
	At \(t=0\), this defines a matrix over \(\Cst(T_HM)\). Then the claim follows from \cref{res:kernelcompact}. For \([\phi(f),\D]\), note that this vanishes at \(t=0\) as functions in \(\Cont(M)\) define central multipliers of \(\Cst(T_HM)\). Therefore, \cref{res:kernelcompact} applies, too.	
	Note that one can restrict \((\mathcal{E},\phi,\D)\) to \(t\geq 0\) and denote the restricted classes by \([\D_t]\). 
	The long exact sequence in \(\KK\)-theory implies that
	\[\ev_0^H\colon \KK(\Cont(M),\Cst(\T_HM|_{[0,1]}))\to \KK(\Cont(M),\Cst(T_HM))\]
	is invertible. Similar to before, one obtains a map
	\[\ev_1^H\circ(\ev_0^H)^{-1}\colon \KK(\Cont(M),\Cst(T_HM))\to \KK(\Cont(M),\C).\]
	It satisfies \(\ev_1^H\circ(\ev_0^H)^{-1}([\D_0])=[\D_1]\). 
	Let \([u]\in\KK(\C,\Cont(M))\) be the class induced by the unital embedding \(\C\to \Cont(M)\). It is well-known that \([u]\otimes [\D_1]\in\KK(\C,\C)\cong\Z\) is the class representing the Fredholm index of \(P\) (see \cite{Cuntz-Meyer-Rosenberg}*{(12.7)}).
	
	Similarly, there is a corresponding map for Connes' tangent groupoid
	\[\ev_1\circ(\ev_0)^{-1}\colon \KK(\Cont(M),\Cont_0(T^*M))\to \KK(\Cont(M),\C).\]
	The arguments in \cref{sec:adiabatic} involving the adiabatic groupoid can be adapted to show that there is a Connes--Thom isomorphism \(\Psi\) such that the following diagram commutes
	\begin{equation}\label{eq:commdiagconnes}
	\begin{tikzcd}
	\KK(\Cont(M),\Cont_0(T^*M)) \arrow[r, "\Psi"] \arrow[d, "\ev_1\circ(\ev_0)^{-1}"]& \KK(\Cont(M),\Cst(T_HM))\arrow[dl, "\ev_1^H\circ(\ev_0^H)^{-1}"]\\
	\KK(\Cont(M),\C).
	\end{tikzcd}
	\end{equation}	
	As in \cite{Baumvanerp}, one can use the natural transformation \(\alpha_M\) that maps  \(\KK(\C,A)\) to \(\KK(\Cont(M),A)\) for \(\Cont(M)\)-algebras \(A\). It makes
	\begin{equation*}
	\begin{tikzcd}
	\KK(\C,\Cont_0(T^*M)) \arrow[r, "\Psi"] \arrow[d, "\alpha_M"]& \KK(\C,\Cst(T_HM))\arrow[d, "\alpha_M"]\\
	\KK(\Cont(M),\Cont_0(T^*M)) \arrow[r, "\Psi"] & \KK(\Cont(M),\Cst(T_HM))
	\end{tikzcd}
	\end{equation*}	
	commute. Since the principal cosymbol of \(P\) was assumed to be unitary, we can take \(\Q_0=\P_0^*\) to construct \(\sigma_H(P)\in \K_0(\Cst(T_HM))\). This class can be represented in \(\KK(\C,\Cst(T_HM))\) using the Fredholm module given by \begin{align*}\mathcal{E}_0&=\Gamma_0(E)\otimes_{\Cont(M)}\Cst(T_HM)\oplus\Gamma_0(F)\otimes_{\Cont(M)}\Cst(T_HM),\\
	\D_0&= \begin{pmatrix*}
	0 & \P_0^*\\
	\P_0 & 0
	\end{pmatrix*}.
	\end{align*}
	Its class is mapped to \([\D_0]\) by \(\alpha_M\). Together with the commutativity of \eqref{eq:commdiagconnes}, this shows that
	\([\ind P]=\ev_1\circ(\ev_0)^{-1}\circ\alpha_M\circ\Psi^{-1}(\sigma_H(P))\). Therefore, 
	the claim is reduced to the Atiyah--Singer Index Theorem. 
\end{proof}

\end{document}